\documentclass[11pt]{article}
\usepackage{latexsym,subfigure}
\usepackage{comment}
\usepackage{amsmath,amsthm,amsfonts,amssymb,graphicx,epsfig,latexsym,color}
\usepackage{epsf,enumerate,url}

\title{Constructing group actions on quasi-trees and applications to
  mapping class groups}
\author{Mladen Bestvina, Ken Bromberg and Koji Fujiwara\thanks{The
    first two authors gratefully acknowledge the support by the National
    Science Foundation. 
The third author is supported in part by
Grant-in-Aid for Scientific Research (No. 19340013)}} 
\date{September 1, 2014}

\newtheorem{thm}{Theorem}[section]
\newtheorem{lemma}[thm]{Lemma}
\newtheorem{cor}[thm]{Corollary}
\newtheorem{prop}[thm]{Proposition}
{}
{}
\newtheorem{question}[thm]{Question}

\theoremstyle{remark}
\newtheorem{example}[thm]{Example}

\newtheorem{definition}[thm]{Definition}
\newtheorem{remark}[thm]{Remark}
\newtheorem{examples}[thm]{Examples}
\newtheorem*{definition*}{Definition}
\newtheorem*{remark*}{Remark}

\newcommand{\bY}{{\bf Y}}

\newcommand{\lc}{\xi}
\newcommand{\pd}{d^\pi}

\renewcommand{\xi}{{\theta}}
\newcommand{\xio}{{\Theta}}
\def\Gamma{G}

\def\cH{{\mathcal H}}
\def\cS{{\mathcal P}}
\def\cC{{\mathcal C}}
\def\cX{{\mathcal X}}
\def\cY{{\mathcal Y}}

\def\cP{{\mathcal P}}
\def\cG{{\mathcal G}}

\def\R{{\mathbb R}}

\def\Z{{\mathbb Z}}

\def\H{{\mathbb H}}

\def\diam{\operatorname{diam}}
\def\asdim{\operatorname{asdim}}

\def\C{{\cal C}}

\renewcommand{\>}{\rangle}

\begin{document}

\maketitle

\begin{abstract}
A quasi-tree is a geodesic metric space quasi-isometric to a tree.  We
give a general construction of many actions of groups on
quasi-trees. The groups we can handle include non-elementary
(relatively) hyperbolic groups,  $CAT(0)$ groups with rank 1 elements, mapping class
groups and $Out(F_n)$. As an application, we show that mapping class
groups act on finite products of $\delta$-hyperbolic spaces so that
orbit maps are quasi-isometric embeddings. We prove that mapping class
groups have finite asymptotic dimension.
\end{abstract}

\tableofcontents

\section{Introduction}
In this paper we define a new combinatorial complex which we call the
{\em projection complex}, and a closely related complex called the
{\it quasi-tree of metric spaces}. To motivate the construction
consider a discrete group $\Gamma$ of isometries of hyperbolic
$n$-space $\H^n$ and let $\gamma\in\Gamma$ be an element with an axis
$\ell\subset\H^n$. Denote by $\bY$ the set of all $\Gamma$-translates
of $\ell$, i.e. the set of axes of conjugates of $\gamma$. When
$A,B\in\bY$, $A\neq B$, denote by $\pi_A(B)\subset A$ the image of $B$
under the nearest point projection $\pi_A:\H^n\to A$. We call this set
the {\it projection} of $B$ to $A$ and we observe:
\begin{enumerate}[(P0)]
\item The diameter $\diam \pi_A(B)$ is uniformly bounded by $\xi\geq 0$,
  independently of $A,B\in\bY$.
\end{enumerate}
This is because a line in $\H^n$ will have a big projection to another
line only if the two lines have long segments with small Hausdorff
distance between them, since $\Gamma$ is discrete (an easy exercise).

When $B\neq A\neq C$ we define a pseudo-distance function (and abusing
the terminology, we frequently drop ``pseudo'')
$$d_A^\pi(B,C)=\diam (\pi_A(B)\cup\pi_A(C))$$
which is symmetric and satisfies the triangle inequality, but in
general we have $d_A^\pi(B,B)>0$. We
observe further, again since $\Gamma$ is discrete,
for a perhaps larger constant $\xi$:
\begin{enumerate}[(P1)]
\item For any triple
  $A,B,C\in\bY$ of distinct elements at most one of the three numbers
$$d_A^\pi(B,C),d_B^\pi(A,C),d_C^\pi(A,B)$$
is greater than $\xi$.
\item For any
  $A,B\in\bY$ the set
$$\{C\in\bY\mid d_C^\pi(A,B)>\xi\}$$
is finite.
\end{enumerate}

\begin{figure}
\centerline{\scalebox{0.6}{\input{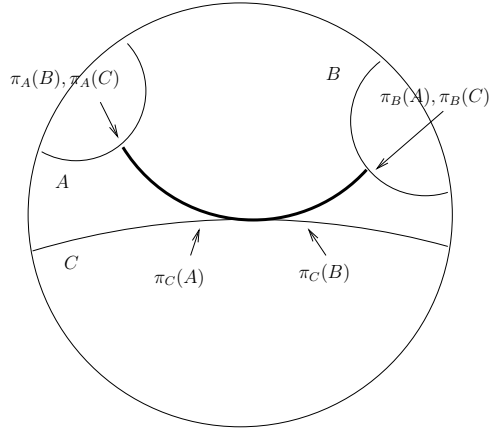}}}
\caption{Axiom (P1). The bold line is the shortest segment between $A$
  and $B$.  Note that $C$ and this segment stay close for a long time,
  therefore $d^\pi_C(A,B)$ is large, while $d^\pi_A(B,C)$ and
  $d^\pi_B(A,C)$ are small.}
\end{figure}

For an even more basic example where (P0)-(P2) hold with $\xi=0$
consider the Cayley tree of the free group $F_2=\<a,b\>$ and for $\bY$
take the $F_2$-orbit of the axis of $a$. We will discuss this example
in more detail in Section \ref{tree example}.

The main construction in this paper reverses this procedure. We start
with a collection of metric spaces $\bY$ and a collection of subsets
$\pi_A(B)\subset A$ for $A\neq B$ satisfying (P0)-(P2) and we
``reconstruct'' the ambient space.
Note that  in general the image of the nearest point projection 
 to $A$ of each point in 
$B$ may contain more than one point (such functions are 
called {\it coarse maps}).

\medskip
\noindent
{\bf Theorem A.}  {\it 
Suppose $\bY$ is a collection of geodesic metric spaces and for every
$A,B\in\bY$ with $A\neq B$ we are given a subset $\pi_A(B)\subset A$
such that (P0)-(P2) hold. Then
there is a geodesic metric space $\cC(\bY)$ that
  contains isometrically embedded, totally geodesic, pairwise disjoint
  copies of each $A\in\bY$ such that for all $A\neq B$ the nearest
  point projection of $B$ to $A$ in $\cC(\bY)$ is a uniformly bounded set uniformly
  close to $\pi_A(B)$.}
\medskip

The space $\cC(\bY)$ will be called a {\it quasi-tree of metric
  spaces}, for reasons explained below. Its construction will depend
on the choice of a sufficiently large parameter $K$, and it would be
more precise to denote the space by $\cC_K(\bY)$. If $K<K'$ there is a
natural Lipschitz map
$$\cC_K(\bY)\to \cC_{K'}(\bY)$$
which is in general not a quasi-isometry, and in fact unbounded sets
may map to bounded sets (see Section \ref{tree example} for an example).

In addition, many properties that hold uniformly for the spaces in $\bY$
carry over to $\cC(\bY)$. To state these results we first
recall some definitions.

A {\it quasi-tree} is a geodesic metric space quasi-isometric to a
tree. There is a characterization of quasi-trees due to Manning
\cite{manning}. A geodesic metric space $X$ satisfies the {\it bottleneck
  criterion} if there exists $\Delta\geq 0$ such that for any two
points $x,y\in X$ the midpoint $z$ of a geodesic between $x$ and $y$ satisfies the property such
that any path from $x$ to $y$ intersects the $\Delta$-ball centered at
$z$. Manning showed that this is equivalent to $X$ being a
quasi-tree. The constant $\Delta$ is called the {\it bottleneck
  constant}. 

The notion of {\it asymptotic dimension} was introduced by Gromov
\cite{gromov} as a large-scale analog of the covering dimension. A
metric space $X$ has asymptotic dimension $\asdim(X)\leq n$ if for
every $R>0$ there is a covering of $X$ by uniformly bounded sets such
that every metric $R$-ball intersects at most $n+1$ of the sets in the
cover. More generally, a collection of metric spaces has $\asdim$ at
most $n$ {\it uniformly} if for every $R$ there are covers of each space as above
whose elements are uniformly bounded over the whole collection.

\medskip
\noindent
{\bf Theorem B.} Let $\cC(\bY)$ be the quasi-tree of metric spaces
$\bY$ constructed in Theorem A.
{\it
\begin{enumerate}[(i)]
\item The construction is equivariant with respect to any group, $G$, that
  acts isometrically on the disjoint union of the spaces in $\bY$
  preserving projections, i.e.,
$d^\pi_{g(A)}(g(B),g(C))=d^\pi_A(B,C)$ for any $A,B,C \in  \bY$
and $g \in G$.
\item If each $X\in\bY$ is isometric to $\R$ then $\cC(\bY)$ is a quasi-tree;
  more generally, if all $X\in\bY$ are quasi-trees with a uniform
  bottleneck constant then $\cC(\bY)$ is a quasi-tree.
\item If each $X\in\bY$ is $\delta$-hyperbolic with
the same $\delta$, then $\cC(\bY)$ is
  hyperbolic.
\item If the collection $\bY$ has $\asdim\leq n$ uniformly, then
  $\asdim(\cC(\bY))\leq n+1$.
\item The quotient $\cC(\bY)/\bY$ obtained by collapsing the embedded
  copies of each $X\in\bY$ to a point is a quasi-tree.
\end{enumerate}}
\medskip

Note that (ii) in particular says that the space $\cC(\bY)$ obtained
from an orbit of axes in $\H^n$ as in the beginning of the
introduction is a quasi-tree and not (quasi-isometric to) $\H^n$.  The
space $\cC(\bY)/\bY$ is the {\it projection complex}
$\cP(\bY)=\cP_K(\bY)$, which depends on $K$. The main technical theorem in this paper is the
fact that $\cP(\bY)$ is a quasi-tree. We think of the quasi-tree of
metric spaces $\cC(\bY)$ as being
obtained from $\cP(\bY)$ by blowing up vertices to metric spaces, and
thus the terminology.

Theorem A and Theorem B are collections of theorems
proved mostly in Section \ref{s3}.

\medskip
\noindent
{\bf Guide to the reader:} Background, motivating examples, the main
results and applications are contained in Sections 1 and 2. In
particular the tree example in Section 2.5 will help the reader follow
the axiomatic approach in Sections 3 and 4. These
latter two sections are more technical. However, they start from a few
simple axioms and do not require any hyperbolic geometry or facts
about mapping class groups. In fact, Sections 3, 4.1 and 4.2 are
entirely self-contained with the exception of the use of Manning's
bottleneck property. The remaining subsections of Section 4 use some
basic facts about $\delta$-hyperbolic spaces and asymptotic
dimension. The theorems about the mapping class group are proved in
Section 5. If one prefers to skip Sections 3 and 4 one can read
Section 5 using the results of the earlier sections as a black box.

\medskip
\noindent
{\bf Acknowledgments.} The authors thank Martin Bridson,
Pierre-Emmanuel Caprace, Kasra Rafi, John H. Hubbard and Johanna Mangahas for helpful
discussions. We especially thank the referees for many useful comments
and for suggestions how to significantly improve the exposition.

\section{Applications of the construction}

In this section we present applications of our construction. The main
one we had in mind when we started this work is presented first. Some
of the other applications were worked out by others after the first
version of this paper was circulated.

Recall that a function
$f:\cX\to \cY$ between metric spaces is a {\it coarse embedding} if
there are constants $A,B$ and a function $\Phi:[0,\infty)\to
  [0,\infty)$ with $\Phi(t)\to\infty$ as $t\to\infty$ such that
$$\Phi(d_{\cX}(x,x'))\leq d_{\cY}(f(x),f(x'))\leq A~ d_\cX(x,x')+B$$
If we can take $\Phi(t)=A^{-1}t-B$,
$f$ is a {\it quasi-isometric embedding}, in other words,
$f$ gives a quasi-isometry between $\cX$ and its image by $f$
in $\cY$.

\subsection{Mapping class groups} 
Our main application in this paper is to the study of mapping class
groups. To apply our methods here we will use the notion of subsurface
projections of Masur-Minsky \cite{mm2} which has been a driving force
behind much of the recent development in the geometry of mapping class
groups. 

Let $\Sigma$ be a closed orientable surface, possibly with finitely
many punctures. The {\it mapping class group} $MCG(\Sigma)$ of
$\Sigma$ is the group of components of the orientation preserving
diffeomorphism group preserving the punctures. For simplicity we will
additionally assume that $\Sigma$ has a complete hyperbolic structure
of finite area in which the punctures correspond to cusps. The
standard reference in the subject is \cite{FM2}. 

To every isotopy class of $\pi_1$-injective non-peripheral subsurfaces
$Y\subset\Sigma$ we assign the 
{\it curve complex} $\cC(Y)$. To two such subsurfaces $Y,Z$ with
$\partial Y\cap\partial Z\neq\emptyset$ (this means that the
intersection is nonempty even after any isotopy) there is the
Masur-Minsky {\it
  subsurface projection} $\pi_Y(Z)\subset \cC(Y)$. We refer the reader
to Section 5.1, where these notions are reviewed. More generally, when
$\beta$ is a simple closed curve that cannot be isotoped to be
disjoint from $Y$, we have a projection $\pi_Y(\beta)\subset \cC(Y)$.  
The mapping class group $MCG(\Sigma)$ acts on the product $\prod_Y
\cC(Y)$ and we have an orbit map
$$\Psi:MCG(\Sigma)\to \prod_Y \cC(Y)$$
which is, as a {\it coarse map} (i.e. a point is 
mapped to a bounded set), more explicitly given by
$$\Psi(g)=(\pi_Y(g(\alpha)))_Y$$ 
where $\alpha$ is a finite {\it binding} collection of simple closed
curves (see Section 5.4).

The remarkable Masur-Minsky distance
formula (Theorem 6.12 in \cite{mm2}) says that the word norm $|g|$ of $g\in MCG(\Sigma)$
is {\it coarsely equal} (i.e., up to a multiplicative and additive error) to
$$\sum_Y \{\{d_{\cC(Y)}(\pi_Y(\alpha),\pi_Y(g(\alpha)))\}\}_M$$ where $M$ is
sufficiently large, and $\{\{x\}\}_M$ is defined as $x$ if $x>M$ and as
0 if $x\leq M$. 
In \cite{mm2} the distance formula is stated for the ``marking graph",
which is quasi-isometric to the mapping class group, see \cite[Section 7]{mm2}.

Morally, this formula says that $\Psi$ is a quasi-isometric
embedding. However, the product space is not a metric space (the
``cut-off'' distance is not a metric). More problematic, although we
now have much information about the individual curve complexes, this
embedding is in an infinite product which is difficult to work with.

In this paper, we use Theorem A to embed the mapping class group in a
{\it finite} product of {\em quasi-trees of curve complexes}. To do so
we show that essential subsurfaces can be grouped in finitely many
subcollections $\bY^1,\bY^2,\cdots,\bY^k$ so that the curve complexes
of subsurfaces in each $\bY^i$ satisfy (P0)-(P2) with respect to
subsurface projections, thus yielding the
quasi-tree of curve complexes $\cC(\bY^i)$. Everything can be done
equivariantly, so that we have an orbit map
$$MCG(\Sigma)\to\cC(\bY^1)\times
\cC(\bY^2)\times\cdots\times\cC(\bY^k).$$ In each $\cC(\bY^i)$ 
the distance is approximated by the Masur-Minsky formula restricted to
the summands in $\cC(\bY^i)$. Then the
Masur-Minsky formula can be interpreted as saying that the map of
$MCG(\Sigma)$ into the product of quasi-trees of curve complexes is a
quasi-isometric embedding. The choice of an orbit for the map is 
not important. As each factor is hyperbolic we have the
following theorem in Section \ref{section.mcg}:

\vskip .5cm
\noindent
{\bf Theorem C.} {\it $MCG(\Sigma)$ equivariantly quasi-isometrically
  embeds in a finite product of hyperbolic spaces.}
\vskip .5cm

The following result follows easily from the definition 
of {\it asymptotic cones} (see \cite{BDS,BDS2}).
\vskip .5cm
\noindent
{\bf Theorem.} [Behrstock-Drutu-Sapir] {\it Every asymptotic cone of
  $MCG(\Sigma)$ embeds by a
  bi-Lipschitz map in a finite product of $\R$-trees.}
\vskip .2cm
\noindent
In fact they prove more including some information on the geometry of the image of the embedding, but their theorem does not imply Theorem C.
 They use the notion of {\em tree-graded space} introduced in \cite{DS}.
\vskip .5cm

We now make a few comments on verifying the axioms (P0)-(P2) in this
setting, as this is the situation that crystallized the correct
axiomatic approach. Axiom (P0) was established by Masur-Minsky as part
of the subsurface projection setup and it follows easily from
definitions. Axiom (P1) was established by Behrstock \cite{jason} and
we refer to it, and to Axiom (P1) in general, as {\it Behrstock's
  inequality}. 
Axiom (P2) is a consequence of the Theorem 4.6 and Lemma 4.2 in \cite{mm2}.

A central idea in \cite{mm2} is the notion of a {\it hierarchy} and this is
used in the verification of the axioms by Masur-Minsky and
Behrstock. This is a powerful tool but it is complicated to define and
difficult to use. Leininger gave a very simple, hierarchy free proof
of (P1) (see \cite{johanna,johanna2}) and here we will show that (P2)
also has a direct, hierarchy free proof (see Lemma
\ref{finiteness}). Using this we can show that our map of the mapping
class group into the product of quasi-trees of curve complexes is a
{\it coarse embedding} without any of the results of \cite{mm2}. In
particular, we obtain a hierarchy free proof of the lower bound in the
Masur-Minsky formula. In fact, the proof of Theorem D below does not
depend on the results \cite{mm2} although the ideas of that paper are
certainly central to our proof.

In \cite{emr} Eskin-Masur-Rafi give a unified approach, using Theorem
C, to studying the large scale geometry
of Teichm\"uller space with either the
Teichm\"uller metric or the Weil-Petersson metric, or of the mapping
class group with the word metric.

It is a theorem of Bell-Fujiwara \cite{bell-fujiwara} that each curve
complex has finite asymptotic dimension. More recently, Richard Webb
\cite{webb} found explicit bounds on the asymptotic dimension of curve
complexes. His bound was improved to a linear bound by
Bestvina-Bromberg \cite{BB} by a different method.
Thus from Theorems B and C we obtain
the following theorem, which motivated this work
(see Section \ref{section.mcg}):

\medskip
\noindent
{\bf Theorem D.} {\it Let $\Sigma$ be a closed orientable surface,
possibly with punctures. Then $\asdim(MCG(\Sigma))<\infty$.
}
\medskip

As a consequence of the bounds on the asymptotic dimension of curve
complexes mentioned above, it follows that $\asdim(MCG(\Sigma))$ is
bounded by an exponential function in the complexity of the surface.

The {\it Coarse Baum-Connes conjecture} (for torsion free subgroups of
finite index) and therefore the {\it Novikov conjecture} follows
\cite{yu}, cf. \cite{roe}. Various other statements that imply the Novikov conjecture
were known earlier (see \cite{kida,ursula2,behrstock-minsky2}). We
note that Hume \cite{hume} improved this result and showed that
$MCG(\Sigma)$ has finite {\it Assouad-Nagata dimension} (meaning that in the
definition of asymptotic dimension the diameter of each set in the
cover is bounded by a linear function of $R$) and
quasi-isometrically embeds in a finite product of trees.

\medskip
\noindent
{\bf Theorem E.} [Theorem \ref{teichmuller}, Theorem \ref{weil}] {\it The Teichm\"uller space of $\Sigma$, with either
  the Teichm\"uller metric or the Weil-Petersson metric, has finite
  asymptotic dimension.
}
\medskip

Recall that the {\it translation length} $\tau(g)$ of an isometry
$g:X\to X$ is
$$\tau(g):=\lim_{k\to\infty} \frac{d_X(x,g^k(x))}k$$
The limit exists and is independent of $x\in X$.
We say the isometry is {\it hyperbolic} if 
$\tau(g) >0$.
When $X$ is a quasi-tree an isometry with unbounded orbits necessarily
has positive translation length, \cite{manning:qfa}.
The following theorem uses the observation that the
$MCG(\Sigma)$-orbit of a curve in a surface of even genus that
separates into subsurfaces of equal genus consists of pairwise
intersecting curves.

\medskip
\noindent
{\bf Theorem F.} {\it The mapping class groups in even
genus can act on quasi-trees with a Dehn twist having unbounded
orbits.} 
\medskip

See Theorem \ref{dehn hyperbolic}. In the case of odd genus one has to
pass to a subgroup of finite index. 
It follows that each Dehn twist has linear growth 
in the word length  in $MCG(\Sigma)$
(known by \cite{FLM}).
 Theorem F provides a sharp
contrast to a result of Bridson \cite{bridson}, who showed that in
semi-simple actions of mapping class groups (of genus $>2$) on
complete CAT(0) spaces Dehn twists are always elliptic.
A group action is {\it semi-simple} if each element
has either a bounded orbit or positive translation length. In the CAT(0)
case one gets a homomorphism from the centralizer of a Dehn twist to
$\R$ by looking at the action on the purported axis (identifying all
{\it parallel} axes to one); in our quasi-tree setting a similar
construction produces only a quasi-morphism on the centralizer.
We say that two (quasi-)geodesics are {\it parallel} if their
Hausdorff distance is finite.
  In
genus $>2$ the centralizer of a Dehn twist has no nontrivial
homomorphisms to $\R$, but does admit many quasi-morphisms.

By a {\it thickening} of a metric space $X$ we mean a quasi-isometric
embedding $X\to Y$. When $X$ is a graph with edges of length 1 and
$d\geq 1$, there is a particular thickening $X\to P_d(X)$ called the
{\it Rips complex} of $X$. The space $P_d(X)$ is a simplicial complex
with the same vertex set as $X$ and with simplices consisting of
finite collections of vertices with pairwise distance at most $d$.

\medskip
\noindent
{\bf Theorem G.} [Corollary \ref{Rips counter}] {\it There is an isometric action of a group on a graph $X$
which is a quasi-tree such that no equivariant thickening admits an
equivariant $CAT(0)$ metric. In particular, for no $d\geq 1$ does the Rips
complex $P_d(X)$ admit an equivariant $CAT(0)$ metric.}
\medskip

It is a long-standing open question whether every $\delta$-hyperbolic
group acts cocompactly and properly by isometries on a CAT(0)
space. One approach is to consider the Rips complex $P_d(X)$ for the
Cayley graph $X$ of the group and large $d$. Theorem G is not a
counterexample to this approach since our $X$ is not locally finite,
but it does point out difficulties. Note that in light of
\cite{MSW} the quasi-trees that arise in our
construction are necessarily locally infinite,
since otherwise we would be able to promote our group 
actions on quasi-trees to group actions on simplicial trees without
fixed points, which is not possible for certain groups. 

\subsection{Hyperbolic-like groups}\label{wpd-defn}

At the beginning of the introduction we indicated how a discrete group
of isometries of $\H^n$ that contains an element with an axis gives
rise to data satisfying our axioms, and thus the same group acts on
the quasi-tree of lines, which itself is a quasi-tree by Theorem B
(ii). 

The essential feature of this example is that the axis $\ell$ is {\it
  $B$-contracting} for some $B\geq 0$. This means that the nearest
point projection to $\ell$ of any metric ball disjoint from $\ell$ has
diameter bounded by $B$. See \cite{rank1}. More generally, one can
define the notion of $B$-contracting for any subset of a metric space using the nearest point projection to the subset. 

To state the theorem, assume that a group $\Gamma$ acts by isometries
on a geodesic metric space $X$, that $\gamma\in \Gamma$ acts
{\it hyperbolically} (i.e. any orbit map is a quasi-isometric embedding, or
equivalently the translation length of $\gamma$ is positive) and that
$\gamma$ is a WPD element \cite{bestvina-fujiwara}, that is, for all
$D>0$ and $x\in X$ there exists $M>0$ such that
$$\{g\in \Gamma\mid d(x,g(x))\leq D, d(f^M(x),gf^M(x))\leq D\}$$
is finite. We also say that two orbits are {\it parallel} if their
Hausdorff distance is finite.

\medskip
\noindent
{\bf Theorem H.}  {\it Let $\Gamma$ act on a geodesic metric space $X$
  such that $\gamma\in\Gamma$ is a hyperbolic WPD element with a
  $B$-contracting orbit. Then the
  collection of parallel classes of $\Gamma$-translates of a fixed
  $\gamma$-orbit (of a point) with nearest point projections satisfies (P0)-(P2)
  and thus $\Gamma$ acts on a quasi-tree. In addition, in this action
  $\gamma$ is a hyperbolic WPD element.} 

\medskip

In this form the theorem is proved in \cite{bc}.
We do not assume that $X$ is hyperbolic nor 
CAT(0). The main part of the
proof consists of verifying (P0)-(P2) and applying Theorem A in this
situation. The rest is included as Proposition \ref{wwpd}.
Dahmani-Guirardel-Osin \cite[Section 4.5]{DGO} prove a variation of Theorem H where $X$ is assumed to be hyperbolic and use it to
construct many examples of {\it hyperbolically embedded subgroups}
(see \cite{DGO} for the definition).

\begin{examples}\label{ex}
The following examples all satisfy Theorem H.
One considers the translates of an axis, or more generally
the orbit of a point,  of a hyperbolic 
WPD element.
\begin{enumerate}[(1)]
\item $\Gamma$ is a discrete group of isometries of $\H^n$ that
  contains an element $\gamma$ with an axis. $\gamma$ is WPD since 
the action of $\Gamma$ is properly discontinuous, and the axis is $B$-contracting since $\H^n$
is $\delta$-hyperbolic. 

\item $\Gamma$ is a group of isometries of a connected
  $\delta$-hyperbolic graph $X$ that contains a hyperbolic, therefore its (quasi-)axis is  $B$-contracting, 
  WPD element.
In particular, this
  construction applies to the curve complex and the mapping class
  group of a compact surface, where pseudo-Anosov elements
are hyperbolic and WPD \cite{bestvina-fujiwara}. This class of groups contains many
  groups with Kazhdan's property (T) and therefore every isometric
  action on a simplicial tree has a fixed point (cf. \cite{harpe-valette}).

\item $\Gamma$ is a discrete group of isometries
(i.e. the group action is metrically properly discontinuous) of a $CAT(0)$-space
  that contains a {\it rank 1 element} $\gamma$ in the sense of Ballmann. That is, $\gamma$ has
  an axis which is $B$-contracting for some $B\geq 0$. For example,
  pseudo-Anosov mapping classes are rank 1 elements in the action on
  the Weil-Petersson completion of Teichm\"uller space. In the
  cocompact setting this is equivalent to the more familiar condition
  that the axis does not bound a half-flat. Those elements 
are WPD although the action of the mapping class group
is  not properly discontinuous.  See \cite{rank1}. There are
  classifications of rank 1 elements in Coxeter groups
  \cite{caprace-fujiwara}, right angled Artin groups
  \cite{behrstock-charney} and cube complexes \cite{CS}.

\item 
  $\Gamma$ is the mapping class group, acting on Teichm\"uller space
  with Teichm\"uller metric, and $\gamma$ is a pseudo-Anosov mapping
  class. By 
  \cite{minsky-contraction} the axis of $\gamma$ is
  $B$-contracting. It is WPD since the action is properly discontinuous. 

\item $\Gamma=Out(F_n)$ acting on Culler-Vogtmann's Outer space $CV$
  \cite{CV}, equipped with the Lipschitz metric (which fails to be
  symmetric, see \cite{AKB}). The action is properly discontinuous. 
See \cite{vogtmann.survey}, \cite{vogtmann.problem},
\cite{vogtmann.icm} for more information on $Out(F_n)$ and Outer space.
An element $f$ of $Out(F_n)$ is {\it
    fully irreducible} if there are no conjugacy classes of proper
  free factors of $F_n$ which are $f$-periodic. Such elements have
  axes in $CV$, see \cite{bers}. In \cite{yael} Yael Algom-Kfir shows
  that there is $\nu>0$ such that the projection of any translate
  $\gamma(X_i)$ to any nonparallel $X_j$ is bounded by $\nu$, and she
  also shows that the axes are $B$-contracting for some $B$. Even though the
  metric is not symmetric, axioms hold. Axioms (P0)-(P1) are explicitly
  verified in \cite{yael} and Axiom (P2) follows quickly from the
  arguments in \cite{yael}, for details see \cite{bc}.

\end{enumerate}

\end{examples}

\begin{remark}\label{WPD2}
Suppose a group $G$ acts on a geodesic space $X$ with a hyperbolic 
element $g \in G$ with a $g$-orbit $\alpha$.
The {\em elementary
  closure}, $EC(g)$, of $g$ is the subgroup of elements $h \in G$ such
that $h(\alpha)$ is parallel to $\alpha$.
When $X$ is $\delta$-hyperbolic, $g\in G$ is a WPD element if and only if $EC(g)$ is virtually cyclic and 
for some (any) $x\in X$ there is $B>0$ such
that any $\phi\in G-EC(g)$ maps the orbit $\langle g \rangle x$ to a set whose
projection to $\langle g \rangle x$ has diameter $\leq B$. In this setting the orbit
is a quasi-geodesic and the projection is the nearest point
projection, coarsely defined (see the comment before Corollary
\ref{nearest_point2}). Thus WPD is equivalent to saying that
the set of translates of a $g$-orbit is ``discrete'' in the sense that
any two are either parallel or have bounded ``overlap'', with parallel
orbits coming from translating by elements in $EC(g)$.
\end{remark}

The work of Epstein-Fujiwara \cite{epstein-fujiwara} implies that
non-elementary (i.e. not virtually cyclic) hyperbolic groups have many unbounded actions on quasi-lines, i.e., geodesic spaces quasi-isometric to a line. Manning \cite{manning} gave a construction of
an action of a group $G$ on a quasi-tree starting with a
quasi-morphism $G\to\R$ (equivalently, an action of $G$ on a
quasi-line) but it is not clear when such actions are non-elementary
(i.e. have unbounded orbits and do not fix an end nor a pair of ends). A map $f:G \to \R$ is a 
{\it quasi-morphism} if there exists a constant $C$ such that
for all $g,h \in G$ 
$$|f(gh)-f(g)-f(h)| \le C.$$
Recently, it has been verified that the actions by Manning
are not elementary for certain cases using our work
\cite{martinez}.
We will also verify that the non elementary 
groups in Example \ref{ex}
have non elementary actions on quasi-trees (Corollary \ref{bushy}).

Recall that an isometric group action is {\it acylindrical} if for
every $D>0$ there exist $R,N>0$ such that $d(x,y)\geq R$ implies that
the set
$$\{g\in G\mid d(x,g(x))\leq D, d(y,g(y))\leq D\}$$ has cardinality at
most $N$.  
Osin develops a
theory of {\it acylindrically hyperbolic groups}: these are groups
that admit a non-elementary acylindrical isometric action on a
hyperbolic space. 

\medskip
\noindent
{\bf Theorem I.} {\it (Osin \cite{osin}) Let a group $\Gamma$, which is
not virtually cyclic, act on a $\delta$-hyperbolic metric space $X$ such that
$\gamma\in\Gamma$ is a hyperbolic WPD element. Then $\Gamma$ is an
acylindrically hyperbolic group. Thus all groups in Examples \ref{ex}
are acylindrically hyperbolic.}
\medskip

From the point of view of this paper, Osin considers a slightly
different projection distance $d_Y(x,z)$ (within uniformly bounded
distance of ours) which is better behaved, so that the action on
the resulting quasi-tree of metric spaces $\cC(\bY)$ constructed in
exactly the same way, but with each copy
$Y\in\bY$ {\it electrified} (i.e., any two 
points in $Y$ is joined by an edge of length 1), is acylindrical.

Caprace and Delzant pointed out the following curious
corollary. Recall that Burger-Mozes \cite{BM2} constructed an example of a simple group, which acts freely and cocompactly on the product of two
trees. Thus the quotient is a finite non-positively curved square
complex with finitely-presented, infinite simple fundamental group.

\begin{cor}[Caprace-Delzant]
Suppose $Z$ is a finite non-positively curved square complex with no
free edges whose fundamental group is simple. Then the universal cover
$\tilde Z$ is isometric to the product of two trees.
\end{cor}

\begin{proof}
By the Ballmann-Brin Rank Rigidity Theorem \cite[Th C]{BaBr} (see also
\cite{CS}) the universal cover $\tilde Z$ is either the product of two
trees or the deck group contains a rank 1 element (there is a third
possibility in general that $\tilde Z$ is a Euclidean building, which
we can exclude since $Z$ is a square complex).  In the latter case,
using Theorem H, we see that $\pi_1(Z)$ acts on a quasi-tree and
contains a hyperbolic WPD element $\gamma$. $\pi_1(Z)$ is non
elementary since it is simple and torsion-free.  Then by the work of
Dahmani-Guirardel-Osin \cite{DGO} the normal closure of $\gamma^m$ is
a free group when $m>0$ is sufficiently large, so $\pi_1(Z)$ is not
simple. (The result of \cite[Section 5 and 6]{DGO} applies to a
hyperbolic WPD element $\gamma$ that acts on a hyperbolic space.)
\end{proof}

\subsection{Bounded cohomology}

As we said Manning \cite{manning} used bounded cohomology/quasi-morphisms to show that
many groups acted on quasi-trees. Conversely, the existence of actions
of a group on a quasi-tree (with a hyperbolic WPD element) has a
consequence that the second bounded cohomology (even with coefficients in
certain representations) is ``big''. One can use such actions to give
unified constructions of quasi-morphisms on various groups $G$, and even
quasi-cocycles with coefficients in unitary representations in
uniformly convex Banach spaces.  
The case of the regular representation on $\ell^2(G)$
is of particular importance (see \cite{monod.icm}).
We investigate this in \cite{uc}.

In fact, Theorem H can be regarded as a completion of Manning's program
\cite{manning} showing that all (known) groups with big second bounded
cohomology admit (many) interesting actions on quasi-trees.

By contrast, there are many groups that do not admit nontrivial (namely, orbits are unbounded)
actions on a quasi-tree. Recall \cite{manning:qfa} that a group $G$
satisfies QFA if every action on a quasi-tree has bounded
orbits. Equivalently (see e.g. \cite{manning}) every quasi-action on a
tree has bounded orbits. If $G$ is an irreducible lattice in a higher
rank semi-simple Lie group with finite center, it is expected that $G$
has QFA. For $SL_n(\Z)$, $n\geq 3$, this is a result of Manning
\cite{manning:qfa}.

Developing Theorem E further and using Theorem H, in \cite{scl} we
construct bounded cohomology classes that are unbounded on powers of a
Dehn twist. In fact, expanding on this idea we give a precise
characterization of mapping classes that have nonzero stable
commutator length.

\subsection{$Out(F_n)$}

There is a program to prove Theorem D and a version of Theorem C for
the outer automorphism group $Out(F_n)$ of a free group $F_n$ of rank
$n$. 
 There are (at least) two analogs of the curve complex, namely the
complex of free factors and the complex of free splittings. Both have
recently been shown to be hyperbolic, the former in \cite{b-feighn2}
and the latter in \cite{handel-mosher}. The analog of subsurface
projections was defined in \cite{b-feighn} and the end result is

\medskip
\noindent
{\bf Theorem J.} {\it \cite{b-feighn} $Out(F_n)$ acts isometrically on
  a finite product of hyperbolic spaces so that every element of
  exponential growth acts with positive translation length.}
\medskip

See \cite{sisto1}, \cite{sisto2}, \cite{sisto3}, \cite{delzant}
for further applications of the projection complex techniques. 

\subsection{The tree example} \label{tree example}

Let $F_2 = \langle a,b\rangle$ be the
free group on two generators. Embed its Cayley graph (tree) in $\R^2$ such
that the $a$ edges are horizontal and the $b$ edges are vertical. The
horizontal lines are the axes of $a$ and its conjugates and we let
$\bY$ be the set of horizontal lines. Note that if $X$ and $Y$ are in
$\bY$ then $\pi_Y(X)$, the nearest point projection of $X$ to $Y$, is
a single point, and we have (P0). Then $d_Y^\pi(X,Z)$ is the diameter of the union of
$\pi_Y(X)$ and $\pi_Y(Z)$ which is of course just the distance between
the two points.

\begin{figure}
\centerline{\scalebox{0.7}{\input{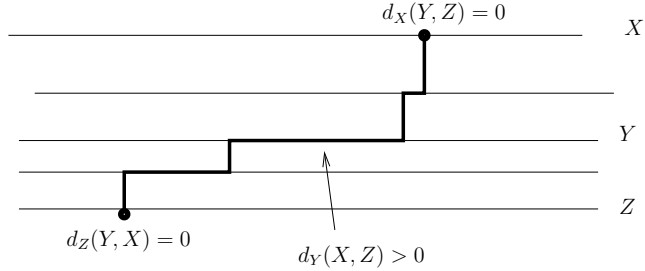}}}
\caption{Axiom (P1) for the set of horizontal 
lines in the Cayley tree. The bold line is the shortest segment 
between $X$ and $Z$}
\end{figure}

In this example it is quite easy to check that the axioms hold. We
first note that to calculate $d_Y^\pi(X,Z)$ we take the unique shortest segment  in the Cayley tree from $X$ to $Z$. If this segment
intersects $Y$ then the intersection will be a closed segment one
endpoint of which is $\pi_Y(X)$ and the other is $\pi_Y(Z)$. Then
$d_Y^\pi(X,Z)$ will be the length of the intersection. If the segment from
$X$ to $Z$ doesn't intersect $Y$ then we will have $\pi_Y(X) =
\pi_Y(Z)$ and $d_Y^\pi(X,Z) = 0$. Therefore if $d_Y^\pi(X,Z)>0$ then $d_X^\pi(Y,Z)
= d_Z^\pi(Y,Z) = 0$ which is exactly (P1) where $\xi=0$.
For (P2) we note that the elements of $\bY$
are all disjoint in the Cayley graph and therefore if the segment from
$X$ to $Z$ has length $D$ then there are at most $D/K$ elements $Y \in
\bY$ with $d_Y(X,Z) > K$ for any $K>0$.
(Notice that in the Cayley tree, $d_Y(X,Z) \ge 1$
if $d_Y(X,Z) > 0$.)

We now define the projection
complex $\cP_K(\bY)$ in this special case. Fix a constant 
$K>0$. The vertex set is
$\bY$. Two distinct vertices $X,Y$ are joined by an edge if and only
if for every $Z\in\bY$ with $X\neq Z\neq Y$ we have $d_Z^\pi(X,Y)\leq K$. 

We leave to the reader to show that for every $K>0$, $\cP_K(\bY)$ is connected (see Proposition \ref{connected}). 

To see that $\cP_K(\bY)$ is a quasi-tree we use Manning's bottleneck
criterion, which can be expressed in the following equivalent and more
convenient form: for each pair of vertices $X$ and $Y$ in $\cP_K(\bY)$
there is a path $\gamma$ joining $X$ and $Y$ in $\cP_K(\bY)$ such that any path from $X$ to
$Y$ passes within uniform distance of any vertex on $\gamma$. The
key to proving this is Proposition \ref{key} which can be paraphrased
to say that if $\{X_0,\dots, X_k\}$ is a path of vertices in
$\cP_K(\bY)$ such that each element is distance 3 or more from a
vertex $Z$ then the projection of the path to $Z$ has uniformly
bounded diameter.

In the special case that we are examining in this section it is
actually quite easy to prove an even stronger version of Proposition
\ref{key}. In this special case, if the path $\{X_0, \dots, X_k\}$ is
distance two or greater from $Z$ then $\pi_Z(X_0) = \pi_Z(X_k)$. To
prove this we take the shortest segment from $X_0$ to $Z$ in the tree and let $W
\in \bY$ be the line that contains the last horizontal sub-segment of
length $>K$ of
this segment before reaching $Z$. Such $W$ must exist since 
the distance between $X_0$ and $Z$ is at least 2. A simple inductive argument shows
that $W$ will be the line that contains the last horizontal
sub-segment of length $>K$ of the shortest
segment from $X_i$ to $Z$ for all $i =0, \dots, k$ and therefore
$\pi_Z(X_i) = \pi_Z(W)$ for all $i = 0, \dots, k$ by (P1) with $\theta=0$.

\begin{figure}
\centerline{\scalebox{0.7}{\input{tree.bottleneck.pstex_t}}}
\caption{ Proposition \ref{key} for the tree.}
\end{figure}

To finish the proof that $\cP_K(\bY)$ is a quasi-tree we examine the
sets, denoted by $\bY_K(X,Y)$, of vertices $Z \in \cP_K(\bY)$ with $d_Z(X,Y) \ge
K$. As mentioned above for each $Z \in \bY_K(X,Y)$ the shortest segment from
$X$ to $Y$ intersects $Z$. We then order the set by how these
intersections appear on the segment. With this ordering it is easy to
check that this set is a path, $\gamma$, from $X$ to $Y$ in $\cP_K(\bY)$. Proposition
\ref{key} we just discussed
implies that any path $\{X_0= X, X_1, \dots, X_k=Y\}$ must go within
distance one of every vertex $Z \in \bY_K(X,Z)$ for if not $\pi_Y(X) =
\pi_Y(Z)$ and $d_Z(X,Y) = 0 \not> K$, a contradiction. Hence,
Manning's bottleneck criterion holds for the path $\gamma$
with the constant 2 and $\cP_K(\bY)$ is a quasi-tree.

Notice that the element $a^nb^n$ is hyperbolic in $\cP_K(\bY)$ when
$K<n$ (cf. Lemma \ref{translation}) and it has bounded orbits when $K\geq n$. Thus for $K<K'$ the
natural map $\cP_K(\bY)\to\cP_{K'}(\bY)$ is Lipschitz but in general
it is not a quasi-isometry.

Also note that $\cP_K(\bY)$ is not locally finite; the infinite set of
horizontal lines that intersect a fixed vertical line are all
connected to each other in pairs.

\subsection{Plan of the paper}

We now briefly indicate the highlights of each section of the paper.

\medskip
\noindent
{\bf Section 3.} We define the projection complex starting from the
axioms. An important technicality is that we have to perturb the
initial pseudo-distance function $d^\pi$ by a bounded amount to a new
function $d$ in order to achieve a certain Monotonicity Property. The
main properties of this perturbed distance are listed in Theorem
\ref{main}. Perhaps the most important property is that the finite set
in axiom (P2) has a natural total order; this is motivated by the
Masur-Minsky hierarchy machinery.

Next, we focus on proving that the
projection complex is a quasi-tree (Theorem \ref{quasitree}). Roughly speaking, the proof
follows the argument for the tree example in section \ref{tree
  example}. There is a significant technical point here. When the
constant $\xi$ is positive, there is no reason that the projections of
the $X_i$ to $Z$ are all the same point, but instead they might be
slowly making progress along $Z$. In order to rule this possibility
out we introduce the notion of a {\it guard} and a closely related
notion of a {\it barrier}.  In order for the sequence $X_i$ to make
progress in $Z$, it first has to do so in a suitable guard. When it
looks like the given guard has been cleared, another one appears that
also must be cleared before any progress in $Z$ is made, etc. See
Lemma \ref{nested_guards} and Proposition \ref{key}. In the Cayley tree example, $W$ is a barrier. We also record an
upper and a
lower bound on the distance in the projection complex in the spirit of
the Masur-Minsky distance formula. See Proposition \ref{connected} and
Lemma \ref{masur-minsky0}. We end the section with the study of the
basic properties of the group action on the projection complex, 
including WPD (see Remark \ref{WPD free}).

\medskip
\noindent
{\bf Section 4.} We define the quasi-tree of metric spaces $\cC(\bY)$, which depends on a constant $K>0$, 
essentially by blowing up each vertex of the projection complex to the
associated metric space. In the Cayley tree example, we 
replace (blow up) the vertex for each horizontal line by the horizontal line itself.
To say it differently, we take the disjoint
union of the given collection of metric spaces and we attach edges
from every point of $\pi_Y(Z)$ to every point of $\pi_Z(Y)$ provided
that the projection distance $d_W(Y,Z)$ does not exceed some threshold
$K$ for all $W\neq Y,Z$. We then develop the basic geometry of
$\cC(\bY)$. We prove the distance formula, analogous to
Masur-Minsky's, in Theorem \ref{masur-minsky1}. We also show that the
nearest point projection of $Z$ to $Y$ in $\cC(\bY)$ coarsely agrees (i.e. in bounded Hausdorff distance) with the given
set $\pi_Y(Z)$. This will prove Theorem A. Technically, the proofs
consist of lifting the notions of guards and barriers from the
projection complex to $\cC(\bY)$.

We then proceed by to prove that various properties that hold for each
$Y\in\bY$ uniformly continue to hold for $\cC(\bY)$
in Section \ref{section.hyperbolic}. This includes
hyperbolicity, being a quasi-tree, having bounded asymptotic
dimension, and quasi-convexity. In particular, this will prove Theorem B. For example, in the Cayley tree example, $\cC(\bY)$
is a quasi-tree. 
Lastly in Section \ref{section.wwpd}, 
 for the purposes of \cite{scl} we also
discuss a certain property of the group action, called WWPD, which is
weaker than WPD. 

\medskip
\noindent
{\bf Section 5.} This section is focused on the mapping class group
and here we prove all the other theorems stated in the
introduction. Subsurface projections are defined only for subsurfaces
whose boundaries intersect. The main technical issue we have to
address is how to divide the collection of all subsurfaces into
finitely many families so that within each family subsurface
projections are well defined. This problem is quickly reduced to
showing that the curve graph (i.e. the 1-skeleton of the curve
complex) has finite coloring. We in fact show that there exist such a
coloring so that the mapping class group acts by permuting the
colors. The finite index subgroup that preserves all colors has the
property that for each of its elements $g$ and every curve $a$, either
$g(a)=a$ or $g(a)$ and $a$ intersect. See Lemmas \ref{colors} and
\ref{finite index}.

\section{The projection complex}\label{section.projection}

We start by introducing the projection complex. To define it we don't
really need the projections $\pi_A(B)$ as in axioms (P0)-(P2); we only
need the pseudo-distance $d_C^\pi(A,B)$. Accordingly the axioms are
weakened. 

\subsection{Projection complex axioms}

Let $\bY$ be a set, $\xi\ge 0$ a constant and assume that for each $Y \in \bY$ we have a
function $$\pd_Y : (\bY\setminus\{Y\}) \times (\bY\setminus\{Y\})
\longrightarrow [0,\infty).$$
The {\it projection complex axioms}
are  the following:
\begin{enumerate}[(PC 1)]

\item $\pd_Y(X,Z) = \pd_Y(Z,X)$;

\item \label{triangle} $\pd_Y(X,Z) + \pd_Y(Z,W) \ge \pd_Y(X,W)$ (triangle inequality);

\item \label{chris}   $\min\{\pd_Y(X,Z), \pd_Z(X,Y) \} \leq \lc$;

\item \label{finite}   for all $X,Z \in \bY$, $\# \{Y |
  \pd_Y(X,Z) > \xi\}$ is finite.

\end{enumerate}

As an analog of uniform boundedness of the projections $\pi_Y(Z)$ we
could require that $d^\pi_Y(Z,Z)\leq\xi$, but this will not be used in
the sequel.

\subsection{Monotonicity}

Given distance functions that satisfy the above axioms it is useful to
modify them by a bounded amount in order to achieve the Monotonicity
property (F) of Theorem \ref{main}. See Remark \ref{nonmonotonic} for
an example where (F) fails. The price we will have to pay is that
triangle inequality will hold only up to a bounded error.

The following definition is motivated by the Masur-Minsky hierarchy
theory.

\begin{definition} For $X,Z\in\bY$ with $X\neq Z$ let
$\cH(X,Z)$ to be the set of pairs $(X',Z')
\in \bY \times \bY$ with $X'\neq Z'$ such that one of the following
four holds:
\begin{itemize}
\item  both $d^\pi_X(X',Z'), d^\pi_Z(X',Z') >2\lc$;

\item $X = X'$ and $d^\pi_Z(X,Z')>2\lc$;

\item $Z = Z'$ and $d^\pi_X(X',Z)>2\lc$;

\item $(X',Z') =(X,Z).$
\end{itemize}
We can now define the modified distance functions
$$d_Y: (\bY\backslash \{Y\}) \times (\bY \backslash \{Y\}) \to [0,\infty)$$
by
$$d_Y(X,Z) =0$$
if $Y$ is contained in a pair in $\cH(X,Z)$ and
$$d_Y(X,Z) = \underset{(X',Z') \in \cH(X,Z)}{\inf} d^\pi_Y(X',Z')$$
otherwise. For example, if  $\pd_Y(W,Z) >2\xi$, then $(W,Z) \in \cH(Y,Z)$ and
$d_W(Y,Z)=0$.
\end{definition}

Note that it is clear from the definition that $d_Y(X,Z) \leq
\pd_Y(X,Z)$ and therefore (PC \ref{chris}) 
still holds for $d_Y$ with the
same constant. However we need to modify  (PC \ref{triangle})
to a
coarse triangle inequality.

\begin{prop}\label{quasiequal}
If $(X',Z') \in \cH(X,Z)$ then for every $Y\in\bY$,
$Y\not\in\{X,Z,X',Z'\}$ we have
$$d^\pi_Y(X,Z) - d^\pi_Y(X',Z') \leq 2\lc.$$
\end{prop}

\begin{proof} If $d^\pi_Y(X,Z) \leq 2\lc$ we are done since the
distances are always nonnegative. We note that if $Y$ is contained in a pair in $\cH(X,Z)$ then $\pd_Y(X,Z) \le 2\theta$ by an application of (PC 2) and (PC 3).
For the rest of the proof we now assume that $\pd_Y(X,Z) > 2\lc$ and in particular that $Y$ is not contained in a pair in $\cH(X,Z)$.

We first assume that $X$ and $Z$ are distinct form $X'$ and $Z'$. By the triangle inequality
$$d^\pi_X(X',Y) + d^\pi_X(Y,Z') \geq \pd_X(X',Z') > 2\lc$$
and therefore
$$\max\{\pd_X(X',Y), \pd_X(Y,Z')\} > \lc.$$
Without loss of generality we assume that $\pd_X(X',Y)> \lc$.

By (PC \ref{chris}) we have $\pd_Y(X,X') \leq \lc$ and again applying
the triangle inequality we have
$$\pd_Y(X,X') + \pd_Y(X',Z) \geq \pd_Y(X,Z) > 2\lc$$ and therefore
$$\pd_Y(X',Z)> 2\lc - \lc = \lc.$$
Another application of (PC \ref{chris}) gives us that $\pd_Z(X',Y) \leq \lc$.

We now apply the triangle inequality exactly as we did at the start of
the proof but replacing $X$ with $Z$. Again we get that
$$\max\{\pd_Z(X',Y), \pd_Z(Z',Y)\}> \lc$$ and since we have just seen
that $\pd_Z(X',Y) \leq \lc$ we must have $\pd_Z(Z',Y) > \lc$. Then by
(PC \ref{chris}), $\pd_Y(Z,Z') \leq \lc$.

To finish the proof  in this case we make one final application of the triangle
inequality to see that
$$\pd_Y(X,X') + \pd_Y(X',Z') + \pd_Y(Z',Z) \geq \pd_Y(X,Z)$$ and
therefore
$$\pd_Y(X,Z) - \pd_Y(X',Z') \leq 2\lc.$$

For pairs of the form $(X',Z)$ with $X' \neq X$ the proof is
easier. As before we have the inequality
$$d^\pi_X(X',Y) + d^\pi_X(Y,Z) \geq d^\pi_X(X',Z) > 2\lc.$$
Since $d^\pi_Y(X,Z)> 2\lc$ we must have $d^\pi_X(Y,Z)\leq \lc$ and therefore $d^\pi_X(X',Y)> \lc$ and $d^\pi_Y(X,X') \leq \lc$. We once again apply the triangle inequality to see that
$$d^\pi_Y(X,X') + d^\pi_Y(X',Z) \geq d^\pi_Y(X,Z)$$
and therefore
$$d_Y(X,Z) - d_Y(X',Z) \leq \lc \leq 2\lc.$$

The statement is trivial if $(X',Z') = (X,Z)$ so the proof is finished.
\end{proof}

This result has number of important consequences. Before stating them
we set notation that helps prevent a proliferation of constants. 
Given a constant $\theta \ge 0$, we say
that $x \succ y$ or $y \prec x$ if $y-x$ is bounded above by a
constant depending only on $\xi$. We also define $x \sim y$ if $x
\succ y$ and $y \succ x$. For example  (PC 3) implies
$$\min\{d_Y(X,Z), d_Z(X,Y)\} \sim 0.$$
Thus, for the purposes of this notation, we regard $\xi$ as a variable
that depends on the particular setting.
Note that transitivity holds, i.e. if $x\succ y$ and $y\succ z$ then
$x\succ z$, but the constant bounding $z-x$ is worse. Thus it is
important to ensure that transitivity is applied only to chains of
bounded length. 

Next for a constant $K>0$ we define $\bY_K(X,Z)$ to be the set of $Y \in \bY$ such that $d_Y(X,Z) >K$.

Here are the properties of the functions $d_Y$, gathered together in one theorem. One can think of them as axioms.

\begin{thm}\label{main}
There exists a $\xio>0$, depending only on $\xi$, such that the following properties hold:\\
\\
\begin{enumerate}[(A)]
\item {\bf Symmetry} $$d_Y(X,Z) = d_Y(Z,X)$$

\item {\bf Coarse equality} For all distinct $X$, $Y$ and $Z$
$$\pd_Y(X,Z) \prec  d_Y(X,Z) \leq \pd_Y(X,Z).$$

\item {\bf Coarse triangle inequality}
$$d_Y(X,Z) + d_Y(Z,W) \succ d_Y(X,W).$$

\item {\bf Inequality on triples}
$$\min\{d_Y(X,Z), d_Z(X,Y) \}\sim 0$$

\item {\bf Finiteness} $\#\{Y |
  d_Y(X,Z)\ge \xio\}$ is finite for all $X,Z \in \bY$.

\item {\bf Monotonicity} If $d_Y(X,Z) \geq \xio$ then both $d_W(X,Y),
  d_W(Z,Y) \leq d_W(X,Z).$

\item {\bf Order} The set $\bY_\xio(X,Z) \cup \{X,Z\}$ is totally
  ordered with least element $X$ and greatest element $Z$ such that given
  $Y_0, Y_1, Y_2 \in \bY_\xio(X,Z) \cup \{X,Z\}$, if
  $Y_0 <Y_1<Y_2$  then
$$d_{Y_1}(X,Z) \prec d_{Y_1}(Y_0,Y_2) \le d_{Y_1}(X,Z),$$
and 
$$d_{Y_0}(Y_1, Y_2) \sim 0 \mbox{ and } d_{Y_2}(Y_0,Y_1) \sim 0.$$

\item {\bf Barrier property} If $Y \in \bY_{\xio}(X_0,Z)$ and $Y \in
  \bY_{\xio}(X_1,Z)$ then $$d_Z(X_0,X_1) <\xio.$$
\end{enumerate}
\end{thm}

\begin{proof} 
For each property we will see that there is some constant $\xio$ so
that the property holds for any larger choice of constant. Therefore, 
 in the proof of each property, we will use the properties we have 
 already showed. Throughout
the proof one should think of $\xi$ as being fixed but $\xio$ as a
variable that won't be fixed until the end of the proof.

(A) - (E). The symmetry property follows from the symmetry property for $d^\pi_Y$
and the definition of $d_Y$. The coarse equality property is just a
restatement of Proposition \ref{quasiequal} with our new notation. The
coarse triangle inequality, the inequality on triples and the
finiteness property all follow from the corresponding properties for
$d^\pi_Y$ plus coarse equality. Note that the inequality on triples
and the finiteness property hold for any $\xio \ge \xi$.  This will be
important in the proof of the order property.

(F). The monotonicity property requires a bit of work. We show that  for any
$\xio > 4\xi$ if
$Y\in\bY_\xio(X,Z)$ then $$\cH(X,Z)\subseteq \cH(X,Y)\cap\cH(Z,Y).$$ If $(X',Z') \in \cH(X,Z)$ then by Proposition
\ref{quasiequal} we have
$$\pd_Y(X,Z) -\pd_Y(X',Z') \leq 2\lc$$ and since $\pd_Y(X,Z) \geq
d_Y(X,Z)\geq \xio > 4\lc$ we have $\pd_Y(X',Z') > 2\lc$. In
particular $(X',Z')$ is in both $\cH(X,Y)$ and $\cH(Z,Y)$ and the
inequalities follow.
We have showed that the monotonicity holds for any constant 
$> 4\xi$.

(G). The proof of the order property is more involved.  Let $W,Y\in\bY_\xio(X,Z)$. 
Using the inequality on triples we choose $\theta'$ with $4\theta< \theta' \sim 0$
such that (for any $X,Y,Z$) $\min\{d_Y(X,Z), d_Z(X,Y) \} \le \theta'$.

To define the order we first establish that if $\xio$ is sufficiently large then the following are equivalent.
\begin{enumerate}[(a)]
\item $d_W(X,Y) > \theta'$;

\item $d_Y(X,W) \le \theta'$;

\item $d_Y(W,Z) > \theta'$;

\item $d_W(Y,Z) \le \theta'$.
\end{enumerate}
Both (a)$\Rightarrow$(b) and (c)$\Rightarrow$(d) follow from the inequality on triples. For (b)$\Rightarrow$(c) we apply the coarse
triangle inequality to see that
$$d_Y(X,W)  + d_Y(W,Z) \succ d_Y(X,Z) > \xio,$$
so if $d_Y(X,W)\le \theta'$ then $d_Y(W,Z) \succ \xio$. In particular if $\xio$ is sufficiently large then $d_Y(W,Z)> \theta' >4\theta$. By swapping $W$ and $Y$ this also shows that (d)$\Rightarrow$(a).

We now define $W<Y$ if any, and hence all, of (a) - (d) hold. Since either $d_W(X,Y) > \theta'$ or $d_W(X,Y) \le \theta'$ (but not both) we must have either $W<Y$ or $Y<W$ (but not both).
To finish the definition of the order we define $X$ to be the least element and $Z$ the greatest
element. We have just shown that any two elements can be compared and
that if $Y<W$ then $W \not< Y$. 

To argue transitivity, assume that $Y_0<Y_1<Y_2$. We  assume $Y_0 \neq X$ and $Y_2 \neq Z$ since if either is held then the rest of the proof is easier and we omit it. 
As noted at the end of its proof, the monotonicity holds
for any constant $>4\theta$, instead of $\xio$, in particular for $\theta'$.
Since $Y_1 < Y_2$ we have $d_{Y_1}(X,Y_2) > \theta'> 4\xi$ and therefore monotonicity
(with respect to $\theta'$) implies that
$$\theta' < d_{Y_0}(X, Y_1) \le d_{Y_0}(X, Y_2),$$ so $Y_0 < Y_2$ and transitivity holds.

We now prove the two inequalities ($\prec$ and $\le$). Since $Y_0 < Y_2$ and therefore $d_{Y_0}(X, Y_2)> 4\theta$ monotonicity (for $\theta'$) also implies that
$$d_{Y_1}(Y_0,Y_2) \le d_{Y_1}(X, Y_2).$$
Since $Y_2 \in \bY_\xio(X,Z)$ we also have that $d_{Y_2}(X,Z)> 4\theta$ (if $\xio \ge 4\theta$). Therefore, again,  monotonicity implies that
$$d_{Y_1}(X, Y_2) \le d_{Y_1}(X,Z)$$
and together these two inequalities give
$$d_{Y_1}(Y_0,Y_2) \le d_{Y_1}(X, Z).$$
By the coarse
triangle inequality we have
$$d_{Y_1}(X,Y_0) + d_{Y_1}(Y_0,Y_2) + d_{Y_1}(Y_2,Z) \succ d_{Y_1}(X,Z).$$
Since $Y_0<Y_1$ and $Y_1<Y_2$, we have $d_{Y_1}(X,Y_0)\leq\xi'$ and $d_{Y_1}(Y_2,Z)\leq\xi'$. It follows that
$$d_{Y_1}(Y_0,Y_2) \succ d_{Y_1}(X,Z).$$

Finally, to see the two claims with $\sim$,  we note that if $\xio$ is sufficiently large than the last coarse inequality implies that $d_{Y_1}(Y_0,Y_2)>\theta$ so the inequality on triples implies that
$$d_{Y_0}(Y_1,Y_2) \le \theta \mbox{ and } d_{Y_2}(Y_0,Y_1) \le \theta$$
which implies both are $\sim 0$.

(H). Finally we prove the barrier property. If the conclusion fails,
i.e. if $d_Z(X_0,X_1)\geq\xio$ then
$Z\in \bY_\xio(X_0,X_1)$ and also, by monotonicity, $Y\in
\bY_\xio(X_0,X_1)$. If $Y<Z$ in $\bY_{\xio}(X_0,X_1)$
then $d_Y(X_1,Z)\leq\xi$ and if $Z<Y$
then $d_Y(X_0,Z)\leq\xi$. Either way, we have a contradiction.
\end{proof}

\begin{remark}\label{nonmonotonic}
The monotonicity property fails for the original distance
$d^\pi$. Below is an example in the setting of geodesics in $\H^2$
(see Example \ref{ex}(1)).

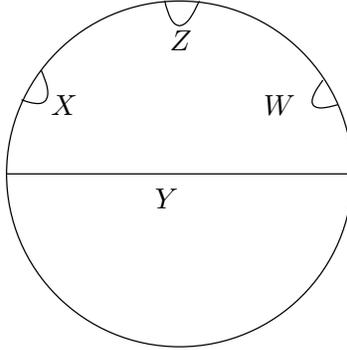
\centerline{\input{monotone.pstex_t}}

In the figure, $\pd_Y(X,Z)$ can be made arbitrarily large, while
$\pd_W(Y,Z)$ is slightly larger than $\pd_W(X,Z)$.
But if  $\pd_Y(X,Z) >2\xi$, then $(X,Z) \in \cH(Y,Z)$, 
therefore $d_W(Y,Z) \le d^{\pi}_W(X,Z)$.

One could define in the same way an order on
$\bY_K(X,Z)\cup\{X,Z\}$ for any $K\geq\xio$, but this order coincides
with the induced order from the larger set
$\bY_\xio(X,Z)\cup\{X,Z\}$. The order on $\bY_K(Z,X)\cup\{Z,X\}$ is
the reverse of the order on (the same set) $\bY_K(X,Z)\cup\{X,Z\}$.
\end{remark}

The following lemma is a consequence of the monotonicity property.
\begin{lemma}\label{chain}
There exists a $K>0$ with $K \prec \xio$ such that the following holds.
Let $\{Y_0, \dots, Y_n\}$ be vertices in $\bY$ such that $d_{Y_i}(Y_{i-1}, Y_{i+1}) > K$ for $i=1,\dots, n-1$. Then for each $i$, $d_{Y_i}(Y_{i-1}, Y_{i+1}) \le d_{Y_i}(Y_0,Y_n)$.
\end{lemma}

\begin{proof}
We will show that $d_{Y_i}(Y_{i-1}, Y_{i+1}) \le d_{Y_i}(Y_{i-1},
Y_{i+2})$. The lemma will then follow via an inductive argument. By
the inequality on triples $d_{Y_{i+1}}(Y_{i-1}, Y_i) \sim 0$. The
coarse triangle inequality implies $d_{Y_{i+1}}(Y_{i-1}, Y_{i+2})
\succ K$ so if $K$ is sufficiently large we have that
$d_{Y_{i+1}}(Y_{i-1}, Y_{i+2}) > \xio$. The monotonicity implies that
$d_{Y_i}(Y_{i-1}, Y_{i+1}) \le d_{Y_i}(Y_{i-1}, Y_{i+2})$.
\end{proof}

\subsection{The projection complex}
Unless otherwise said $\xio$ is the constant from Theorem \ref{main}.
For $K\geq \xio$  we now define the projection complex $\cS_K(\bY)$.
We always assume $K\geq \xio$.
\begin{definition}
The {\it projection complex} $\cS_K(\bY)$ is the following graph.  The
vertex set of $\cS_K(\bY)$ is $\bY$. Two distinct vertices $X$ and $Z$
are connected with an edge if $\bY_{K}(X,Z)$ is empty. Denote the
distance function for this graph by $d(,)$.
\end{definition}

In particular $d(X,Z) = 1$ if $\bY_{K}(X,Z) = \emptyset$. Note that
for different values of $K$ the spaces $\cS_K(\bY)$ are not
necessarily quasi-isometric to each other (the vertex sets are the
same, but for larger $K$ there are more edges, see Section \ref{tree
  example} for an explicit example). Our goal is to show that $\cS_K(\bY)$ is
quasi-isometric to a tree. We begin by showing that $\cS_K(\bY)$ is
connected and obtain an upper bound on the distance function.

\begin{prop}\label{connected}
If $X$ and $Z$ are vertices in $\bY$ then $d(X,Z) \le |\bY_K(X,Z)| +
1$. In particular, $\cS_K(\bY)$ is connected.
\end{prop}

\begin{proof}
Label the elements of $\bY_K(X,Z) \cup \{X,Z\}$ by $Y_0, Y_1, \dots,
Y_{k+1}$ where the indices respect the order and $k =
|\bY_K(X,Z)|$. We claim that $X=Y_0, Y_1, \dots, Y_{k+1} = Z$ is a
path from $X$ to $Z$. To see this we note that the monotonicity
property implies that if $Y \in \bY_K(Y_i, Y_{i+1})$ then $Y \in
\bY_K(X,Z)$ and $Y= Y_j$. However, since $Y_j$ cannot be between $Y_i$
and $Y_{i+1}$ we have $d_{Y_j}(Y_i, Y_{i+1}) < \xio$, a
contradiction. Therefore $\bY_K(Y_i,Y_{i+1}) = \emptyset$,
$d(Y_i,Y_{i+1}) = 1$ and we have our path from $X$ to $Z$.
\end{proof}

\subsection{Guards}
By contrast to Proposition \ref{connected}, the cardinality of
$\bY_{K}(X,Z)$ gives no lower bound on $d(X,Z)$. For example, it is
possible that $\bY_{K}(Y_1, Z) = \emptyset$ and therefore the distance from $X$
to $Z$ is two (even though $k$ is large). This highlights a key
difficulty in the paper. From the viewpoint of $X$, there appear to be
many projections larger than the $K$-threshold between $Y_1$ and
$Z$. However, from the viewpoint of $Y_1$ there are no large
projections between $Y_1$ and $Z$.
 
 A key concept in the paper is the notion of a {\em guard} and this
 notion is defined to deal with this problem. The notion 
 depends on the constant $K$. Roughly speaking, $W$ is
 a guard for $Y$ if from every viewpoint there are no large
 projections between $W$ and $Y$. 

\begin{definition}
$W\in\bY$ is a {\it guard} for $Y$ if for
 every vertex $X \in \bY$ with $W \in \bY_\xio(X,Y)$ and every $Z \in
 \bY_{K}(X,Y) \subset \bY_{\xio}(X,Y)$ then $Z \le W$. 
\end{definition}

Note that if
 $W$ is a guard for $Y$ then $d(W,Y)=1$.
 
\begin{lemma}\label{guard_constant}
For $K$ sufficiently large and vertices $X,Y,Z$ and $W$, if $W \in
\bY_{\xio}(X,Y)$, $Z \in \bY_{K}(X,Y)$ and $W< Z$ in $\bY_\xio(X,Y)$,
then $Z \in \bY_{K/2}(W,Y)$.

In particular, if  $\bY_{K/2}(W,Y) = \emptyset$ then $W$ is a guard for $Y$.
\end{lemma}

\begin{proof}
Given $X,Y,Z$ and $W$ as above, by the order property we have
$$d_{Z}(W,Y) \succ d_Z(X,Y) > K$$
and therefore if $K$ is sufficiently large then
$$d_{Z}(W,Y) >K/2.$$
\end{proof}
Note that it follows from this lemma and the order property that the
least element of $\bY_{K/2}(X,Z)$ (if nonempty) is a guard for $X$ and
the greatest element is a guard for $Z$.

By definition, if $X_0, X_1$ are vertices adjacent in $\cS_K(\bY)$ 
then the projection, $d_W(X_0, X_1)$, to another vertex $W$ will be bounded above by $K$. However, if $W$ is distance two or more from one of $X_0, X_1$, we get a stronger bound.
 \begin{lemma}\label{movebound}
Let $X_0$ and $X_1$ be adjacent vertices in $\cS_K(\bY)$  and
assume $W$ is a vertex in $\bY$ with $d(X_0, W) \ge 2$. 
Then 
$$d_W(X_0, X_1) \sim 0$$
and
$$d_W(X_0,Z) \sim d_W(X_1,Z)$$
for all $Z \in \bY$.
\end{lemma}

By our convention, the constants for the notation `` $\sim$'' in the statement
do not depend on $X_0, X_1, W, Z$, but only on $\theta$ (and $\xio$).

\begin{proof} Since $d(X_0,W) \ge 2$ there exists $Y \in \bY_K(X_0,W)$. If $d_W(X_0,X_1) > \xio$ then by monotonicity we have
$$d_Y(X_0,X_1) \ge d_Y(X_0,W) > K$$
which contradicts $d(X_0, X_1)  = 1$ and therefore $d_W(X_0, X_1) \le \xio$.

Applying the coarse triangle inequality we have
$$d_W(Z, X_0) + d_W(X_0, X_1) \succ d_W(Z,X_1)$$ which implies half of
the second inequality. The other half is proved by swapping $X_0$ and
$X_1$.
\end{proof}

\begin{remark}
The estimate $d_W(X_0, X_1) \sim 0$ in Lemma \ref{movebound} is the key place where we use monotonicity. In particular $d_W(X_0,X_1)$ is bounded by a constant that doesn't depend on $K$. Without monotonicity we would only have $d_W(X_0, X_1) \le K$ which is by definition true for any adjacent vertices $X_0$ and $X_1$ in $\cP_K(\bY)$. The other places where monotonicity is used, it is only for convenience to simplify the argument. The estimate in Lemma \ref{movebound} is essential for what follows. Remark \ref{non-tree} gives an example of what can go wrong.
\end{remark}

\begin{lemma}\label{nested_guards}
If $K$ is sufficiently large the following holds.  Let $X_0$ and $X_1$
be adjacent vertices with $d(X_i,Z) \ge 3$. Let $W$ be a guard for $Z$
such that $W \in \bY_{K/2}(X_0,Z)$. If $W \not\in \bY_{K/2}(X_1,Z)$
then there exists a guard $W'$ for $Z$ such that $W' \in \bY_{K/2}(X_1, Z)$
and $W \in \bY_{\xio}(W',Z)$.
\end{lemma}

\begin{figure}
\centerline{\scalebox{0.7}{\input{guards.pstex_t}}}
\caption{Lemma \ref{nested_guards}}
\end{figure}

\begin{proof}
We assume that $W \not\in \bY_{K/2}(X_1,Z)$. Note that $d(W,Z) =1$ and since $d(X_0,Z) \ge 3$ we have $d(X_0,W) \ge 2$ and we can apply Lemma \ref{movebound}. From Lemma \ref{movebound} we see that
$$d_W(X_1,Z) \succ d_W(X_0, Z) > K/2$$
and if $K$ is sufficiently large $W \in \bY_{\xio}(X_1,Z)$.

Since $d(X_1,Z)\ge 3$ we also have $d(X_1,W) \ge 2$ so there must be elements in $\bY_{K/2}(X_1,Z)$ that are less than $W$ in $\bY_{\xio}(X_1,Z)$. We let $W'$ be the greatest such element. By the order property
$$d_{W}(W',Z) \succ d_W(X_1,Z) \succ K/2$$
and again, if $K$ is sufficient large then $W \in \bY_{\xio}(W',Z)$.

We now show that $W'$ is a guard for $Z$. Note that for any $X$ with
$d_{W'}(X,Z)>\xio$ we also have $d_{W}(X,Z) > \xio$ by
monotonicity. If $V \in \bY_K(X,Z)$ then $V \le W$ 
in $\bY_{\xio}(X,Z)$ since $W$ is
a guard. If $W' < V$ then $V \in \bY_{K/2}(W',Z) \subseteq
\bY_{K/2}(X_1,Z)$ by Lemma \ref{guard_constant} and monotonicity and
therefore $V \neq W$ since $W \not\in \bY_{K/2}(X_1,Z)$. However, this
contradicts our choice of $W'$ as the greatest element of
$\bY_{K/2}(X_1,Z)$ that is less than $W$. So, $V \le W'$.
\end{proof}

\subsection{Barriers}

\begin{definition}
A {\em barrier} between a path $\{X_0, \dots, X_k\}$ and a vertex $Z$
is a vertex $Y$ such that $Y \in \bY_\xio(X_i, Z)$ for all $i=0,\dots,
k$. 
\end{definition}

By Theorem \ref{main} if there is a barrier between $\{X_0, \dots,
X_k\}$ and $Z$ then $d_Z(X_i,X_j) < \xio$ for all $i,j$.

\begin{prop}\label{key}
If $K$ is sufficiently large the following holds. 
Assume that $\{X_0,X_1,\cdots,X_k\}$ is a path in $\cS_K (\bY)$ and $Z$ a
vertex of $\cS_K (\bY)$ such that $d(Z,X_i)\geq 3$ for all $i$. Then
there is a barrier $W$ between the path and $Z$. In particular,
$d_Z(X_0,X_i)\sim 0$ for all $i$.
\end{prop}

\begin{proof}
We will inductively choose a family of guards $W_i$ for $Z$ such that
$W_i \in \bY_{K/2}(X_i, Z)$ and if $i>j$ then either $W_i = W_j$ or
$W_j \in \bY_{\xio}(W_i, Z)$.

We choose $W_0$ to be the greatest element of $\bY_{K/2}(X_0, Z)$, so
in particular $\bY_{K/2}(W_0,Z) = \emptyset$ by the order and
monotonicity properties. By Lemma \ref{guard_constant}, $W_0$ is a
guard for $Z$. Now assume that $W_0$ through $W_i$ have been
chosen. If $W_i \in \bY_{K/2}(X_{i+1},Z)$ then we let $W_{i+1} =
W_i$. If not, by Lemma \ref{nested_guards}, there exists a guard
$W_{i+1}$ in $\bY_{K/2}(X_{i+1}, Z)$ with $W_i \in \bY_{\xio}(W_{i+1},
Z)$. For any $j<i$, by the induction hypothesis, we have that $W_j \in
\bY_{\xio}(W_i, Z)$ and by monotonicity therefore $W_j \in
\bY_{\xio}(W_{i+1},Z)$.

Let $W= W_0$. Again applying monotonicity we have that 
$\bY_{\xio}(X_i, Z) \supseteq \bY_{\xio}(W_i,Z)$, 
therefore $W \in \bY_{\xio}(X_i, Z)$, so that $W$ is a
barrier between the path and $Z$ and that $d_Z(X_0,X_i) <
\xio$. \end{proof}

For geodesic paths we have the following corollary, analogous to
Masur-Minsky's Bounded Geodesic Image Theorem.

\begin{cor}\label{geodesic key}
If $K$ is sufficiently large the following holds. 
Assume that $\{X_0, X_1, \dots, X_k, Z\}$ is a geodesic path in
$\cP_K(\bY)$. Then $d_Z(X_0, X_i) \sim 0$ for all $i$.
\end{cor}

\begin{proof}
If $i \le k-2$ then this follows directly from Proposition \ref{key}. By Lemma \ref{movebound}, $d_Z(X_{k-2}, X_{k-1}) \sim 0$ and $d_Z(X_{k-1},X_k) \sim 0$ so the general statement then follows from the coarse triangle inequality.
\end{proof}

\subsection{$\cS_{K}(\bY)$ is a quasi-tree}
Recall \cite{manning} that a geodesic metric space $\cX$ satisfies the
{\it bottleneck property} if there is a constant $\Delta\geq 0$ such
that for any two points $x,z\in \cX$ there is a midpoint $y$
(i.e. $d(x,y)=d(y,z)=\frac 12d(x,z)$) such that any path from $x$ to
$z$ intersects the $\Delta$-neighborhood of $y$. Manning proved in
\cite{manning} that
$\cX$ is quasi-isometric to a simplicial tree (i.e. it is a {\it
  quasi-tree}) if and only if it satisfies the bottleneck property.

There is a slight reformulation of the bottleneck property that is
easier to deal with: {\it $\cX$ has the bottleneck property if and
  only if there is a constant $\Delta'$ such that for any two points
  $x,y\in\cX$ there is a path $p$ from $x$ to $y$ such that the
  $\Delta'$-neighborhood of any other path from $x$ to $y$ contains
  $p$.}

We prove this property implies the original property.
Let $g$ be a geodesic from $x$ to $y$, and $m$ be the mid point.
We claim that there is a point $m'$ in $p$ which is in the
$(2\Delta'+1)$-neighborhood of $m$.
To see this, let $p(i)$ be points  on $p$ from $x$ to $y$
with $d(p(i),p(i+1)) \le 1$. By the property, for each $i$, there is a
point $g(j_i)$ on $g$ with 
$d(p(i),g((j_i)) \le \Delta'$.
By triangle inequality, $d(g(j_i),g(j_{i+1})) \le 2\Delta'+1$.
Since $g$ is a geodesic, there must be $i$ such that 
$d(m,g(j_i)) \le 2\Delta'+1$. Set $m'=p(i)$.
Then, $d(m,m') \le 3\Delta'+1$

Now for any path $q$ from $x$ to $y$, there must be a point $m''$ in
$q$ such that $d(m',m'')$ is at most $\Delta'$.
So, $d(m,m'')$ is at most $4\Delta'+1$.

If the space is a graph, we only need to consider vertices
rather than all points in the conditions and arguments. 

We can now prove:

\begin{thm}\label{quasitree}
For $K$ sufficiently large $\cS_K({\bf
Y})$ is a quasi-tree.
Moreover, the quasi-isometry constant to a tree is 
uniform. 
\end{thm}

\begin{proof}
We will verify the modified bottleneck property with $\Delta'=2$.
This also implies a uniform bound on the quasi-isometry constant
\cite[Section 4]{manning}.  Let
$X,Z$ be two vertices of $\cS_K(\bY)$. The ordered set $\bY_K(X,Z)$ is
a path from $X$ to $Z$ (see the proof of Proposition \ref{connected}).  We
now check that any path $X=X_0,X_1,\cdots,X_k=Z$ from $X$ to $Z$
passes within 2 of any vertex $Y$ in $\bY_K(X,Z)$. If not, then by
Proposition \ref{key} we have $d_Y(X,Z)<\xio$
contradicting the fact that $Y\in\bY_K(X,Z)$.
\end{proof}

We could also use the original distance $d^\pi$ to define a projection
complex $\cS^\pi_K(\bY)$. However it is not a quasi-tree in
general. We sketch the construction below.

\begin{example}\label{non-tree}
(1) For any large integer $K>0$, 
$\cP^\pi_K(\bY)$ can be an arbitrarily large loop and is hence not a quasi-tree 
with a  quasi-isometry constant bounded or even a $\delta$-hyperbolic space
with $\delta$ bounded.

Suppose $\bY = \{Y_0, \dots, Y_n\}$ is finite and each $Y_i$ is a copy
of $\R$.  Fix $K>10$. For $0<i<j<n$ we define $\pi_{Y_j}(Y_i) =
\{-1\}$ and $\pi_{Y_i}(Y_j) = \{K+2\}$. For $i>0$ we define
$\pi_{Y_i}(Y_0) = \{1\}$ and $\pi_{Y_0}(Y_i) = \{K\}$. For $i< n$ we
define $\pi_{Y_i}(Y_n) = \{K\}$ and $\pi_{Y_n}(Y_i)=\{1\}$. See Figure
\ref{figure.example}. We then define $d^\pi_{Y_j}(Y_i, Y_k)$ in the
usual way. It is then straightforward to check that the axioms hold
with $\xi =3$. Furthermore, $d_{Y_j}(Y_i, Y_{i+1}) \le 2 < K$ for all
$j$ and we also have $d_{Y_j}(Y_0, Y_n)=K-1 <K$ for all $j$.
Therefore, there are edges between $Y_i$ and $Y_{i+1}$ and between
$Y_0$ and $Y_n$ in $\cS^\pi_K(\bY)$.

On the other hand, if $i<j<k$ and $i \neq 0$ or $k \neq n$ then
$d^\pi_{Y_j}(Y_i, Y_k)\ge K+1$ so there can be no other edges and
$\cP^\pi_K(\bY)$ is a loop of length $n+1$. We leave it as an exercise
to show that $\cP_K(\bY)$ is a complete graph on $n+1$ vertices
($(Y_0,Y_n) \in {\cal H}(Y_i,Y_j)$ if $0<i<j<n$, so $d_{Y_k}(Y_i,Y_j)
\le d^\pi_{Y_k}(Y_0,Y_n) =K-1$ if $0<i<k<j<n$, hence there is an edge
between $Y_i,Y_j$ in $\cS_K(\bY)$).

\begin{figure}
\centerline{\scalebox{0.6}{\input{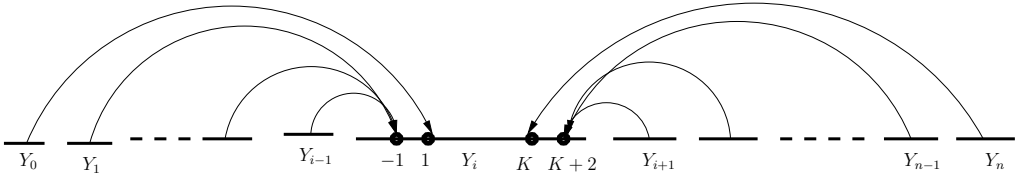}}}
\caption{Arrows indicate projections $\pi_{Y_i}(Y_j)$}
\label{figure.example}
\end{figure}

(2) Now we want to produce for a given $K>10$ an example with $\theta
=3$ such that $\cP^\pi_K(\bY)$ contains arbitrarily large
isometrically embedded loops, so it is not a quasi-tree or even a
hyperbolic space.  For simplicity, we only give an example such that
$\cP^\pi_K(\bY)$ contains a loop of length $n$ and a loop of length
$m$.  The idea is to use the examples from (1) for $n$ and $m$ and
arrange the projections such that $\cP^\pi_K(\bY)$ is a bouquet of
loops of length $n$ and $m$.

Suppose $\bY = \{Y_0, \dots, Y_n, \dots, Y_{n+m}\}$ and each $Y_i$ is $\R$.
Define 
\begin{itemize}
\item
For $i<j$ such that $\{i,j\}\cap \{0,n,n+m\}=\emptyset$
let $\pi_{Y_i}(Y_j)=\{K+2\}, \pi_{Y_j}(Y_i)=\{-1\}$.
\item
$\pi_{Y_0}(Y_i)=\{K\}$ and 
$\pi_{Y_i}(Y_0)=\{1\}$ for all $0<i$.
\item
$\pi_{Y_n}(Y_i)=\{1\}$ for all $i<n$ and
$\pi_{Y_n}(Y_i)=\{K\}$ for all $n<i$.
\item
$\pi_{Y_i}(Y_n)=\{K\}$ for all $i<n$ and
$\pi_{Y_n}(Y_i)=\{1\}$ for all $n<i$.
\item
$\pi_{Y_{n+m}}(Y_i)=\{1\}$ and 
$\pi_{Y_i}(Y_{n+m})=\{K\}$ for all $0<i$.
\end{itemize}

Again, the axioms holds for $\theta=3$. As in (1), 
 the vertices $Y_0, \cdots, Y_n$ form
a loop of length $n$ and the vertices $Y_{n}, \cdots, Y_{n+m}$
form a loop of length $m$ in $\cP^\pi_K(\bY)$, which consists 
of the two loops. On the other hand, $\cP_K(\bY)$ is a complete graph.
Note that $\cP_L(\bY)$ is a line of length $n+m$ if $ 10< L <K-1$.
See Figure \ref{figure.loop}.

\begin{figure}
\centerline{\scalebox{0.6}{\input{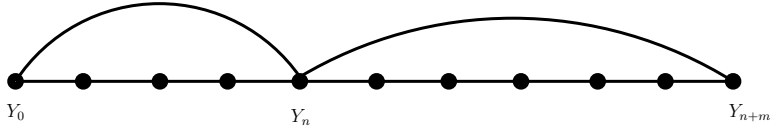}}}
\caption{$\cP^\pi_K(\bY)$ has two loops. $\cP_K(\bY)$ is a complete graph with the 
same vertex set. $\cP_L(\bY)$ and $\cP^{\pi}_L(\bY)$ are the lines of length $n+m$  without the two upper edges if $L<K-1$.}
\label{figure.loop}
\end{figure}

Similarly, we can produce an example such that 
$\cP^\pi_K(\bY)$ contains an isometrically embedded 
loop of length $n$ for all $n>0$, which is not a quasi-tree, 
while $\cP_K(\bY)$ is a complete graph. Moreover, 
$\cP_L(\bY)$ is an infinite line if $ 10< L < K-1$.

(3) Building on the examples in (2), we can produce an example such that 
$\cP^\pi_K(\bY)$ is not a quasi-tree for any $K$.
The idea is that for each large positive 
integer $K$ we first produce $\bY^K$ and projections as we did in (2) 
such that  $\cP^\pi_K(\bY^K)$ contains arbitrarily long loops. 
Next we put $\bY^K$ together for all $K$ and obtain $\bY$, then define 
projections between elements in $\bY^K$ with different $K$'s
as we did for $n$ and $m$ in (2).
 Then for each large positive
integer $L$, the resulting graph $\cP^\pi_L(\bY)$ contains
the graph $\cP^\pi_L(\bY^L)$ as a subgraph, therefore has 
arbitrarily large embedded loops.
On the other hand the quasi-tree $\cP_L(\bY)$ is  unbounded 
since it contains $\cP_L(\bY^K)$ for all $K$ but 
each of them is a geodesic line for $L<K-1$.
We leave the details to the reader.
\end{example}

\begin{lemma}\label{masur-minsky0}
There exists a $K'>0$ such that if $Y \in \bY_{K'}(X,Z)$ then every geodesic from $X$ to $Z$ in $\cS_K(\bY)$ contains $Y$.
In particular
$$d(X,Z) \ge |\bY_{K'}(X,Z)| + 1.$$
\end{lemma}

\begin{proof}
Let $X=X_0, X_1, \dots, X_k = Z$ be a geodesic from $X$ to $Z$ that doesn't contain $Y$. We will show that $d_Y(X,Z) \prec 5K$.

If $d(X_i, Y) \ge 3$ for all $i$ then by Proposition \ref{key} we have
$d_Y(X,Z) \sim 0$. Now assume that $d(X_i, Y) < 3$ for some $i$. Let
$i^-$ be the first time that $d(X_{i^-}, Y)< 3$ and $i^+$ the last
time that $d(X_{i^+},Y) <3$. Then $i^+ - i^- \le 4$ since $d(X_{i^-},
X_{i^+}) \le 4$. For convenience we will assume $i^->0$ and $i^+<k$;
an obvious modification of the argument works when this is not the
case.  Again applying Proposition \ref{key} we have that $d_Y(X,X_{i^-
  - 1}) \sim 0$ and $d_Y(X_{i^+ + 1},Z) \sim 0$.

Since the path doesn't contain $Y$ then for all $X_i$ we have
$d_Y(X_i, X_{i+1}) \leq K$. Using this estimate and the coarse
triangle inequality six times we have $$d_Y(X_{i^- - 1}, X_{i^+ + 1})
\prec 5K.$$ Combining with our bounds on $d_Y(X,X_{i^- - 1})$ and
$d_Y(X, X_{i^+ + 1})$ and applying the coarse triangle inequality two
more times we have $d_Y(X,Z) \prec 5K$. Therefore there exists a $K'$
with $K' \sim 5K$ such that if $Y \in \bY_{K'}(X,Z)$ then every
geodesic from $X$ to $Z$ contains $Y$. This implies the lemma.
\end{proof}

The next corollary is in preparation for studying axes of isometries
on the projection complex.

Let $\alpha$ be a biinfinite geodesic in $\cP_K(\bY)$ and define for a constant $L>0$
$$\bY_L(\alpha) = \{Y \in \bY | \exists X,Z \in \alpha \mbox{ such
  that } d_Y(X,Z) > L\}$$ and the {\em stable part} of $\alpha$ by
$$\bY(\alpha) = \{Y \in \bY | \mbox{ if a biinfinite geodesic } \beta \mbox{ is parallel to
} \alpha \mbox{ then } Y \in \beta\} \subset \alpha.$$
In other words, $\bY(\alpha)$ is the intersection of all biinfinite geodesics
parallel to $\alpha$ (including $\alpha$).
 Clearly, if $\alpha$ is
parallel to $\beta$ then $\bY(\alpha)=\bY(\beta)$.

\begin{cor}\label{big projection geodesic}
Let $K'$ be the constant from Lemma
  \ref{masur-minsky0}.
\begin{enumerate}[(i)]
\item $\bY_{K'}(\alpha) \subset \alpha$.

\item There exists a $K''\ge K' >0$ such that the following holds. Let
$D>0$ be a constant, and 
  $X_0, X_1$  vertices in $\cP_K(\bY)$ and $Y_0, Y_1$ vertices in
  $\alpha$ such that $d(X_i, Y_i) < D$. If $Z \in \bY_{K''}(\alpha)$
  lies between $Y_0$ and $Y_1$ on $\alpha$ and $d(X_i,Z) > 2D+2$ then $Z \in
  \beta$ for any geodesic $\beta$ between $X_0$ and $X_1$.

\item $\bY_{K''}(\beta) \subset \bY(\alpha)$ for any geodesic $\beta$ that is parallel to $\alpha$.
\end{enumerate}
\end{cor}

\begin{proof}
(i) follows directly from Lemma \ref{masur-minsky0}. 

For (ii) we note that there is a (geodesic) path from $Y_i$ to $X_i$ such that
every vertex in the path has distance at least 3 from $Z$ so by
Proposition \ref{key}, $d_Z(X_i, Y_i) \sim 0$. Since $Z \in
\bY_{K''}(\alpha)$ there exists $Y'_0, Y'_1 \in \alpha$ with
$d_Z(Y'_0, Y'_1) > K''$. Since $Z$ lies between $Y_0$ and $Y_1$ we can
assume that $Y_0$ and $Y'_0$ are on the same side of $Z$ (and
similarly for $Y_1$ and $Y'_1$). Therefore there is a geodesic path
from $Y_i$ to $Y'_i$ that is disjoint from $Z$ and so by Corollary
\ref{geodesic key}, $d_Z(Y_i,Y'_i) \sim 0$. Applying the coarse
triangle inequality we have that $d_Z(X_i, Y'_i) \sim 0$ and $d_Z(X_0,
X_1) \succ K''$. Therefore if $K''$ is sufficiently large $d_Z(X_0,
X_1)>K'$ and by Lemma \ref{masur-minsky0} we have that $Z$ lies in
every geodesic between $X_0$ and $X_1$.

Assume that $Z \in \bY_{K''}(\beta)$ and let $\gamma$ be parallel to
$\alpha$ (and hence $\beta$). The geodesic $\gamma$ is contained in
the $D$-Hausdorff neighborhood of $\beta$ for some $D>0$. Therefore we can find
vertices $X_0, X_1 \in \gamma$ and vertices $Y_0, Y_1 \in \beta$ such
that $d(X_i, Y_i) \le D$, $d(X_i, Z)> 2D+2$ and $Z$ lies between $Y_0$
and $Y_1$. Then by (ii), $Z \in \gamma$ and (iii) follows.
\end{proof}

Finally, we establish that the projection complex has infinite
diameter under mild conditions.

\begin{prop}[$\cS_K(\bY)$ is unbounded]
Suppose that for every $R>0$ and $A\in\bY$ there exist $B,C\in\bY$
such that $d_A^\pi(B,C)>R$. Then the diameter of $\cS_K(\bY)$ is
infinite.
\end{prop}

\begin{proof}
Let $K'$ be the constant from Lemma
  \ref{masur-minsky0}.
Choose $A_0,A_1,A_2\in\bY$ such that
  $d_{A_1}(A_0,A_2)>K'$. Applying the assumption to $A_2$, find $B,C$
  so that $d_{A_2}(B,C) \succ 3K'$. It follows from the coarse triangle
  inequality that for either $A_3=B$
    or $A_3=C$ we have $d_{A_2}(A_1,A_3) > K'$. Continuing in the same
    fashion (by induction), we can extend the sequence $A_i$ forever with
    $d_{A_i}(A_{i-1},A_{i+1})>K'$. 
By Lemma \ref{chain}, for each $0<j<i$ we have $d_{A_j}(A_0,A_i) >K'$.
Thus by Lemma \ref{masur-minsky0}
    $d_{\cS_K(\bY)}(A_0,A_i)\geq i$.
\end{proof}

\subsection{Group action on the projection complex}\label{s:WPD}

Now assume that $G$ is a group that acts on the set $\bY$ in such a
way that
projection distances are $G$-equivariant, 
i.e. $d^\pi_{g(A)}(g(B),g(C))=d^\pi_A(B,C)$ for all $A,B,C\in\bY$ and
$g\in G$. Then $G$ acts naturally on the projection complex
$\cS_K(\bY)$ by automorphisms.

The following proposition is clear since $\cS_K(\bY)$ is connected. 

\begin{prop}
Suppose the action of $G$ on $\bY$ has finitely many orbits. Then the
action of $G$ on $\cS_K(\bY)$ is cobounded (i.e. a Hausdorff
neighborhood of an orbit is the whole space).
\end{prop}

Next, we construct axes of (powers of) elements with unbounded orbits. Note that axes are geodesics by definition.  

Let $K'$ be the constant from Lemma
  \ref{masur-minsky0} and $K''$ the constant from Corollary \ref{big
    projection geodesic}.

\begin{lemma}[Axial isometry]\label{translation}
Suppose $g\in G$ and
$Y\in\bY$ such that
$$d_Y(g^{-N}(Y),g^N(Y))>K'$$ for some $N>0$. Then $g^N$ has an axis $\alpha$ that contains $g^{kN}(Y)$ for all $k \in \Z$. In particular, $g$ acts on
$\cS_K(\bY)$ with positive translation length. Furthermore if
$$d_Y(g^{-N}(Y),g^N(Y))>K''$$
then $g$ has an axis that contains all $g$-translates of $Y$.
\end{lemma}

\begin{proof}
By the $G$-equivariance of the projection distance, 
Lemma \ref{chain} applies to $\{Y, g^N(Y), \cdots, g^{kN}(Y)\}$, so that 
$d_{g^{iN}(Y)}(Y,g^{kN}(Y))>K'$ for $0<i<k$.  Now, Lemma
\ref{masur-minsky0} gives
$$d(Y,g^{kN}(Y))= kd(Y, g^N(Y))$$
which implies that the translation length $$\tau(g)=\lim_{k\to\infty}
\frac{d(Y,g^{kN}(Y))}{kN}= \frac{kd(Y, g^N(Y))}{kN}\geq \frac{1}N>0.$$
To construct $\alpha$ take a geodesic segment between $Y$ and $g^N(Y)$ and translate it by the action of $g^N$ to get a bi-infinite path.

Now assume $d_Y(g^{-N}(Y),g^N(Y))>K''$. For all $k \in \Z$,
$g^k(\alpha)$ will be parallel to $\alpha$ and $g^k(Y) \in
\bY_{K''}(g^k(\alpha))$. By (iii) of Corollary \ref{big projection
  geodesic}, $g^k(Y) \in \bY(\alpha)$. In particular, $g^k(Y) \in
\alpha$. By replacing the geodesic segment in $\alpha$ from
$g^k(\alpha)$ to $g^{k+1}(\alpha)$ with the $g^k$-translate of the
geodesic segment in $\alpha$ from $Y$ to $g(Y)$ we obtain a
$g$-invariant geodesic.
\end{proof}

Using the same idea but a bit more work we can find a copy of $F_2$ in $G$ that acts on an embedded tree in $\cP_K(\bY)$ such that 
any non-trivial element in $F_2$ has an axis in the tree. 

\begin{prop}[Free subgroup of axial elements]\label{free group}
Fix $Z_1,Z_2 \in \bY$ and $g_1,g_2 \in G$ and then define $Z^{k}_j = g_j^k(Z_j)$ with $k = \pm 1$. Assume that for a constant $L$ and 
all permutations of $i,j \in \{1,2\}$ with $i\not =j$ and $k \in \{-1,1\}$ we have 
\begin{itemize}
\item $d_{Z_j}(Z^{-1}_j, Z^1_j) > L \ge K'$;

\item $d_{Z_j}(Z^k_j, Z_i) > L \ge K'$.
\end{itemize}
Then:
\begin{enumerate}
\item $F = \langle g_1, g_2 \rangle$ is a non-abelian free group.

\item There exists a trivalent $F$-invariant tree $S$ isometrically embedded in $\cP_K(\bY)$ and the $F$-action is proper and minimal.

\item For every non-trivial $\phi \in F$ there is vertex $W \in S$
  such that 
$$d_W(\phi^{-1}(W), \phi(W))>L.$$
Further, $\phi$ has an axis contained in $S$.
\end{enumerate}
\end{prop}

\begin{proof}
Let $\tilde{F} = \langle a,b \rangle$ be the free group of words in
$a$ and $b$.  Then $\tilde F$ is the fundamental group of the barbell
group and therefore acts on its universal cover, a trivalent tree
$\tilde S$. We can assume that the axes $\alpha$ and $\beta$ of $a$
and $b$ in $\tilde S$ are disjoint but connected by a single edge
whose endpoints are vertices $\tilde Z_1 \in \alpha$ and $\tilde Z_2
\in \beta$. If each edge of $\tilde S$ has length one we can also
assume that the translations of $a$ and $b$ are both one.  Define a
homomorphism from $\tilde{F}$ to $F = \langle g_1, g_2 \rangle \subset
G$ by $a \mapsto g_1$ and $b \mapsto g_2$. We then choose a map $\psi:
\tilde S \to \cP_K(\bY)$ that is equivariant with respect to this
homomorphism and with $\psi(\tilde Z_i) = Z_i$. By scaling the length
of the edges of $\tilde S$ we can make this map an isometry on each
edge. We will show that it is in fact a global isometry and the
conclusions of the proposition will follow.

Note that there are exactly two $\tilde F$-orbits of vertices in
$\tilde S$ with one containing $\tilde Z_1$ and the other containing
$\tilde Z_2$. Therefore if $\tilde Y_0$, $\tilde Y_1$ and $\tilde Y_2$
are consecutive vertices in $\tilde S$ then there exists a (unique) $w
\in \tilde F$ such that $w(\tilde Y_1) = \tilde Z_1$ or $w(\tilde Y_1)
= \tilde Z_2$. Assume it is the former. Then $w(\tilde Y_0)$ and
$w(\tilde Y_2)$ will be distinct elements in the set $\{a(\tilde Z_1),
a^{-1}(\tilde Z_1), \tilde Z_2\}$. The $\psi$-image of this set is
$\{Z_1^1, Z^{-1}_1, Z_2\}$ so for all possibilities we have that
$$d_{Y_1}(Y_0, Y_2) = d_{\psi(w(\tilde Y_1))}(\psi(w(\tilde Y_0)),\psi(w(\tilde Y_2)))>L$$
where $Y_i = \psi(\tilde Y_i)$ by our assumption. 
The latter case is similar. 

By Lemma \ref{chain} it follows that for a chain of consecutive
vertices $\tilde Y_0, \dots, \tilde Y_n$ in $\tilde S$ we have
$d_{Y_j}(Y_0, Y_n) > L$ for all $j = 1,\dots,n-1$ again with $Y_i =
\psi(\tilde Y_i)$. By Lemma \ref{masur-minsky0}, $Y_j$ is contained in
every geodesic from $Y_0$ to $Y_n$ and it follows that $\psi$ is a
global isometry. This implies (1) and (2).

For (3) we note that every $\phi \in F$ is the image of some $\tilde
\phi$ in $\tilde F$. The $\psi$-image of the axis of $\tilde \phi$ in
$\tilde S$ will be an axis for $\phi$. If $\tilde W$ is contained in
this axis then it is contained in a consecutive chain of vertices from
$\tilde \phi^{-1}(\tilde W)$ to $\tilde \phi(\tilde W)$ and therefore,
by the previous argument, $d_W(\phi^{-1}(W), \phi(W))>L$ where $W =
\psi(\tilde W)$.\end{proof}

\begin{example}[Cayley tree]
In Section \ref{tree example} we discussed
the Cayley tree  of a free group. In that example, we can
take the axis of $a$ as $Y$, $g=a^nb$  in Lemma \ref{translation} where $n>>K$; 
and in Proposition \ref{free group} take $g_1=a^{2n}b$ with $Z_1$ the axis of $a$ and again $n>>K$. We then conjugate $g_1$ with $a^n b a^{-n}$ to get $g_2$  and let $Z_2 = a^nba^{-n}(Z_1)$. (For a more general strategy for choosing $g_1$ and $g_2$ see the proof of Corollary \ref{bushy}).
As we will see any non-trivial element in $F$  is WPD
by Proposition \ref{wpd} since the common 
stabilizer of a pair of distinct points in $\bY$ is trivial (see Remark \ref{WPD free}).
\end{example}

We state a corollary (of Theorem \ref{quasitree} and Proposition
\ref{free group}). We have to verify the axioms (P0), (P1) and (P2)
and in this general setting it will be done in \cite{bc} (i.e., when
we prove Theorem H).  In particular, it will apply to all groups that
are listed in Example \ref{ex}.  In this paper we have verified the
axioms for discrete subgroups of isometries of $\H^n$ and it is
straightforward to generalize this to hyperbolic groups. Even in these
cases the result is new.

Recall that an action on a quasi-tree is {\it non-elementary} if the
orbits are unbounded, and there is no fixed end, nor a pair of ends.

\begin{cor}\label{bushy}
Let $\Gamma$ be a group which acts on a geodesic metric space $X$ with a WPD element with respect to the action that has a $B$-contracting orbit.

 If $G$ is not virtually cyclic then it has a non-elementary cobounded action on a quasi-tree.
\end{cor}

\begin{proof}
Let $a \in G$ be a (hyperbolic) element
that is WPD with an axis $\alpha$ (if there is no geodesic axis,
take an invariant quasi-geodesic, or the orbit of an point). 
Let $\bY$ be the collection of 
parallel classes of $G$-translates of $\alpha$.
As we said, under the assumption, the axioms (P0)-(P2) are 
satisfied, \cite{bc}.
We apply our construction to $\bY$ and obtain a $G$-quasi-tree
$\cP_K(\bY)$ by Theorem \ref{quasitree}.

 To see that the action 
 is non-elementary we need to find $g_1, g_2 \in \Gamma$ for which
we can apply Proposition \ref{free group}. 
Let $Y \in \bY$ be the equivalence class of $\alpha$.
$\bY$ contains an element $Z \not=Y$ (since otherwise,
every $G$-translate of $\alpha$ is parallel to $\alpha$, but 
then $G$ is virtually cyclic by WPD) such that  $h(Y)=Z$ for some 
$h \in G$. Set $Z_1 = Y$ and $Z_2 = Z$. 
For an $n$ to be determined shortly we also set $Z_1^{\pm 1} = a^{\pm n}h^{\mp 1}(Z_1)$ and $Z^{\pm 1}_2 = h(Z^{\pm 1}_1)$. 
Note that for any $X_0,X_1 \neq Y$ both $d_Y(X_0, a^{\pm n}(X_1))$ and $d_Y(a^{-n}(X_0), a^n(X_1))$ 
grow linearly in $n$ and it follows that given $L>0$, for $n$ sufficiently large 
$Z_1, Z_2$ and $Z^{\pm 1}_1, Z^{\pm 1}_2$ satisfy the assumption of Proposition \ref{free group}. 
For example $d_{Z_1}(Z_2, Z_1^1) = d_{Z_1}(Z_2, a^n(h^{-1}Z_1))$ is $>L$ for large $n$. 
We next set $g_1 = a^n h^{-1} a^n$ and $g_2 = hg_1 h^{-1}$ and check 
that $g_i^j(Z_i) =Z^j_i$ for $i=1,2$ and $j = \pm 1$ so that 
$g_1$ and $g_2$ also satisfy the assumption of Proposition \ref{free group}. 
By (1) and (2) of Proposition \ref{free group} we then have that 
$\langle g_1, g_2 \rangle$ is free and acts isometrically on a isometrically embedded tree
 in $\cP_K(\bY)$ so the action of $G$ is non-elementary. \end{proof}

In the rest of this section we study the WPD property (as defined in Section \ref{wpd-defn}) of the action on
the projection complex. We start by constructing a total order on the
combinatorial axis of an element.

If $g\in G$ is hyperbolic and has an axis $\alpha$ define the {\it
  combinatorial axis} as $\bY(g) = \bY(\alpha)$. This does
  not depend on the choice of $\alpha$.
$\bY(g)$ is  possibly empty. We recall 
 the {\em elementary closure}, $EC(g)$, of $g$ is the subgroup of elements $h \in G$ such
that $h(\alpha)$ is parallel to $\alpha$.

\begin{prop}[Combinatorial axis and elementary closure]\label{axis}
Assume $\bY(g)$ is not empty.  Then there is a total order on $\bY(g)$ such that
\begin{enumerate}
\item [(i)] $\bY(g)$ is $EC(g)$-invariant and the $EC(g)$-action preserves
  the order up to sign.
\item [(ii)] The order is unique if we require $g(Y)>Y$ for some (every) 
  $Y\in \bY(g)$.

\item [(iii)] $\bY(g)$ is order-isomorphic to $\Z$ and $EC(g)$
  acts as isometries of $\Z$ under this isomorphism.

\item [(iv)] Assume that $Y_0,Y_1, Y_2, Y'_0, Y'_2 \in \bY(g)$ with $Y_1$
  between both the pair $Y_0$ and $Y_2$ and the pair $Y'_0$ and
  $Y'_2$. Then $d_{Y_1}(Y_0,Y_2) \sim d_{Y_1}(Y'_0,Y'_2)$.

\end{enumerate}
\end{prop}

\begin{proof}
Let $\alpha$ be an oriented axis for $g$ so that $g$ is a positive
translation with respect to the orientation. The vertices of $\alpha$
have an order which induces a total order on $\bY(\alpha)$.

Let $\beta$ be parallel to $\alpha$ and $h \in EC(g)$. If $h(Y)
\not\in \beta$ then $Y \not\in h^{-1}(\beta)$. Since $h^{-1}(\beta)$
will also be parallel to $\alpha$ we have that if $Y \in \bY(g)$ then
$g(Y) \in \bY(g)$, proving (i).

For (ii) we note that by our choice of orientation for $Y \in
\alpha$, $g(Y)$ appears after $Y$.

The vertices in $\bY(g)$ are a discrete set in $\alpha$ so the order
coming from $\alpha$ will be order isomorphic to $\Z$ and we can
accordingly label them $Y_n$. In particular if $k<n<m$ then $Y_n$ is
between $Y_k$ and $Y_m$ on $\alpha$. If $h \in EC(g)$ then the same
must be be true on $h(\alpha)$ for otherwise we could build a geodesic
parallel to $\alpha$ that did not contain $Y_n$ by replacing the
geodesic segment from $Y_k$ to $Y_m$ on $\alpha$ with the segment with
the same endpoints on $h(\alpha)$. If $h_*$ is the induced map on $\Z$
then this implies that $|h_*(n) - h_*(m)| = |n-m|$ proving (iii).

For (iv) we can assume that both $Y_0$ and $Y'_0$ are less than $Y_1$
in the total order. Then $d_{Y_1}(Y_0, Y'_0) \sim 0$ by Corollary
\ref{geodesic key} and similarly $d_{Y_1}(Y_2, Y'_2) \sim 0$. The
coarse triangle inequality then implies (v).
\end{proof}

The following provides a sufficient condition for an element to be
hyperbolic and WPD. 

\begin{prop}[Axial and WPD]\label{wpd}
Assume that $g \in G$ satisfies
\begin{enumerate}[(i)]
\item there exists a vertex $Y$ and an $N>0$ such that $d_Y(g^{-N}(Y), g^N(Y))> K''$;
\item there exists an $m>0$ such that the subgroup of $G$ that fixes $$Y, g(Y), \dots, g^m(Y)$$ is finite.
\end{enumerate}
Then $g$ has an axis and the action of $g$ on $\cP_K(\bY)$ is WPD.
\end{prop}

\begin{proof}
By Lemma \ref{translation} $g$ is hyperbolic and has an axis. 
Fix a $D>0$ which for simplicity we'll assume is an integer and let $M= 7D+7+m$. We need two claims:
\begin{enumerate}[(a)]
\item If $d(Y, \phi(Y)) \le D$ and $d(g^M(Y), \phi(g^M(Y))) \le D$ then the commutator $[\phi,g]$ lies in a finite set of elements.

\item There are only finitely many $\psi \in G$ with $$d(Y, \psi(Y))
  \le D\mbox{ and }d(g^M(Y), \psi(g^M(Y))) \le D$$ and with $[\phi,g]
  = [\psi,g]$.
\end{enumerate}
These two claims imply that the set
$$\{\phi \in G| d(Y, \phi(Y)) \le D \mbox{ and } d(g^M(Y), \phi(g^M(Y)) \le D\}$$
is finite and $g$ is WPD.

We now prove (a). By (ii) of Corollary \ref{big projection geodesic},
$g^{3D+3}(Y), \dots g^{4D+4+m}(Y)$ will be in every geodesic from
$\phi(Y)$ to $\phi(g^M(Y))$. We also note that $$\bY(\phi g \phi^{-1})
= \phi(\bY(g))$$ so $\phi(Y), \phi(g^M(Y)) \in \bY(\phi g \phi^{-1})$
and therefore $g^i(Y) \in \bY(\phi g \phi^{-1})$ for $i=3D+3, \dots,
4D+4+m$. Furthermore the order (up to sign) that the $g^i(Y)$ appear
in $\bY(g)$ must be the same as their order in $\bY(\phi g \phi^{-1})$
and in particular $\phi g \phi^{-1} (g^i(Y)) = g^{i \pm 1}(Y)$ for $i=
3D+4, \dots, 4D+3+m$ since $g$ and $\phi g \phi^{-1}$ have the same
translation length. We first need to show that $\phi g \phi^{-1}
(g^i(Y)) = g^{i + 1}(Y)$ instead of $g^{i-1}(Y)$.

Assume not and that $\phi g \phi^{-1} (g^i(Y)) = g^{i - 1}(Y)$. Then
$\phi$ reverses the order of the $g^i(Y)$ in $\bY(\phi g \phi^{-1})$
and in particular, $g^{4D+3+m}(Y)$ occurs before $g^{3D+4}(Y)$. Since
$d(\phi(Y), \phi(g^M(Y))) = M \tau(g)$ and
$d(g^{4D+4+m}(Y),g^{3D+3}(Y)) = (D+m+1) \tau(g)$ one of $d(\phi(Y),
g^{4D+3+m}(Y))$ or $d(g^{3D+4}(Y), \phi(g^M(Y)))$ must be no greater
than $(M-(D+m+1))\tau(g)/2 = (3D+3)\tau(g)$. Assume it is the
former. The proof is similar in the latter case. Since $\tau(g) \ge 1$
and $d(Y,\phi(Y))\le D$ the triangle inequality implies that $d(Y,
g^{4D+4+m}(Y)) \le D + (3D+3)\tau(g) \le (4D +3)\tau(g)$. On the other
hand $d(Y, g^{4D+4+m}(Y)) = (4D + 4 +m) \tau(g) > (4D +3)\tau(g)$,
contradiction.

Therefore $\phi g \phi^{-1} (g^i(Y)) = g^{i +1}(Y)$ and
$[\phi,g](g^{i+1}(Y)) = \phi g \phi^{-1} (g^i(Y)) = g^{i+1}(Y)$ for $i
= 3D+4, \dots, 4D+2+m$. Now notice that the subgroup that fixes $Y, g(Y), \dots,
g^m(Y)$ will be isomorphic to the subgroup that fixes $g^{3D+3}(Y),
\dots, g^{3D+3+m}(Y)$. Hence the finiteness of the former implies the
finiteness of the later. Therefore there are finitely many
possibilities for $[\phi,g]$.

For claim (b) we note that if $[\phi,g] = [\psi, g]$ then $\psi^{-1}
\phi$ conjugates $g$ to itself and therefore $\psi^{-1}\phi \in
EC(g)$. By Proposition \ref{axis}, $\bY(g)$ is order isomorphic to
$\Z$ and the induced map $(\psi^{-1}\phi)_*$ on $\Z$ is an
isometry. If $(\psi^{-1}\phi)_*$ was a reflection then it would
conjugate $g$ to $g^{-1}$ so we must have that $(\psi^{-1}\phi)_*$ is
a translation. Since the translation distance of $(\psi^{-1}\phi)_*$
on $\Z$ will be at most the translation distance of $\psi^{-1}\phi$ on
$\cP_K(\bY)$ and $\phi$ and $\psi$ translate $Y$ at most $D$ we have
that the translation length $(\psi^{-1}\phi)_*$ is at most $2D$. There
is a bijection from the subgroup that fixes $\bY(g)$ point-wisely to
the set of elements that translate $\bY(g)$ any fixed length. Since
the former is finite by (ii) so is the later. This implies that there
are finitely many possible elements that translate $\bY(g)$ with
translation length $\le 2D$ and hence finitely many possibilities for
$\psi^{-1}\phi$ and $\psi$ proving (b) and the proposition.
\end{proof}

\begin{remark}\label{WPD free}
In Corollary \ref{bushy} we produced 
a non-elementary cobounded action on a quasi-tree 
if $G$ is non elementary
by finding a free subgroup $F<G$ using Proposition \ref{free group}.
Furthermore, each non-trivial element in $F$ will be WPD
on the quasi-tree if the stabilizer of two vertices in $\cP_K(\bY)$  is finite
(for example, in many examples in Example \ref{ex}), where
$\bY$ is the set of translates of an axis.
This is because  we only need to verify  (i) of Proposition \ref{wpd}
since (ii) is a trivial consequence of (i) under the extra assumption. 
But when we apply Proposition \ref{free group} in the proof of
Corollary \ref{bushy}, putting $L=K''$, 
for all non-trivial elements $\phi \in F$ we have 
$d_W(\phi^{-1}(W), \phi(W)) > L$, which verifies 
the condition (i) for $N=1$, therefore $\phi$ is WPD.
\end{remark}

\section{A quasi-tree of metric spaces}\label{s3}

\subsection{Axioms and construction}\label{axiom}

In all examples in  Example \ref{ex} the set $\bY$ and the functions $d^\pi_Y$ all
arose from geometric settings. We now formalize this. For each $Y \in
\bY$ let $\cC(Y)$ be a geodesic metric space. In the introduction our
notation was such that $Y$ itself was a metric space and
$\cC(Y)=Y$. But now we will make a distinction, motivated by the
example where elements $Y\in\bY$ represent incompressible subsurfaces
of a surface $\Sigma$ and $\cC(Y)$ is the curve complex of $Y$.
Let $\pi_Y$ be a
function, called {\it projection}, from $\bY \backslash \{Y\}$ to
subsets of $\cC(Y)$. We then define $\pi_Y$ on $x \in \C(X)$
for $X\neq Y$ by $\pi_Y(x) = \pi_Y(X)$. On $\cC(Y)$ itself we define $\pi_Y$ to
be the identity map. (Strictly speaking $\pi_Y$ takes points in $\cC(Y)$ to
singleton subsets of $\C(Y)$.) We now assume there is a constant $\theta \ge 0$ such that 
\begin{enumerate}[(P0)]
\item for all $X \neq Y$, $\diam(\pi_Y(X)) \leq \lc$;
\end{enumerate}
We then define $$d^\pi_Y(X,Z)=\diam\{\pi_Y(X) \cup
  \pi_Y(Z)\}.$$
We assume that axioms (P1) and (P2) hold for $\theta$
(see the introduction).
Then, as we said, the projection complex axioms (PC 1) -- (PC 4) in Section \ref{section.projection}  immediately 
follow for $d^\pi_Y$ and $\theta$.

Note that the examples \ref{ex} that were discussed at the start of the paper
all arise in this way. We also define
$d^\pi_Y(x,z)=\diam\{\pi_Y(x)\cup\pi_Y(z)\}$, and similarly for
$d^\pi_Y(x,Z)$. Note that $d^\pi_Y(x,z)$ still makes sense if $x \in \cC(Y)$
and/or $z \in \cC(Y)$ as does $d^\pi_Y(x,Z)$ if $x \in \cC(Y)$.

We define $d_Y(X,Z)$ exactly as before.
Moreover,  (1) if neither $x \in \cC(Y)$ nor $z
\in \cC(Y)$ then we set $d_Y(x,z) = d_Y(X,Z)$; (2) if either $x \in \cC(Y)$ or $z
\in \cC(Y)$ then $d_Y(x,z) = d^\pi_Y(x,z)$; (3) if $Y\neq Z$, 
then $d_Y(x,Z) =
d^\pi_Y(x,Z)$.
 In these last two cases we don't have the monotonicity
lemma and in fact the lemma doesn't even make sense. Finally we define
$\bY_K(x,z)$ to be the set of $Y$ such that $d_Y(x,z) > K$. These sets
are almost the same as $\bY_K(X,Z)$ although they may possibly contain
$X$ or $Z$. We similarly define $\bY_K(x,Z)$.

The following definition depends not only on the choice of $K$ but
also on the choice of a constant $L$.

\begin{definition}
A {\it quasi-tree of metric spaces} is
the path metric space $\cC(\bY)=\cC_K(\bY)$ obtained by taking the disjoint
union of the metric spaces $\cC(Y)$ for $Y\in\bY$ and if $d(X,Z) = 1$
in $\cS_K(\bY)$ we attach an edge of length $L$ from every point in
$\pi_X(Z)$ to every point in $\pi_Z(X)$. 
\end{definition}

For any two choices of $L$
the corresponding complexes will be quasi-isometric; however, 
we will fix $L$ as a function of
$K$ in Lemma \ref{coarsedistanceestimate} below, and we regard the
construction of $\cC(\bY)$ as depending on $K$ only.
In this way, we can assure that the metric spaces $\cC(Y)$ will be totally geodesically
embedded in $\cC(\bY)$ but that $L$ will still be comparable to
$K$. This will streamline some of our proofs. 
Note that $|L-K|$ is bounded above
by a constant depending only on $\theta$, and that in particular, $L(K) < 2K$ 
and $K < 2L(K)$ if $K$ is sufficiently large (we could assume $K \le L$
then $K<2L$ is trivial).

\begin{lemma}\label{coarsedistanceestimate}
There exists an $L = L(K)$ with $L \sim K$ such that
$$d_{\cC(\bY)}(x,z) \ge d^\pi_Y(x,z)$$ for all $Y \in \cC(\bY)$ with
equality if and only if both $x$ and $z$ are in $Y$. In particular
each $\cC(Y)$ is totally geodesically embedded in $\cC(\bY)$.
\end{lemma}

 Note that in this lemma we
use the unmodified projection functions, $d^\pi_Y$ as we will need to
apply the triangle inequality an indeterminate number of times. To
simplify notation we will restrict the discussion to the case when
each $\cC(Y)$ is a connected graph endowed with length metric with
each edge of length 1 and the projections $\pi_Y(X)\subset \cC(Y)$ are
sets of vertices. The general case is an easy modification, or indeed,
one may replace $\cC(Y)$ by the Vietoris-Rips complex whose vertices
are the points of $\cC(Y)$, and edges correspond to pairs of points at
distance $\leq 1$. Also in Lemma \ref{movebound2} and \ref{nested_guards2} we view all points as vertices.

\begin{proof} 
Let $\cC'(\bY)$ be the space obtained by collapsing $\cC(Z)$ for every
$Z \in \bY\backslash \{Y\}$. Let $x_0, x_1, \dots, x_k$ be a shortest
path of adjacent vertices between the images of $x$ and $z$ in
$\cC'(\bY)$. Thus each $x_i$ is either a vertex in $\cC(Y)$ or it is
some $Z\in\bY\setminus\{Y\}$.

We'll show that $d^\pi_Y(x_i, x_{i+1}) \leq d_{\cC'(\bY)}(x_i,
x_{i+1})$ with equality if and only if both $x_i$ and $x_{i+_1}$ are in $\cC(Y)$. There are three cases. If neither $x_i$ or $x_{i+1}$ are in
$\cC(Y)$ then by the coarse equality
$$d^\pi_Y(x_i,x_{i+1}) \prec d_Y(x_i, x_{i+1}) < K$$
and
$$d_{\cC'(\bY)}(x_i, x_{i+1}) =L.$$
Since $d^\pi_Y(x_i,x_{i+1})$ is bounded above by $K$ plus a constant depending
only on $\theta$, 
$$d^\pi_Y(x_i, x_{i+1}) < d_{\cC'(\bY)}(x_i, x_{i+1})$$
if $L$ is sufficiently large, but also we may assume $L \sim K$. If $x_i$ and
$x_{i+1}$ are both in $\cC(Y)$ then $d_{\cC'(\bY)}(x_i,
x_{i+1})=d^\pi_Y(x_i, x_{i+1}) = 1$. If exactly one of the two is in
$\cC(Y)$ we have $d^\pi_Y(x_i, x_{i+1}) \sim 0$ and $d_{\cC'(\bY)}(x_i,
x_{i+1}) = L$ so $d^\pi_Y(x_i, x_{i+1}) <
d_{\cC'(\bY)}(x_i, x_{i+1})$ for sufficiently large $L$. Again $L$ can be chosen such that $L \sim K$.

The triangle inequality then shows that
$$d_{\cC'(\bY)}(x_0, x_k) \geq d^\pi_Y(x_0, x_k) = d^\pi_Y(x,z)$$
with equality if and only if all of the $x_i$ are in $\cC(Y)$.
Since the projection to $\cC'(\bY)$ is 1-Lipschitz we have
$$d_{\cC(\bY)}(x,z) \geq d^\pi_Y(x,z)$$
with equality if and only if $x$ and $z$ are in $\cC(Y)$.

To see that $\cC(Y)$ is totally geodesically embedded in $\cC(\bY)$ we
observe that $d^\pi_Y$ is the metric on $\cC(Y)$ and we have just
shown that if $x$ and $z$ are in $\cC(Y)$, any path in $\cC(\bY)$ that
leaves $\cC(Y)$ has length strictly longer than
$d^\pi_Y(x,z)$. Therefore every geodesic from $x$ to $z$ is contained
in $\cC(Y)$.
\end{proof}

\subsection{Distance estimate in $\cC(\bY)$}

The main result of this section is Theorem
\ref{masur-minsky1}, which is a distance estimate in the style of Masur-Minsky.
We start by writing down a straightforward estimate for
an upper bound for the distance in $\cC(\bY)$. This is obtained by
constructing a ``standard path'' joining two points and computing its
length. 

\begin{definition}\label{sp}
A {\it
  standard path} from $x\in\cC(X)$ to $z\in\cC(Z)$ is any path that
passes through $\cC(W)$ if and only if $W\in \bY_K(X,Z)\cup\{X,Z\}$,
it passes through them in the natural order,
and within each $\cC(W)$ the path is a geodesic.
\end{definition}

\begin{lemma}\label{standard_path}
For $K$ sufficiently large
$$d_{\cC(\bY)}(x,z) \le 6K +4\sum_{Y \in \bY_K(x,z)} d_Y(x,z)$$
for all $x,z \in \cC(\bY)$, and moreover the length of any standard
path from $x$ to $z$ is bounded above by the same expression.
\end{lemma}

\begin{proof} 
Let $X$ and $Z$ be the vertices in $\bY$ with $x \in \cC(X)$
and $z \in \cC(Z)$. Let $\bY_K(X,Z) \cup \{X,Z\} = \{X=Y_0,Y_1, \dots,
Y_k =Z\}$ with labeling respecting the order (cf. Proposition \ref{connected} and its proof). Let $x^+_i$ be a point
in $\pi_{Y_i}(Y_{i+1})$ and $x^-_i$ a point in $\pi_{Y_i}(Y_{i-1})$,
where defined. At the endpoints let $x^-_0 = x$ and $x^+_{k} =
z$. Since the distance between $x^+_i$ and $x^-_{i+1}$ is $L$ we have
\begin{equation*}
d_{\cC(\bY)}(x,z) \le kL + \sum d_{\cC(Y)}(x^-_i, x^+_i).
\end{equation*}

Now we estimate $d_{\cC(\bY)}(x^-_i, x^+_i)$. For $i \in \{1, \dots, k-1\}$ 
we have
\begin{eqnarray*}
d_{\cC(\bY)}(x^-_i, x^+_i) & \le & d^\pi_{Y_i}(Y_{i-1}, Y_{i+1}) \\
& \prec & d_{Y_i}(Y_{i-1}, Y_{i+1}) \\
& \prec & d_{Y_i}(x,z)
\end{eqnarray*}
where the second line follows from the coarse equality property and the third follows from the order property. Since $d_{Y_i}(x,z) > K$ this implies that 
$$d_{\cC(\bY)}(x^-_i, x^+_i) < 2 d_{Y_i}(x,z)$$
for $K$ sufficiently large.

Since $L = L(K) \sim K$ we also have that $L< 2K$ if $K$ is sufficiently large and since $d_{Y_i}(x,z) > K$ we have $L < 2d_{Y_i}(x,z)$ and
$$L + d_{\cC(\bY)}(x^-_i, x^+_i) \leq 4d_{Y_i}(x,z).$$

We similarly have that $d_{C(\bY)}(x^-_i, x^+_i) \prec d_{Y_i}(x,z)$
when $i=0,k$. 

If $d_{Y_i}(x,z) > K$ we have
$d_{\cC(\bY)}(x^-_i, x^+_i) < 2d_{Y_i}(x,z)$ while if $d_{Y_i}(x,z)
\le K$ then $d_{\cC(\bY)}(x^-_i, x^+_i) < 2K$. We can write this as a
single inequality
$$d_{\cC(\bY)}(x^-_i, x^+_i) < 2\max\{K,d_{Y_i}(x,z)\}$$ that applies
to both cases. Now

\begin{eqnarray*}
d_{\cC(\bY)}(x,z) & \le & kL + \sum d_{\cC(Y)}(x^-_i, x^+_i) \\
& \le & L+ 4\sum_{i=1}^{k-1} d_{Y_i}(x,z) + 2\sum_{i=0,k} \max\{K, d_{Y_i}(x,z)\}\\
& \le &6K + 4\sum_{Y \in \bY_{K}(x,z)} d_Y(x,z)
\end{eqnarray*}

\end{proof}

We aim to find a lower bound in the spirit of Lemma
\ref{masur-minsky0} for the projection complex $\cS_K(\bY)$. See
Theorem \ref{masur-minsky1}.
We will need a version of Proposition \ref{key} for $\cC(\bY)$. The proof will be a word for word repeat of Proposition \ref{key} but first we need a new version of Lemma \ref{movebound}.

\begin{lemma}\label{movebound2}
Let $X_0$ and $X_1$ be vertices in $\cP_K(\bY)$ with $d(X_0, X_1) = 1$ and
let $x_0$ and $x_1$ be vertices in $\cC(X_0)$ and $\cC(X_1)$ such that
$x_0 \in \pi_{X_0}(X_1)$ and $x_1 \in \pi_{X_1}(X_0)$. Let $W$ be a
vertex in $\bY$ and $w$ a vertex in $\cC(W)$ with $d_{\cC(\bY)}(x_i, w)
\ge 2L$. Then either
$$d_W(x_0, x_1) \sim 0$$
or
$$d_W(x_i, w) \succ L\mbox{ for }i=0,1.$$
\end{lemma}

\begin{proof} 
First assume $X_0 = W$. Since $x_0 \in \pi_W(x_1) = \pi_{X_0}(x_1)$ we
have $d^\pi_W(x_0, x_1) \le \diam(\pi_W(X_1)) \sim 0$. Of course, we get
the same bound if $X_1 = W$.

If either $d(X_0, W) \ge 2$ or $d(X_1, W) \ge 2$ then 
$d^\pi_W(x_0, x_1) = d^\pi_W(X_0, X_1) \sim 0$ by Lemma \ref{movebound}.

This leaves us with the case where $d(X_0,W) = d(X_1,W) =1$. We first
observe that if $d_{X_0}(X_1,W) > \xio$ then $d_W(x_0, x_1) =
d_W(X_0, X_1) \sim 0$. The same estimate holds if
$d_{X_1}(X_0, W) > \xio$.

The final sub-case is when both $d_{X_0}(X_1,W) \le \xio$ and
$d^\pi_{X_1}(X_0,W) \le \xio$. It is here that we use the lower bound 
$d_{\cC(\bY)}(x_i, w) \ge 2L$. To do so we need the upper bound
$$d_{\cC(\bY)}(x_0, w)  \leq  d_{X_0}(x_0, w) + L + d_W(x_0,w)$$
which is obtained by taking the path made up of a path in 
$\cC(X_0)$ connecting $x_0$ to $\pi_{X_0}(w)$, an edge from  $\pi_{X_0}(W)$ to $\pi_W(X_0)$ and a path in $\cC(W)$ from $\pi_W(X_0)$ to $w$. Since $x_0 \in \pi_{X_0}(X_1)$ we have $d_{X_0}(x_0, w) \prec d_{X_0}(X_1, W)$. Combining the bounds gives $d_W(x_0,w) \succ L$ and the same bound holds for $d_W(x_1, w)$.
\end{proof}

\begin{lemma}\label{nested_guards2}
For $K$ sufficiently large the following holds. Let $x_0$ and $x_1$
be adjacent vertices in $\cC(\bY)$ and let $Y$
be a vertex in $\cS_K(\bY)$ such that $d_{\cC(\bY)}(x_i, \cC(Y)) \ge
3L$. If $W$ is a guard for $Y$ with $W \in \bY_{K/2}(x_0,Y)$ and $W
\not\in \bY_{K/2}(x_1,Y)$ then there exists a guard $W'$ for $Y$ with
$W' \in \bY_{K/2}(x_1, Y)$ and $W \in \bY_{\xio}(W', Y)$.
\end{lemma}

\begin{proof}
Let $X_0$ and $X_1$ be the vertices of $\cS_K(\bY)$ such that $x_i \in \cC(X_i)$. If $X_0 = X_1 \neq W$ then $W \in \bY_{K/2}(x_1, Y)$ and the lemma is vacuous. If $X_0 = X_1 = W$ then
\begin{eqnarray*}
3L & \le & d_{\cC(\bY)}(x_i, \cC(Y)) \\
& \le & d_{\cC(\bY)}(x_i, \pi_Y(W)) \\
& \le & d_{W}(x_i, Y) + L
\end{eqnarray*}
and therefore $d^\pi_W(x_i, \pi_W(Y)) \ge 2L$. Since $L \sim K$ if $K$ is sufficiently large then $2L > K$ and $W \in \bY_{K/2}(x_1, Y)$,
therefore the lemma is vacuous as well.

We now assume that $X_0 \neq X_1$. We can now apply Lemma \ref{movebound2} with $w$ a point in $\pi_W(Y)$. Note that $d_{\cC(\bY)}(w, \cC(Y)) = L$ so $d_{\cC(\bY)}(x_i, w) \ge 2L$.

Lemma \ref{movebound2} gives us two possibilities. First we may have $d_W(x_1, w) \succ L \succ K$ in which case $W \in \bY_{K/2}(x_1, Y)$ for $K$ sufficiently large.

Therefore if $W \not\in \bY_{K/2}(x_1, Y)$ then Lemma \ref{movebound2}
gives $d_W(x_0,x_1) \sim 0$. For $K$ sufficiently large the coarse
triangle inequality then implies that $W \in \bY_\xio(x_1, Y)$ as $W
\in \bY_{K/2}(x_0,Y)$. Since $W$ is a guard for $Y$ every vertex in
$\bY_K(x_1, Y)$ must be less than $W$ in
$\bY_\xio(x_1,Y)$. Furthermore $\bY_K(x_1,Y)$ can't be empty for if it
was then, as above, $d(x_1, \cC(Y)) \le d_{X_1}(x_1, Y) + L \le K+L <
3L$ if $K$ is sufficiently large.  Therefore there must be elements
($\not= W$, could be $=X_1$) of $\bY_{K}(x_1,Y)$ that are less than
$W$ in $\bY_{\xio}(x_1,Y)$. The rest of the proof now is a repeat of
the proof of Lemma \ref{nested_guards}. Namely, we take $W'$ to be the
greatest element of $\bY_{K/2}(x_1, Y)$ that is less than $W$ in
$\bY_\xio(x_1,Y)$. The proof that $W \in \bY_{\xio}(W',Y)$ and that
$W'$ is a guard is exactly as in the proof of Lemma
\ref{nested_guards}.
\end{proof}

We define the notion of a {\em barrier} for a path in $\cC(\bY)$ just
as we did for paths in $\cS_K(\bY)$. Namely, if $\{x_0, x_1, \dots,
x_k\}$ is a path in $\cC(\bY)$ and $Z$ a vertex in $\cS_K(\bY)$ then
$Y \in \bY$ is a barrier between them if $Y \in \bY_\xio(x_i, Z)$ for
$i=0, \dots, k$. Note that it is possible that $x_i \in \cC(Y)$. If
neither $x_i$ nor $x_j$ are in $\cC(Y)$ then Theorem \ref{main}
implies that $d_Z(x_i, x_j) < \xio$. If exactly one of the two is in
$\cC(Y)$ then $d_Z(x_i, x_j) < \xio$ from the inequality on
triples. If they are both in $\cC(Y)$ then $d_Z(x_i, x_j) = \pi_Z(Y) <
\xio$ by (P0).

\begin{prop}\label{key2}
Let $\{x_0, x_1, \dots, x_k\}$ be a path in $\cC(\bY)$ and $Z\in\bY$
such that $d_{\cC(\bY)}(x_i, \cC(Z)) \ge 3L$ for all $i$.  Then there
is a barrier $C$ in $\bY$ between the path and $Z$. In particular,
$d_Z(x_0,x_i) < \xio$.
\end{prop}

\begin{proof} 
The proof is a word for word repeat of the proof of Proposition \ref{key} with Lemma \ref{nested_guards} replaced with Lemma \ref{nested_guards2} and the upper case $X_i$ replaced with the lower case $x_i$.
\end{proof}

\begin{remark}
It is not hard to derive Proposition \ref{key} from Proposition
\ref{key2}. In particular a path in $\cS_K(\bY)$ that is $3$ or more
away from a vertex $Z$ can be lifted to path in $\cC(\bY)$ that is 
$3L$ away from $\cC(Z)$.
\end{remark}

The next lemma establishes that the nearest point projection to
$\cC(Z)$ agrees, to within a bounded error, with the prescribed
projections. 

\begin{lemma}\label{nearest_point}
Let $x$ be a vertex in $\cC(\bY)$, $Z$ a vertex in $\cS_K(\bY)$ and
$z$ a nearest point in $\cC(Z)$ to $x$ in $\cC(\bY)$. Then
$$d_Z(x,z) \prec 2K.$$
\end{lemma}

\begin{proof}
Let $y$ be the last point in a geodesic from $x$ to $z$ such that
$d_{\cC(\bY)}(z, y) = d_{\cC(\bY)}(y, \cC(Z)) \ge 3L$. Then by Proposition
\ref{key2}, $d_Z(x,y) \sim 0$.
The case that such $y$ does not exist, i.e., $d_{\cC(\bY)}(z, x) <3L$,
will be discussed at the end.

If a path in $\cC(\bY)$ of length at most $kL-1$
maps to a path in $\cS_K(\bY)$
then the image path will have length at most $k-1$. 
By the way we chose $y$, $d_{\cC(\bY)}(z, y) \le 4L-1$.
Therefore the
geodesic from $y$ to $z$ will map to a path of length at most $3$ 
(and at least $1$) in
$\cS_K(\bY)$. Let $Y$ and $Z'$ be the vertices of $\cS_K(\bY)$ such
that $y \in \cC(Y)$ and $Z'$ is the last vertex in the path before
$Z$. Since $d(Y,Z') \le 2$, the coarse triangle inequality implies
that $d_Z(Y,Z') \prec 2K$. (We are assuming 
$Z\neq Z'$ here, but the case $Z=Z'$ is similar and 
left to the reader.) Since $Z'$ is the last vertex before $Z$ we
also have that $z \in \pi_Z(Z')$ and therefore $d_Z(z,y) \prec
2K$. Since $d_Z(x,y) \sim 0$, 
another application of the coarse triangle inequality then gives
$d_Z(x,z) \prec 2K$ as claimed.

Now we are left with the case  $d_{\cC(\bY)}(z, x) <3L$.
If $x \in \cC(Z)$, then $z=x$ and there is nothing to prove.
Otherwise, letting $y=x$ in the above discussion, we have 
 $d_Z(z,x) \prec 2K$.
\end{proof}

The nearest point projection $\cC(\bY)\to\cC(Z)$ is not really a
function since the image of a point is not always a single
point. However, it is a {\it coarse map}, i.e. the diameter of the
image set is uniformly bounded by Lemma \ref{nearest_point}. Recall
that a coarse map $F$ between two metric spaces is {\it coarsely
  Lipschitz} if there exist constants $a,b>0$ such that $\diam
F(A)\leq a\diam(A)+b$. 

\begin{cor}\label{nearest_point2}
For every $Z\in\bY$
the nearest point projection $\cC(\bY)\to\cC(Z)$ is coarsely Lipschitz
and the image of $\cC(Y)$ for $Y\neq Z$ is in a uniform neighborhood
of the bounded set $\pi_Z(Y)$.
\end{cor} 

\begin{proof}
Let $x_1,x_2$ be two vertices of $\cC(\bY)$ that are joined by an edge
and say $x_i\in\cC(X_i)$ for $i=1,2$. We need to argue that the images
of $x_i$ are uniformly close. There are several cases.

{\it Case 1.} $x_1,x_2$ are joined by an edge of length 1. Then
$X_1=X_2$ and the images of $x_1,x_2$ are uniformly close, by Lemma
\ref{nearest_point}, to $\pi_Z(X_1)=\pi_Z(X_2)$.

{\it Case 2.} $x_1,x_2$ are joined by an edge of length $L$; thus
$d(X_1,X_2)=1$. If $X_1\neq Z\neq X_2$ then $d_Z(X_1,X_2)\leq K$ and
we again see from Lemma \ref{nearest_point} that the images of $x_1$
and $x_2$ are uniformly close. Finally, if $X_1=Z\neq X_2$, then
$x_1\in \cC(Z)$ is its own image, while the image of $x_2$ is at most
$2L$ away.
\end{proof}

\begin{proof}[Proof of Theorem A]
This now follows from Lemma \ref{coarsedistanceestimate} and Corollary
\ref{nearest_point2}.
\end{proof}

The next two statements say that $\cC(\bY)$ is a quasi-tree-like
union of spaces $\cC(Y)$.

\begin{prop}\label{standard_points}
Let $X,Z\in \bY$, $x\in\cC(X)$, $z\in\cC(Z)$. If $Y \in \bY_\xio(x,z)$ 
then any path from $x$ to $z$ in $\cC(\bY)$
contains a vertex $w$ such that
\begin{itemize}
\item $d_{\cC(\bY)}(w,\cC(Y))< 3L$,
\item $d_Y(x,w) \prec K$.
\end{itemize}
It follows that  $d_{\cC(\bY)}(w,\pi_Y(x))\prec  3L+3K.$
(A similar statement holds with $z$ in place of $x$.)
\end{prop}

\begin{proof}
By Proposition \ref{key2} every path from $x$ to $z$ must intersect
the $3L$-neighborhood of $\cC(Y)$ if $Y \not=X,Z$.
This is trivially true if $Y=X$ or $Y=Z$. Let $w$ be the first vertex in the
path with $d_{\cC(\bY)}(w, \cC(Y)) < 3L$ and let $w'$ be the vertex
that precedes it. (If $w= x$ then the lemma holds trivially.) By
Proposition \ref{key2}, $d_Y(x,w') \sim 0$. Since $w$ and $w'$ are
adjacent in $\cC(\bY)$ they will map to either adjacent vertices in
$\cS(\bY)$ or the same vertex. In either case $d_Y(w,w') \prec K$ and
by the coarse triangle inequality $d_Y(x,w) \prec K$.

Now let $\tilde w \in \cC(Y)$ be a nearest point from $w$
to $\cC(Y)$. We have $d_{\cC(\bY)}(w,\tilde w)< 3L$.
By Lemma \ref{nearest_point},
$d_Y(\tilde w,w) \prec 2K$. Therefore by the coarse triangle inequality
 $d_{\cC(\bY)}(w,\pi_Y(x)) \prec 3L + 2K + K$.
\end{proof}

\begin{lemma}\label{forced}
There exists $K'>0$ so that the following holds. If $x\in\cC(X)$,
$z\in \cC(Z)$, and $Y\in \bY_{K'}(x,z)$, then every geodesic $V$ in
$\cC(\bY)$ from $x$ to $z$ intersects $\cC(Y)$ in a geodesic segment
$[v,w]$ and moreover $d_Y(x,v)\prec K'$, $d_Y(z,w)\prec K'$.
$Y$ is possibly $X$ or $Z$.
\end{lemma}

\begin{proof}
First note that by Lemma \ref{coarsedistanceestimate} the
intersection, if nonempty, is a geodesic segment (possibly a single
point). From Proposition \ref{standard_points} it follows that there are
points $v',w'$ along $V$ so that $d(v',\pi_Y(x))\prec 3L+3K$ and 
$d(w',\pi_Y(z))\prec 3L+3K$. In particular, $d(v',w')\prec
6L+6K+d_Y(x,z)$. 

Assuming the subsegment $[v',w']\subset V$ is disjoint from $\cC(Y)$, we
estimate the number of $\cC(W)$'s $[v',w']$ has to pass through as being
at least $\frac {d_Y(x,z)}K-1$ (the diameter of the projections
to $Y$ of the union of two consecutive $\cC(W)$'s is at most $K$). 
Thus the number of
edges of length $L$ the segment passes through is at least $\frac
{d_Y(x,z)}K$, and we have
$$\frac{L d_Y(x,z)}K\prec 6L+6K+d_Y(x,z)$$
Since $L/K>1$ we get a contradiction when $d_Y(x,z)$ is large enough. 
We have shown that if $K'$ is large enough then 
$[v',w'] \cap \cC(Y) \not= \emptyset$.

Thus $[v',w']\cap \cC(Y)$ is a geodesic segment $[v,w]$. We will
argue that $v$ is uniformly close to $\pi_Y(x)$; the argument that
$w$ is uniformly close to $\pi_Y(z)$ is symmetric. Let $v''$ be the
vertex on the segment $[x,v]\subset V$ immediately preceding $v$ (if
$x=v$ there is nothing to prove). If
$d(\pi_Y(x),\pi_Y(v''))>K'$ we may apply the argument of the preceding
paragraph to the geodesic $[x,v'']$ to deduce $[x,v'']\cap
\cC(Y)\neq\emptyset$, a contradiction. Thus
$d(\pi_Y(x),\pi_Y(v''))\leq K'$ and so $d_Y(x,v)\prec K'$.
\end{proof}

The following is the distance estimate analogous to the Masur-Minsky
formula. 

\begin{thm}\label{masur-minsky1}
There is $K'>K$ such that for $x\in \cC(X), z\in\cC(Z)$
$$\frac 12\sum_{W \in \bY_{K'}(x,z)} d_W(x,z)\leq
d_{\cC(\bY)}(x,z)\leq 
6K+4\sum_{W\in\bY_K(x,z)}d_Y(x,z)
$$
\end{thm}

\begin{proof}
The upper bound is Lemma \ref{standard_path}.
Let $K'$ be the constant from Lemma \ref{forced} and assume that
$d_Y(x,z)>6K'$. Then any geodesic from $x$ to $z$ intersects $\cC(Y)$
in a segment of length $\succ 4K'$, which is $>3K'$. The estimate
follows after renaming $6K'$ to $K'$.
\end{proof}

\subsection{Hyperbolicity of $\cC(\bY)$}\label{section.hyperbolic}

In this section we prove that if all $\cC(Y)$ uniformly satisfy the
bottleneck property, or hyperbolicity, or quasi-convexity,
then $\cC(\bY)$ satisfies the same property.

\begin{thm}\label{bottleneck}
Suppose that all $\cC(Y)$ for $Y\in\bY$ are quasi-trees in a uniform
way, so that
there is $\Delta$ such that all $\cC(Y)$ for $Y\in\bY$ satisfy
the bottleneck property with this $\Delta$. Then $\cC(\bY)$ satisfies
the bottleneck property so it is a quasi-tree.
\end{thm}

\begin{proof}
Let $x\in\cC(X)$ and $z\in\cC(Z)$ be given and let
$Y_1,Y_2,\cdots,Y_s$ be the elements of $\bY_K(X,Z)$ with
indexing reflecting the order. There is a standard path
(see the proof of Lemma \ref{standard_path}) $V$ in $\cC(\bY)$ from
$x$ to $z$ that projects to $\{X,Y_1,Y_2,\cdots,Y_s,Z\}$ and within
each $\cC(Y_i)$ (we let $Y_0=X,Y_{s+1}=Z$) it is a geodesic. We will argue
that any path $U$ from $x$ to $z$ comes within a bounded distance from
any point on $V$. This verifies the modified bottleneck property
discussed just before
Theorem D.

Fix a point $v \in \cC(Y_i)$ on $V$ and let $\{x=x_0, x_1, \dots,
x_k=z\}$ be the vertices of an arbitrary path $U$ between $x$ and $z$. We
project the $x_j$ to $Y_i$ and let $y_j$ be points in
$\pi_{Y_i}(x_j)$. Note that $d_{\cC(Y_i)}(y_j,y_{j+1}) \prec K$ so the
$y_j$ form a coarse path in $\cC(Y_i)$ from $y_0 = \pi_{Y_i}(x)$ to
$y_k = \pi_{Y_i}(z)$. Since $\cC(Y_i)$ satisfies the bottleneck
property with constant $\Delta$,
$d(y_0,\pi_{Y_i}(Y_{i-1})) \sim 0$ and $d(y_k,\pi_{Y_i}(Y_{i+1})) \sim 0$ 
by the order property, 
 there will be some $y_\ell$ with
$d_{\cC(Y_i)}(y_\ell, v) \prec \Delta+ K$. Note that if $K$ is
sufficiently large then at least one of $d_{Y_i}(x, x_{\ell})$ and
$d_{Y_i}(z, x_{\ell})$ must be large enough to apply Proposition
\ref{standard_points}. Assume it is the former. Applying Proposition
\ref{standard_points} there exists a vertex $x_{\ell'}$ on the path
between $x$ and $x_{\ell}$ such that
$$d_{\cC(\bY)}(x_{\ell'}, \cC(Y_i)) < 3L$$
and
$$d_{Y_i}(y_\ell, x_{\ell'}) \prec K$$
since $y_\ell \in \pi_{Y_i}(x_\ell)$. Let $w \in \cC(Y_i)$ be the closest point in $\cC(\bY)$ to $x_{\ell'}$. Then by Lemma \ref{nearest_point} and the coarse triangle inequality we have
$$d_{Y_i}(w, v) \prec d_{Y_i}(w, x_{\ell'})
+d_{Y_i}(x_{\ell'}, y_\ell)+d_{Y_i}(y_\ell, v)
\prec \Delta + 4K$$
and, since $d_{\cC(\bY)}(w, x_{\ell'})<3L$,
$$d_{\cC(\bY)}(x_{\ell'}, v) \prec \Delta + 4K + 3L.$$
This proves that the bottleneck property holds
since $x_{\ell'} \in U$.
\end{proof}

A geodesic metric space is {\it quasi-convex} if there is $N>0$ such
that for any two geodesic segments $[u,v]$ and $[u',v']$, if
$d(u,u')\leq 1$ and $d(v,v')\leq 1$ then $[u',v']$ is contained in the
Hausdorff $N$-neighborhood of $[u,v]$. Note that this implies that if
$d(u,u')\leq C$, $d(v,v')\leq C$ then $[u',v']$ is contained in the
Hausdorff $(C+1)N$-neighborhood of $[u,v]$.

Also note that if each $\cC(Y)$ is quasi-convex with the same constant (then we say {\it uniformly quasi-convex}),
then there is a uniform bound on the Hausdorff distance 
of any two standard paths between any two points in $\cC(\bY)$.

\begin{lemma}\label{Hausdorff}
Suppose that each $\cC(Y)$ is quasi-convex with the same constant $N$.
There is $M>0$ so that for any $x$ and $z$, the Hausdorff distance
between any geodesic from $x$ to
$z$ and any standard
path (see Definition \ref{sp}) from $x$ to $z$ is at most $M$.
\end{lemma} 

\begin{proof}
If $[v,w]$ is a segment in a standard path $U$ obtained by intersecting
with some $\cC(W)$, then the endpoints are
within uniform distance of any geodesic $V$ from $x$ to $z$ by Proposition
\ref{standard_points} since  $W \in \bY_{\xio}(x,z)$ (the only case the lemma
does not apply is when $W=X,Z$ and $W \not\in \bY_{\xio}(x,z)$,
but then the claim is true with the bound $\xio$). 
We claim that $[v,w]$ is within uniform distance from $V$.
If  $d_W(x,z) \le K'$, then the length of the geodesic $[v,w]$
is bounded by a constant $\prec 3K'$, therefore 
$[v,w]$ is within uniform distance from $V$.
If $d_W(x,z) >K'$, then by Lemma \ref{forced} $V$
intersects $\cC(Y)$ in a geodesic segment $[v',w']$ whose
endpoints are uniform distance from the endpoints of $[v,w]$.
By the uniform quasi-convexity of $\cC(Y)$,
the claim follows. Thus the standard path $U$ is contained in a
uniform neighborhood of the geodesic $V$.

Now we show that the geodesic $V$ 
is contained in a uniform neighborhood
of the standard path $U$. Let $\bY_K(x,z)=\{Y_1,Y_2,\cdots,Y_k\}$ and let
$i_1<i_2<\cdots<i_s$ be the indices of those $Y_i$ with
$d_{Y_i}(x,z)>K'$, where $K'$ is large (at least as large as in Lemma
\ref{forced}, but in fact a bit larger, see below). Then $V\cap
\cC(Y_{i_j})$ is an interval $I_{i_j}$ and the intervals
$I_{i_1},I_{i_2},\cdots,I_{i_s}$ occur along $V$ in order of their
indices (if $I_{i_j}$ occurs after $I_{I_{j+1}}$ apply Lemma
\ref{forced} to the subsegment of $V$ that starts with $I_{i_j}$ to
get a contradiction -- this is where we need $K'$ to be larger by $\xio$
than in Lemma \ref{forced}).
Let $I_{i_j}'$ be the geodesic
segment $\cC(Y_{i_j}) \cap U$. 
Since $U$ is a standard path, the endpoints 
of  $I_{i_j}'$ are within distance $\prec \xio$ 
from $\pi_{Y_{i_j}}(x), \pi_{Y_{i_j}}(z)$, respectively.
Also, by Lemma \ref{forced}, the endpoints
of $I_{i_j}$ are within distance $\prec K'$
from $\pi_{Y_{i_j}}(x), \pi_{Y_{i_j}}(z)$, respectively. 
Therefore, $I_{i_j}$ and $I'_{i_j}$ 
are contained in a uniform neighborhood of each other by 
the uniform quasi-convexity of $\cC(Y)$.
It suffices to argue
that each complementary interval in $V$
and the corresponding (with respect to the order)
complementary interval 
in $U$ are contained in a uniform neighborhood of each other.

\begin{figure}
\centerline{\scalebox{0.7}{\input{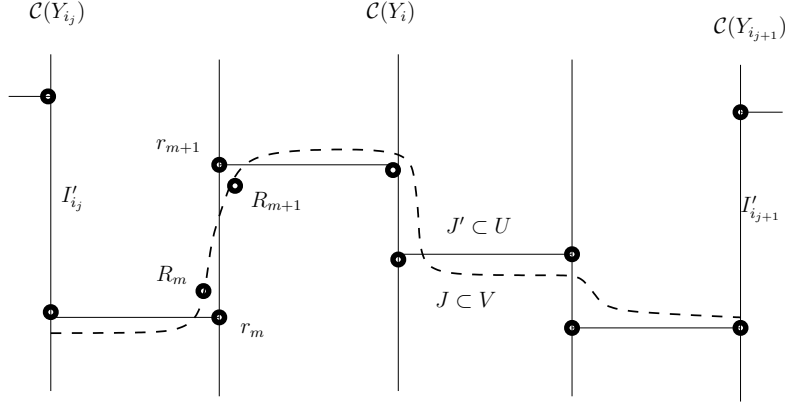}}}
\caption{Lemma \ref{Hausdorff}. $J$ is the dashed line}
\end{figure}
Let $J$ be one such complementary interval, say between $I_{i_j}$ and
$I_{i_{j+1}}$. The corresponding interval $J'$ in $U$ is 
between $I'_{i_j}$ and $I'_{i_{j+1}}$. We already know the endpoints
of $J$ and $J'$ are uniformly close. 
Note that
$Y_i\in\bY_\xio(Y_{i_j},Y_{i_{j+1}})$ for $i_j<i<i_{j+1}$, so applying
Proposition \ref{standard_points} again to $J$ we find that each endpoint $r_m$
of
each segment of $J'$ in the standard path within some $\cC(Y_i)$ is within
uniform distance
of some point $R_m$ on $J$. 
(The bound is perhaps worse than $3L+3K$
since the endpoints of $J$ and $J'$ do not exactly coincide,
but they are uniformly close, which is enough.)
Index the points $r_m$ in
order in which they occur along the standard path, and note that we do
not know that the corresponding points $R_m$ appear in linear order
along $J$. However, since $d(r_m,r_{m+1})$ is uniformly bounded
(by $L+K'$), it
follows that $d(R_m,R_{m+1})$ is uniformly bounded. Moreover, the
first point $R_1$ and the last point $R_n$ are within a uniform
distance of the corresponding endpoints of $J$. It follows that the
$R_m$'s cut $J$ into segments of bounded length and also
$r_m$'s cut $J'$ into segments of bounded length, 
therefore $J$ and $J'$ are contained in a uniform neighborhood of 
each other, and the lemma follows.

The extremal cases, when $J$ contains an endpoint of $V$, differs only
in notation and is left to the reader.
\end{proof}

\begin{remark}
A similar argument shows that $\cC(\bY)$ is quasi-convex.
\end{remark}

Recall that a geodesic metric space is $\delta$-{\it hyperbolic} if for any
three points $x,y,z$ any geodesic $[x,z]$ is contained in the
$\delta$-neighborhood of the union $[x,y]\cup [y,z]$ of any two
geodesics joining $x$ to $y$ and $y$ to $z$. A space is {\it hyperbolic} if
it is $\delta$-hyperbolic for some $\delta$.

\begin{thm}\label{hyp}
Assume that each $\cC(Y)$ is $\delta$-hyperbolic with the same
$\delta$. Then $\cC(\bY)$ is hyperbolic.
\end{thm}

\begin{proof}
Let $x,y,z$ be three vertices of $\cC(\bY)$.
Recall that $\delta$-hyperbolic spaces are quasi-convex, with the
constant depending only on $\delta$. Thus Lemma \ref{Hausdorff}
applies and it suffices to show that a standard path $U$ from $x$ to $z$
is contained in a uniform neighborhood of the union of two geodesics
$[x,y]$ and $[y,z]$.

Let $W\in\bY_K(x,z)$.
We claim that  $[\pi_W(x),\pi_W(z)]$ is contained in 
a uniform neighborhood of $[x,y]\cup [y,z]$.
First consider
the case when $d_W(x,y)>\xio$, $d_W(y,z)>\xio$. Then a geodesic
$[\pi_W(x),\pi_W(y)]\subset \cC(W)$ is contained in a uniform
neighborhood of $[x,y]$ by Proposition \ref{standard_points} and Lemma
\ref{forced} (see the first paragraph of the 
proof of Lemma \ref{Hausdorff}).
 Likewise, a geodesic $[\pi_W(y),\pi_W(z)]$ is contained
in a uniform neighborhood of $[y,z]$. Since $\cC(W)$ is
$\delta$-hyperbolic, $[\pi_W(x),\pi_W(z)]$ is contained in 
the $\delta$-neighborhood of $[\pi_W(x),\pi_W(y)]\cup [\pi_W(y),\pi_W(z)]$ and
consequently in a uniform neighborhood of $[x,y]\cup [y,z]$.

Now suppose that $d_W(x,y)\leq \xio$.
Since $W \in \bY_K(x,z)$, it follows 
$d_W(y,z)>\xio$. 
Again by Proposition \ref{standard_points} and Lemma
\ref{forced} we have that $[\pi_W(y),\pi_W(z)]$ is in a uniform
neighborhood of $[y,z]$. By quasi-convexity, 
it follows from $d_{\cC(\bY)}(\pi_W(x),\pi_W(y)) \le \xio$
 that $[\pi_W(x),\pi_W(z)]$ is
contained in a uniform neighborhood of $[\pi_W(y),\pi_W(z)]$ and hence
of $[y,z]$. The case when $d_W(y,z)\leq\xio$ is handled symmetrically.

By the definition of a standard path and the uniform quasi-convexity
of $\cC(Y)$, a standard path $U$ from $x$ to $z$ is contained
in a uniform neighborhood of the union of  $[\pi_W(x),\pi_W(z)]$
for all $W$ with  $W\in\bY_K(x,z)$ (see the proof of Lemma \ref{Hausdorff}).
Therefore it follows that $U$ is contained in a uniform 
neighborhood of $[x,y]\cup [y,z]$.
\end{proof}

\subsection{Group action and WWPD}\label{section.wwpd}
Now assume that $G$ is a group that acts on the set $\bY$, that for
each $Y\in\bY$ we have a geodesic metric space $\cC(Y)$ and
projections $\pi_Y$ satisfying the axioms (P0), (P1), (P2), and
that $G$ preserves this structure, i.e. there are isometries
$F_g^Y:\cC(Y)\to \cC(g(Y))$ so that
\begin{itemize}
\item $F_{g'}^{g(Y)}F_g^Y=F_{g'g}^Y$ for all $g,g'\in G$, $Y\in \bY$,
  and
\item $\pi_{Y}(X)=\pi_{g(Y)}(g(X))$ for all $g \in G$ and
$X,Y \in \bY$.
\end{itemize}
Then projection distances are preserved,
i.e. $d^\pi_{g(A)}(g(B),g(C))=d^\pi_A(B,C)$ for all $A,B,C\in\bY$ and
$g\in G$, and therefore $G$ acts naturally on $\cC(\bY)$. To simplify
notation, we will denote the isometry $F_g^Y$ simply by $g:\cC(Y)\to
\cC(g(Y))$.

We defined WPD for group actions in Section \ref{s:WPD}.
Here we define a weaker property, {\it WWPD}, to allow for elements with
large centralizers. We restrict ourselves to
actions on hyperbolic spaces. For a motivation, see Remark
\ref{WPD2}. 

Let $G$ act on a $\delta$-hyperbolic metric space $X$.  A hyperbolic
element $g \in G$ has a {\it quasi-axis}, which is a $g$-invariant
quasi-geodesic. 
As before  the {\it elementary closure} of $g$ (in $G$), $EC(g)$, 
is the subgroup in $G$ of elements $h$ such that $h(\gamma)$ is
parallel to $\gamma$. (We can define the elementary closure
of $g$ in a subgroup of $G$.) Equivalently, it
is the stabilizer of the set of $\gamma(\pm \infty)$, the points at infinity of
$\gamma$. The elementary closure does not depend on the choice of
$\gamma$.

\begin{definition}
Let $G$ act on a $\delta$-hyperbolic metric space $X$.  We say $g\in
G$ is a {\it WWPD element} if
\begin{enumerate}[(1)]
\item
$ g $ acts as a hyperbolic isometry on $X$,
\item
there is $x\in X$, a subgroup $N \subset G$ with $g\in N$ and a
constant $B>0$ such that
\begin{itemize}
\item for $h\in G-N$ the projection of $h(\langle g \rangle x)$ to $\langle g \rangle x$ has
  diameter $\leq B$,
\item $N$ is contained in  $EC(g)$,
and there is a homomorphism $N\to Q$ to a virtually cyclic
  group $Q$ whose kernel fixes every $g^k(x)$, $k\in\Z$.
\end{itemize}
\end{enumerate}
Moreover, if each element of $N$ fixes the points $\gamma(\pm \infty)$ pointwise,
then we say $g$ is a $WWPD^+$ element.
\end{definition}

\begin{remark}
This definition is not independent of the choice of $x$. The set of
translates of the $g$-orbit of $x$ is again ``discrete'' as in the
definition of WPD, but this time we allow a big group that fixes the
whole orbit pointwise.
Note that the image of $\langle g \rangle $ in $Q$ has finite index.
\end{remark}

\begin{prop}\label{wwpd}
Suppose each $\cC(Y)$ is $\delta$-hyperbolic so that $\cC(\bY)$ is
hyperbolic. Let $g\in G$ so that $g(Y)=Y$ and denote by $K_{\cC(Y)}$
the kernel of the action of $Stab_G(Y)$ on $\cC(Y)$.  Assume that
$g:\cC(Y)\to\cC(Y)$ is a hyperbolic WPD element for the action of
$Stab_G(Y)/K_{\cC(Y)}$ on $\cC(Y)$.  Then $g$ is a WWPD element for
the action of $G$ on $\cC(\bY)$. If moreover $Stab_G(Y)$ is virtually
cyclic then $g$ is a WPD element for the action of $G$ on $\cC(\bY)$.
\end{prop}

\begin{proof}
We take $N$ to be $EC(g)$.
Then $K_{\cC(Y)}<  N < Stab_G(Y) $. The first inclusion is clear and 
the second one follows from Corollary \ref{nearest_point2} since 
a quasi-axis of $g$ is contained in $\cC(Y)$.

Define $Q:=N/K_{\cC(Y)}$ with the obvious quotient map $N\to
Q$, and we choose $x\in\cC(Y)$. Note that $N$ is also the 
elementary closure of $g$ in $Stab_G(Y)$ and since $g$ is
WPD in $Stab_G(Y)/K_{\cC(Y)}$, $Q$ is virtually cyclic. If $h \in
G-N$ then either $h \not\in Stab_G(Y)$ and it moves the orbit
$\langle g \rangle x$ to another $\cC(Y')$ and the projection to
$\cC(Y)$ is uniformly bounded by Corollary \ref{nearest_point2},
or $h \in Stab_G(Y)$ and the projection to $\langle g\rangle x$ is
bounded by the WPD assumption by Theorem H (as we said after Theorem
H, WPD implies that since $h \not\in N$ the projection of
$h(\langle g \rangle x)$ to $\langle g \rangle x$ satisfies (P0),
namely, it is uniformly bounded).  Therefore $g$ is WWPD.

For the moreover part, note that under the assumption on $Stab_G(Y)$,
$\langle g \rangle $ has finite index in this group.  On the other
hand the set of elements in $G$ in the definition of WPD (elements
that almost fix two points at a large distance on a quasi-axis of $g$)
is contained in $Stab_G(Y)$ by Corollary \ref{nearest_point2},
therefore the concerned set is finite, hence $g$ is WPD.
\end{proof}

\begin{examples}[WPD and WWPD]
Let $\Gamma$ be a discrete group of isometries of $\H^n$ and $\bY$ the
collection of translates of the axes of a hyperbolic element $\gamma$ of
$\Gamma$, as in Example \ref{ex}(1). Then $\gamma$ is a WPD
element of $\cC(\bY)$, where $\cC(Y)\cong\R$ is the axis $Y$. Similar
conclusions hold in the other  examples. 

We will see examples of WWPD elements in Section
\ref{section.embedding}.  Elements $g \in MCG$ that are pseudo-Anosov
when restricted to a subsurface $Y$ (or Dehn twists when $Y$ is an
annulus) will be WWPD elements, but not WPD in general, for the action
of the mapping class group on $\cC(\bY^i)$, where $Y \in \bY^i$. To be
precise we may only have the action on the color preserving subgroup
in Lemma \ref{finite index}, and assume $g$ is contained in the
subgroup otherwise take a finite power to satisfy this.  To verify
that $g$ is WWPD is immediate from Proposition \ref{wwpd} since $g$ is
WPD for the action of $MCG(Y)$ on the curve complex of $Y$.
In fact, $g$ is $WWPD^+$ (see \cite{scl}).

\end{examples}

\subsection{Asymptotic dimension}

In this section we will show that if the collection of spaces $\cC(Y)$
has asymptotic dimension $\leq n$ uniformly, then $\asdim\cC(\bY)\leq
n+1$. 

Asymptotic dimension is invariant under quasi-isometries (or even a {\it
  coarse invariant}). In particular, asymptotic dimension of a
finitely generated group is well-defined. A general reference for
asymptotic dimension is \cite{bell-dranishnikov}; in connection to the
coarse setting see \cite{roe}; for the original definition and an
interesting discussion see Gromov's article \cite{gromov}.

We now review some basic facts.
We will need the following theorem.

\medskip
\noindent{\bf Bell-Dranishnikov's Hurewicz Theorem} \cite{hurewicz}. 
{\it Let
  $f:\cX\to \cY$ be a Lipschitz map with $\cX$ a geodesic space. Suppose
  that there exists $n$ such that for every $R$ the family $\{F_y=f^{-1}(B(y,R))\mid y\in \cY\}$ has
  $\asdim(F_y)\leq n$ uniformly. Then $\asdim(\cX)\leq \asdim(\cY)+n$.}

\medskip
This should be thought of as a generalization of the {\bf Product Formula}, $$\asdim(\cX\times \cY)\leq
\asdim(\cX)+\asdim(\cY).$$ For example, if $1\to A\to B\to C\to 1$ is a short exact sequence of
finitely generated groups then $\asdim(B)\leq
\asdim(A)+\asdim(C)$. Likewise, asymptotic dimension of the hyperbolic
plane is $\leq 2$ by considering the projection to a line whose fibers
are horocycles tangent to a fixed point at infinity (e.g. the
projection to the $y$-coordinate in the upper half-space model). More
generally one can apply this argument to a semi-simple Lie group and
its associated symmetric space (see \cite{bell-dranishnikov} for precise 
statements).

We will also use the following theorem.
  
\medskip
\noindent{\bf Union Theorem.} {\it Let $n \ge0$ be an integer, $\cX=\cup \cX_\alpha$ and assume
  that $\asdim(\cX_\alpha)\leq n$ uniformly. Also assume that for every
  $R>0$ there is a subset $\cY_R\subset X$ such that $\asdim(\cY_R)\leq n$
  and the sets $\cX_\alpha\setminus \cY_R$ and $\cX_\beta\setminus \cY_R$ are
  $R$-separated for $\alpha\neq\beta$ (i.e. $d(x,y)>R$ for any $x\in
  \cX_\alpha\setminus \cY_R$ and $y\in \cX_\beta\setminus \cY_R$). Then
  $\asdim(\cX)\leq n$. Furthermore the uniformity constants for
  $\asdim(\cX)$ only depend on the uniformity constants for $\cX_\alpha$
  and $\cY_R$.}

\begin{remark} 
The uniformity statement is not in \cite{bell-dranishnikov} but is
easily seen from the proof. 
\end{remark}

We noted above that asymptotic dimension is not only a quasi-isometric
invariant but is also a {\it coarse invariant},
in particular  $\asdim(\cX)\leq\asdim(\cY)$ if 
there exists a coarse embedding $f:\cX\to \cY$ (\cite{roe}).

 Using this fact will simplify our proof that the asymptotic dimension
 of the mapping class group is finite.

\subsection{$\cC(\bY)$ has finite asymptotic dimension}
We would like to show that $\cC(\bY)$ has finite asymptotic dimension
under the assumption that the asymptotic dimensions of the spaces
$\cC(Y)$ are uniformly bounded. To do so we will apply the
Bell-Dranishnikov Hurewicz Theorem to the map from $\cC(\bY)$ to
$\cS_K(\bY)$. The theorem is most natural to apply when the pre-images
of balls are Hausdorff neighborhoods of pre-images of
points. This is not the case in our situation and we need the following
technical lemma to deal with this issue.

\begin{lemma}\label{ball_pre-images}
Fix a vertex $Y$ in $\cS_K(\bY)$. Given $R>0$ and distinct vertices $X$ and $Z$ with $d(X,Y) = d(Z,Y) = m$ and $x \in \cC(X)$, $z \in \cC(Z)$ with $d_{\cC(\bY)}(x,z) < R$,
then  there exist a vertex $X_1$ with $d(X_1,Y) =m-1$ and 
$d_{\cC(\bY)}(x,X_1) < R + 2mL+\xi$.
\end{lemma}

\begin{proof}
By Lemma \ref{coarsedistanceestimate} we have $d^\pi_Y(x,z)\leq R$.  Since
$d(X,Y) = d(Z,Y) = m$ there is a path $X=X_0,X_1,\cdots, X_N=Z$ in
$\cS_K(\bY)$ of length $N\leq 2m$ with $d(X_1,Y) = m-1$. By Lemma \ref{coarsedistanceestimate}, for adjacent vertices in $\cS_K(\bY)$ we have $d^\pi_X(X_i,
X_{i+1}) < d_{\cC(\bY)}(X_i, X_{i+1}) =L$ so the triangle inequality implies that
$d^\pi_X(X_1,Z) < (2m-1)L$ and
\begin{eqnarray*}
d_X^\pi(x, X_1) & < &d^\pi_X(x, Z) + d_X^\pi(Z, X_1) \le
d^\pi_X(x,z) +\xi + d^\pi_X(Z,X_1)\\ &< & R + (2m-1)L + \xi.
\end{eqnarray*}
By the definition of $d^\pi_X(x,X_1)$ the distance in $\cC(X)$ from
$x$ to any point $\pi_X(X_1)$ is not more than
$d^\pi_X(x,X_1)$. Furthermore there is an edge in $\cC(\bY)$ from any
point in $\pi_X(X_1)$ to $\cC(X_1)$ of length $L$
and therefore the distance from $x$ to $\cC(X_1)$ in $\cC(\bY)$ is
less than $R+2mL + \xi$. The lemma is proved.
\end{proof}

\begin{thm}\label{SKY}
If the metric spaces $\cC(Y)$ for $Y\in\bY$ have asymptotic dimension
uniformly bounded by $n$ then $\cC(\bY)$ has asymptotic dimension
$\leq n+1$.
\end{thm}

\begin{proof}
Consider the projection map $p:\cC(\bY)\to \cS_K(\bY)$. The target is
a quasi-tree so its asymptotic dimension is $\leq 1$. We will verify the
conditions of Bell-Dranishnikov's Hurewicz Theorem for $p$. 
Let $B_m$ denote the ball of radius $m$ in $\cS_K(\bY)$ (centered at
some vertex). We will prove by induction on $m$ that
$\asdim(p^{-1}(B_m))\leq n$. Uniformity is not an issue since all of
our choices of constants will be independent of the vertex in
$\cS_K(\bY)$. When $m=0$ this is true by definition of $n$.

Now suppose $\asdim(p^{-1}(B_m))\leq n$ and we will argue
$\asdim(p^{-1}(B_{m+1}))\leq n$. To that end, we write
$$p^{-1}(B_{m+1})=\bigcup_{Y\in B_{m+1}}p^{-1}(Y)$$
and check that the hypotheses of the Union Theorem hold. Each
$p^{-1}(Y)$ has $\asdim\leq n$ by definition of $n$.

Let $R$ be given and set
$$\cY_R=N_{\tilde R}(p^{-1}(B_m))$$ the Hausdorff $\tilde
R$-neighborhood of $p^{-1}(B_m)$, where 
$$\tilde R=R+(2m+2)L +\xi.$$ 
By induction, $p^{-1}(B_m)$, and hence $\cY_R$, have $\asdim\leq n$. If
$X$ and $Z$ are distinct vertices at distance $m+1$ from the center of
$B_{m+1}$ then by Lemma \ref{ball_pre-images}, $p^{-1}(X)-\cY_R$ and
$p^{-1}(Z)-\cY_R$ are $R$-separated. It now follows from
Bell-Dranishnikov's Hurewicz Theorem that
$$\asdim(\cC(\bY))\leq n+1.$$
\end{proof}

\begin{question} Is $\asdim(\cC(\bY))\leq n$? 
\end{question}

\section{Mapping class group}\label{section.mcg}
We now apply our tools to the study of the mapping class group. In
this final section we will prove Theorems C, D, E, F and G mentioned
in the introduction.

\subsection{Curve complexes}
We will apply our previous work to a collection of curve graphs of a
subsurface of a fixed surface $\Sigma$, as in the work of Masur and
Minsky \cite{MM}, \cite{mm2}. We begin by recalling the
definition of the curve graph and projections. We
follow an approach that is not standard but is convenient.

Let $\Sigma$ be a compact orientable surface with boundary such that
$\chi(\Sigma)<0$,  possibly with finitely
many punctures (to be precise we mean compact after we fill in the punctures). 
Let $\cC_0(\Sigma)$ be the set of homotopy classes
of simple closed curves and properly embedded simple arcs that are not
peripheral or boundary compressible. We then define the {\em curve
  graph}, $\cC(\Sigma)$, to be the 1-complex obtained by attaching an
edge to disjoint closed curves or arcs in $\cC_0(\Sigma)$. We could
also attach higher dimensional simplices but the resulting complex is
quasi-isometric to its 1-skeleton so we stop at the curve graph.

\begin{remark} 
The graph we have constructed is often called the {\em curve and arc
  graph}, \cite{mm2}.
  The usual curve graph is quasi-isometric to the curve and
arc graph and so we will use the less cumbersome name of curve
graph. We also note that in the usual definition of the curve graph
there are exceptional cases, the punctured torus and the
sphere with 3 or 4 punctures, where the graph needs to be defined
differently. One advantage of the curve-arc graph is that one
definition works for all cases.

We also note that if $\Sigma$ is a 3-punctured sphere then
$\cC(\Sigma)$ is bounded and we could ignore such subsurfaces. However
there is also no harm in including them.
\end{remark}

We now define projections between curve graphs of essential
(i.e. connected, boundary components essential and nonperipheral)
subsurfaces of $\Sigma$. If $Y$ and $Z$ are essential subsurfaces, we
can only define the projection of
 $\cC(Z)$ to $\cC(Y)$ if $\partial Z$
intersects $Y$ essentially. We then define 
the {\it subsurface projection} $\pi_Y(Z)\subset \cC(Y)$
by taking the
intersection of $\partial Z$ with $Y$ and identifying homotopic curves
and arcs. If $z$ is vertex in $\cC(Z)$ then we define $\pi_Y(z) =
\pi_Y(Z)$.

We will also need the curve graph for a simple closed
curve. The definition here has a somewhat different flavor although
once we make the definition we can use it just as we do for the other
curve complexes. The simplest way to define the curve graph is to fix
a complete hyperbolic metric on the interior of $\Sigma$. If $\gamma$
is an essential non-peripheral simple closed curve let $X_\gamma$ be
the annular cover of $\Sigma$ to which $\gamma$ lifts. Let
$\cC_0(\gamma)$ be the set of complete geodesics in $X_\gamma$ that
cross the core curve and we form $\cC(\gamma)$ by attaching an edge to
vertices that represent disjoint geodesics. It is easy to check that
the distance in $\cC(\gamma)$ is the intersection number plus one and that
$\cC(\gamma)$ is quasi-isometric to ${\mathbb Z}$.

We now define projections to and from $\cC(\gamma)$. If $Y$ is an
essential subsurface such that $\partial Y$ intersects $\gamma$ let
$\pi_\gamma(Y)$ be those components of the pre-image of the geodesic
representatives of $\partial Y$ in $X_\gamma$ that intersect the core
curve. If $\beta$ is a simple closed curve that intersects $\gamma$ we
similarly define $\pi_\gamma(\beta)$ where we replace the $\partial Y$
with $\beta$. Finally if $\gamma$ intersects $Y$ essentially then
define $\pi_Y(\gamma)$ by restricting $\gamma$ to $Y$.

With these definitions in hand we will not
distinguish between essential subsurfaces and simple closed curves.

Since by definition $\pi_X(Y)$ is a collection of disjoint curves and arcs  we have 
$\diam \pi_X(Y) \le 1$, which verifies Axiom (P0).

The following lemma (without the explicit bound) was proved by
Behrstock \cite{jason} using the Masur-Minsky theory of hierarchies
\cite{mm2}. For a simple proof due to Leininger that produces the
explicit bound below see \cite{johanna,johanna2}.

We say that subsurfaces $X$ and $Y$ {\it overlap} if $\partial X\cap\partial
Y\neq\emptyset$ (this means that $\partial X$ and $\partial Y$ cannot
be made disjoint by a homotopy). Note that in that case $\pi_X(Y)$ and
$\pi_Y(X)$ are defined.
\begin{lemma}[Axiom P1]\label{triples}
Let $X$, $Y$ and $Z$ be overlapping subsurfaces. If
$$d^\pi_X(Y,Z)>10$$
then
$$d^\pi_Y(X,Z)<10.$$
\end{lemma}

We also have a finiteness statement for the number of large
projections between two overlapping subsurfaces. The statement we
require was proved in \cite{mm2} using their hierarchy technology. For
completeness we give a more direct proof here. While not necessary for
our applications we note that the proof below, unlike in \cite{mm2},
gives an explicit constant that is independent of the complexity of
the surface.

\begin{lemma}[Axiom P2]\label{finiteness}
Given subsurfaces $X$ and $Y$ there are
only finitely many subsurfaces $Z$ with $d_Z^\pi(X,Y)>3$.
\end{lemma}

\begin{proof}
More generally, we will prove that if $x,y$ are two arcs or curves
then there are only finitely many subsurfaces $Z$ with
$d_Z^\pi(x,y)>3$. The proof is in the spirit of Leininger's proof of
(P1). 

First assume that $x,y$ fill the surface. Suppose $Z$ is a subsurface
such that $\partial Z,x,y$ are all in minimal position and without
triple intersections. Further assume that some arc component of $x\cup
y-(x\cap y)$ intersects $\partial Z$ in at least 3 points. Then, as in
Leininger's argument, a component of $x\cap Z$ is disjoint from a
component of $y\cap Z$, so $d^\pi_Z(x,y)\leq 3$. In particular, the
condition $d^\pi_Z(x,y)> 3$ forces the intersection numbers
$i(x,\partial Z)$ and $i(y,\partial Z)$ to be bounded by twice the
number of components of $x\cup
y-(x\cap y)$, and there are only finitely many such subsurfaces $Z$.

For the general case, consider the smallest subsurface $\Sigma'$ that
contains $x\cup y$. Note that if $Z$ is a subsurface and
$Z\not\subset\Sigma'$, then there is a curve $w$ in $Z$ disjoint from
$x\cap Z$ and from $y\cap Z$, and this implies $d^\pi_Z(x,y)\leq
2$. If $Z\subset \Sigma'$ the proof concludes as in the filling case.
\end{proof}

Let $\bY$ be a collection of subsurfaces in $\Sigma$ that pairwise
overlap. Since $\{\mathcal C (Y)\}_{Y \in \bY}$ satisfies (P0)-(P2),
we obtain $\cC(\bY)$ by Theorem A.

In view of Theorem \ref{SKY}, we recall the following theorem
\cite{bell-fujiwara}.
\begin{thm}\label{bell-fujiwara}
Every curve graph has finite asymptotic dimension.
\end{thm}

It now yields:
\begin{thm}
Let $\bY$ be a collection of subsurfaces that pairwise overlap. Then
$\cC(\bY)$ has finite asymptotic dimension.
\end{thm}
This is because $\bY$ contains only finitely 
many subsurfaces up to homeomorphism, therefore
there is a uniform upper bound on their asymptotic dimension. 

We now say a word about the proof of Theorem \ref{bell-fujiwara} as
this is the only place were the dimension bound is not
computable. Gromov proved that $\delta$-hyperbolic groups have finite
asymptotic dimension. Here is a proof. Assume that $R\gg\delta$ is an
integer. For every vertex $v$ in the Cayley graph of the group at
distance $5kR$ from 1, $k=1,2,3,\cdots$, consider the
set $$U_v=\{x\in\Gamma\mid d(1,x)\in [5(k+1)R,5(k+2)R]\mbox{ and }
v\mbox{ lies on some geodesic }[1,x]\}$$ An easy thin triangle
argument shows that if $v,w$ are two vertices at distance $5kR$ from 1
such that both $U_v$ and $U_w$ intersect the same $R$-ball, then
$d(v,w)\leq 2\delta$. This gives a bound on the number of $U_v$'s that
can intersect the same $R$-ball, and this bound is independent of $R$;
thus $\asdim(\Gamma)<\infty$. We can also apply this argument to a
tree $T$ to show that $\asdim(T) \le 1$.

Bell-Fujiwara \cite{bell-fujiwara} modified this argument to show that
curve complexes have finite asymptotic dimension. They are hyperbolic
by the celebrated work of Masur-Minsky \cite{MM}, but not locally
finite, resulting in an infinite bound. The trick is to use {\it
  tight} geodesics in place of arbitrary geodesics. Finiteness
properties of tight geodesics proved by Bowditch \cite{bhb:tight}
imply that asymptotic dimension is finite. Note that Bowditch's
finiteness statement is proved via a geometric limit argument with
hyperbolic 3-manifolds and does not give a computable bound. It would
be interesting to give a new proof of Bowditch's result that gives a
computable bound. One could then obtain a computable bound for the
asymptotic dimension of the mapping class group.

\subsection{Partitioning subsurfaces into finitely many collections}

We would like to apply our construction of the projection complex to
subsurfaces and their associated curve complexes. To do so we need to
partition the set of all subsurfaces into finitely many collections
where any two subsurfaces in the same collection overlap.

\begin{lemma}\label{colors}
There is a coloring $\phi:\cC(\Sigma)^{(0)}\to F$ of the set of simple
closed curves on $\Sigma$ with a finite set $F$ of colors so that if
$a,b$ span an edge then $\phi(a)\neq\phi(b)$.
\end{lemma}

\begin{proof}
Let $T$ be the set of all connected double covers of $\Sigma$. If $a$
is a simple closed curve in $\Sigma$ define a function $f_a$ on the
set $T$ as follows. For a double cover $\tilde \Sigma\to\Sigma$
define $f_a(\tilde\Sigma)$ as 0 if $a$ does not lift to
$\tilde\Sigma$, and otherwise as the set $\{\alpha,\beta\}$ of
homology classes in $H_1(\tilde\Sigma;\Z_2)$ determined by the two
lifts of $a$. 

The set $F$ of colors is the set of all such functions
-- it is clearly finite.

We now show that if $a,b$ are disjoint nonparallel simple closed curves,
then $f_a\neq f_b$.

We will use the following construction of double covers. Let $\cC$ be
a nonseparating collection of disjoint simple closed curves and
properly embedded arcs in $\Sigma$. Then $\cC$ determines a double
cover $\tilde\Sigma\to\Sigma$ by cutting along $\cC$ and gluing
cross-wise two copies of the resulting surface (equivalently, the
associated index two subgroup is given by curves that intersect $\cC$
in an even number of points). In particular for any $a$ we can find a
cover $\tilde\Sigma\to\Sigma$ where $a$ lifts by applying the above
construction to a non-separating curve or properly embedded arc that
is disjoint from $a$. If $a$ represents a non-trivial homology class
and $b$ represents a differently class then $f_a(\tilde\Sigma) \neq
f_b(\tilde\Sigma)$. Therefore we can assume that $a$ and $b$ are
homologous.

For each component $S$ of $\Sigma\backslash (a\cup b)$ whose boundary
is contained in $a \cup b$ choose a simple curve $c$ such that
$S\backslash c$ is connected and let $\cC$ be the union of such
  curves. There is at least one and at most three such components so
  $\cC$ contains between one and three curves. Note that the curve $c$
  exists since $S$ will have one or two boundary components and can't
  be a disk or annulus. Therefore $S$ must have positive genus and
  hence contain a simple curve that doesn't separate $S$. Let
  $\tilde\Sigma$ be the double cover associated to $\cC$ by the
  construction above. If $f_a(\tilde\Sigma)= f_b(\tilde\Sigma)$ then
  there will be lifts $\tilde a$ and $\tilde b$ of $a$ and $b$ that
  bound a surface $\tilde S \subset \tilde\Sigma$ such that $\tilde S$
  doesn't contain either of the other lifts of $a$ and $b$. Then the
  restriction of the covering map to $\tilde{S}$ will be a
  homeomorphism and its image will contain a component of $\cC$. This
  is a contradiction so we must have $f_a(\tilde\Sigma)\neq
  f_b(\tilde\Sigma)$.
\end{proof}

\begin{lemma}[Color preserving subgroup]\label{finite index}
There is a finite index subgroup $G$ of the mapping class group
$MCG(\Sigma)$ (where $\Sigma$ is closed) such that every element of
$G$ preserves the colors from the proof of Lemma \ref{colors}.
\end{lemma}
We call this subgroup {\it the color preserving subgroup}.

\begin{proof}
The group $Aut(\pi_1(\Sigma))$ lifts to an action (up to homotopy) on
the union of connected double covers of $\Sigma$. Let $\Gamma$ be the
subgroup of $Aut(\pi_1(\Sigma))$ that fixes the $\Z_2$-homology of
this union. This will be a finite index subgroup of
$Aut(\pi_1(\Sigma))$ so its image $G$ in $Out(\pi_1\Sigma)\cong
MCG(\Sigma)$ will have finite index in $MCG(\Sigma)$ and will fix the
colors form Lemma \ref{colors}.
\end{proof}

\begin{prop}\label{collections}
Let $\Sigma$ be a compact surface with (possibly empty) boundary.  Let
$\bY$ be the collection of connected incompressible subsurfaces of
$\Sigma$ that are not the sphere with 3 boundary components. Then
$\bY$ can be written as a finite disjoint union
$$\bY^1\sqcup\bY^2\sqcup\cdots\sqcup\bY^k$$ so that
\begin{itemize}
\item the boundaries of any two surfaces in any $\bY^i$ intersect, and
\item there is a subgroup $G<MCG(\Sigma)$ of finite index that
  preserves each $\bY^i$: if $W\in \bY^i$ and $g\in G$ then $g(W)\in
  \bY^i$. 
\end{itemize}
\end{prop}

\begin{proof}
The mapping class group acts on $\bY$ and there are finitely many
orbits under the action. Let $G$ be the subgroup given by Lemma
\ref{finite index}. Since $G$ has finite index in $MCG(\Sigma)$, the
action of $G$ on $\bY$ also has finitely many orbits. These orbits are
our $\bY^i$ and by definition are invariant under the $G$-action.

We now show that if $W_0 \neq W_1$ are in $\bY^i$ then they have
intersecting boundary. There is a $g \in G$ such that $W_0 =
g(W_1)$. Since $g$ preserves the colors if the $W_0$ and $W_1$ don't
have intersecting boundary then $g$ must fix $\partial W_0 = \partial
W_1$ and $W_0$ must be the complement of $W_1$. By assumption the
$W_i$ are not spheres with three boundary components. They are also
not annuli for if they were then we would have $W_0 = W_1$. In
particular $W_0$ must contain a non-peripheral simple closed curve
$\gamma$. Since $g(\gamma)$ will be disjoint from $\gamma$ it will
have a different color. As $G$ fixes the colors this is a
contradiction.
\end{proof} 

Here is a perhaps unexpected application of our construction. This is an expansion of Theorem F in the introduction.

\begin{thm}\label{dehn hyperbolic}
\begin{enumerate}[(i)]
\item
Let $f$ be a Dehn twist in the curve $\gamma$ on $\Sigma$. There is a
finite index subgroup $\Gamma\subset MCG(\Sigma)$ and an action of
$\Gamma$ on a quasi-tree such that any power $f^k$ of $f$, $k\neq 0$,
that belongs to $\Gamma$ is a hyperbolic isometry.
\item
If $\Sigma$ has even genus $g$ and $\gamma$ separates into two subsurfaces
of genus $g/2$ then we may take $\Gamma=MCG(\Sigma)$.
\item In these actions, there is a bound to the diameter of the
  projection of a fixed quasi-axis of $f^k$ to any non-parallel
  translate. 
\end{enumerate}
\end{thm}

By
contrast, semisimple actions of mapping class groups on $CAT(0)$
spaces always have the property that Dehn twists are elliptic (see
\cite{bridson}).
From (i) it follows that a Dehn twist has linear growth 
in the word length of $\Gamma$, therefore in $MCG(\Sigma)$
(known by \cite{FLM}).

\begin{proof}
If $\Gamma$ is the subgroup of Proposition \ref{collections} or if
$\gamma$ is as in (ii) and $\Gamma=MCG(\Sigma)$ then the
$\Gamma$-orbit of $\gamma$ consists of pairwise intersecting
curves. Let $\bY$ be this orbit and consider the action of $G$ on the
quasi-tree of curve complexes $\cC(\bY)$. Since each curve complex
$\cC(g\gamma)$ is quasi-isometric to a line (and they are all
isometric to each other), it follows from Theorem \ref{bottleneck}
that $\cC(\bY)$ is a quasi-tree. Since a nontrivial power of $f$ acts
as a hyperbolic isometry on $\cC(\gamma)$ the claim follows. The
quasi-axis of $f^k$ is the curve complex $\cC(\gamma)$ and the
non-parallel translates are $\cC(g\gamma)$ where $g$ doesn't fix
$\gamma$. Since the projection of $\cC(g\gamma)$ to $\cC(\gamma)$ has
diameter 1, the last statement is a consequence of Theorem A.
\end{proof}

Here is another related application to the
Rips complex, $P_d(\cG)$, of a graph $\cG$.
 Rips has shown that if
$\cG$ is $\delta$-hyperbolic then for $d$ sufficiently large,
$P_d(\cG)$ is contractible \cite{gromov:hyp}. It has been hoped that with the same
assumptions, for $d$ sufficiently large $P_d(\cG)$ is $CAT(0)$. The
quasi-tree given by (ii) gives a counter-example to this conjecture,
at least for infinite valence graphs.

\begin{cor}\label{Rips counter}
There exist infinite diameter, infinite valence graphs that are
quasi-isometric to trees but whose Rips complex is never $CAT(0)$.
\end{cor}

\begin{proof}
Let $\cG$ be the quasi-tree given by (ii) of Theorem \ref{dehn
  hyperbolic}. Then $MCG(\Sigma)$ acts on $\cG$ with the Dehn twist
about the curve $\gamma$ acting hyperbolically.
 Then $MCG(\Sigma)$
will act on $P_d(\cG)$ for all $d$ and the Dehn twist will still act
hyperbolically. 
Moreover, since the action on $\cG$ is always semi-simple,
\cite{manning:qfa}, so is the action on $P_d(\cG)$.
Therefore, by Bridson's theorem \cite{bridson},
$P_d(\cG)$ is not $CAT(0)$.
\end{proof}

\subsection{Embedding $MCG$ into a finite product of $\cC(\bY)$'s}
\label{section.embedding}

Fix a set of finite generators for $MCG(\Sigma)$ and for all $g \in
MCG(\Sigma)$ let $|g|$ be the word length norm. We need the following
proposition. Recall that a finite collection of simple closed curves
is {\it binding} if every nonperipheral curve intersects at least one
curve in $\alpha$. If $W$ is any subsurface and $g\in MCG(\Sigma)$,
the restrictions $\pi_W(\alpha)$ and $\pi_W(g(\alpha))$ are nonempty and we
denote by $d^\pi_W(\alpha,g(\alpha))$ the diameter of their union in
the curve complex of $W$.

\begin{prop}\label{coarsegroupbound}
Let $\alpha$ be a finite binding collection of simple closed curves on
$\Sigma$. Given any $B>0$ there exists a $C>0$ such that if $|g|>C$
then there is a subsurface $W$ such that $d^\pi_W(\alpha, g(\alpha)) >
B$.
\end{prop}

\begin{proof} 
Fix a hyperbolic metric on $\Sigma$. When we discuss the Hausdorff
limit of a sequence of curves we assume that they have been realized
by hyperbolic geodesics in this metric.

Assume that the lemma is false. Then there exists a sequence of $g_i$
such that $|g_i| \rightarrow \infty$ but $d^\pi_W(\alpha,
g_i(\alpha))\leq B$ for all subsurfaces $W$. We pass to a subsequence
(which we don't relabel) such that $g_i(c)$ has a Hausdorff limit for
each curve $c$ in $\alpha$ (see e.g. \cite{casson-bleiler} for basic
facts about Hausdorff convergence in the lamination space). There are
then three possibilities:
\begin{itemize}
\item If the Hausdorff limits are all simple closed curves then the
  sequences $g_i(c)$ must become constant. However there are only
  finitely many elements of $MCG(\Sigma)$ that have the same image on a
  set of binding curves. This contradicts $|g_i| \rightarrow \infty$.

\item Fix a $c$ in $\alpha$ and let $\lambda$ be the Hausdorff limit
  of $g_i(c)$. Also assume that there is a minimal component
  $\lambda_Y$ of $\lambda$ that fills a non-annular subsurface
  $Y$. Let $c'$ be a curve in $\alpha$ that intersects $Y$. We will
  modify an argument of F. Luo (see \cite[Section 4.3]{MM}) to show that $d^\pi_Y(g_i(c), c')
  \rightarrow \infty$. If $d_{\cC(Y)}(\pi_Y(c'), \pi_Y(g_i(c)))$ is
  bounded we can pass to a subsequence where the distance is
  constant. For each $i$ let $x_i \in \cC(Y)$ be adjacent to
  $\pi_Y(g_i(c))$ but closer to $\pi_Y(c')$. We can pass to another
  subsequence such that $x_i$ converges in the Hausdorff topology to a
  lamination $\lambda'$. As the $x_i$ and $\pi_Y(g_i(c))$ are disjoint
  $\lambda'$ and $\lambda_Y$ can't intersect and since $\lambda_Y$
  fills $Y$ this implies that $\lambda' = \lambda_Y$, perhaps with
  some isolated leaves added. We can repeat this until we have a
  sequence in $\cC(Y)$ disjoint from $\pi_Y(c')$ that converges to the
  filling lamination $\lambda_Y$ (plus isolated leaves). This is a
  contradiction so we must have $d_Y(g_i(c), c') \rightarrow \infty$.

\item 
The final case is when the Hausdorff limit $\lambda$ isn't a
collection of simple curves but doesn't have a component that fills
a non-annular subsurface. In this case there must be a leaf of
$\lambda$ that spirals around a simple closed curve
$\beta$. Let $c'$ be a curve in $\alpha$ that
intersects $\beta$. Again fix a hyperbolic metric on $\Sigma$. We also
fix an annular neighborhood $X$ of $\beta$. 
Then $d^\pi_X(g_i(c), c')=i_X(g_i(c),c')$.
Since $\lambda$ spirals around $\beta$ we have $i_X(g_i(c),
c') \rightarrow \infty$ and therefore
$d^\pi_X(g_i(c), c') \rightarrow \infty$.
\end{itemize}
\end{proof}

Let $\Gamma$ be the subgroup of $MCG(\Sigma)$ from Proposition
\ref{collections} and let $$\bY^1,\cdots,\bY^k$$ be the orbits of
subsurfaces under $\Gamma$. Note that by construction one of the
collections consists of the single surface $\Sigma$. Let
$$\Pi = \cC(\bY^1)\times\cC(\bY^2)\times\cdots\times\cC(\bY^k)$$ be
the product of quasi-trees of curve complexes. Then $MCG(\Sigma)$ acts
on $\Pi$. For elements in $\Gamma$ the coordinates are fixed while
other elements will permute them.

Define $\Psi: MCG(\Sigma) \to 
\Pi$
by choosing a base vertex as the image of 1 and extending the map
equivariantly. Note that one of the factors in the target is just the
curve complex $\cC(\Sigma)$. We put the $l_1$-metric on the product
space $\Pi$. By construction $\Psi$ is Lipschitz. 

\begin{prop}\label{embedding}
$\Psi$ is a coarse embedding.
\end{prop}

\begin{proof} We will show that the restriction of $\Psi$ to $\Gamma$ is a coarse embedding. This will imply the proposition
since $\Gamma$ has finite index in $MCG$.
Note that if $\Psi$ is a coarse embedding or not
does not depend on the choice of the base point.

Say the basepoint has $\cC(\bY^i)$-coordinate equal to a curve
$\gamma_i$ in a surface $W_i$, and in the special factor $\cC(\Sigma)$
the coordinate is a curve $\gamma$. We may choose the binding set
$\alpha$ to contain $\gamma$, the $\gamma_i$ and the boundary
components of the $W_i$'s.

Note that for all subsurfaces $W$ the diameter of $\pi_W(\alpha)$ in
$\cC(W)$ is bounded by a fixed constant $D>0$. For example we could
choose $D$ to be one plus the number of intersection points.

Fix some $B>0$ and let $C$ be the constant given by Proposition
\ref{coarsegroupbound} with respect to $\alpha$ and $B+2D$. We'll show
that if $|g|>C$ then $d_{\Pi}(\Psi(id), \Psi(g))>B$ which implies that
$\Psi$ is a coarse embedding.

By Proposition \ref{coarsegroupbound} there exists a subsurface $W$
such that $d^\pi_W(\alpha, g(\alpha))>C$. The subsurface $W$ is in one
of the collections $\bY^i$. Since $\pi_W(\gamma_i)$ and
$\pi_W(g(\gamma_i))$ are contained in $\pi_W(\alpha)$ and
$\pi_W(g(\alpha))$ and the latter have diameter bounded by $D$ we have
$d^\pi_W(\gamma_i, g(\gamma_i)) \ge d(\alpha, g(\alpha))-2D$. By
Proposition \ref{coarsedistanceestimate} we then have
\begin{eqnarray*}
d_\Pi(\Psi(id), \Psi(g)) & \ge &d_{\cC(\bY^i)}(\gamma_i, g(\gamma_i))\\ &\ge& d^\pi_W(\gamma_i, g(\gamma_i))\\
& \ge & \pi_W(\alpha, g(\alpha)) - 2D\\
& \ge & B
\end{eqnarray*}
and the proposition is proved.
\end{proof}

It is also true that $\Psi$ is a quasi-isometric embedding. We will
not need this stronger result to prove Theorem D, but we include the proof since it may
be of independent interest.

\medskip
\noindent
{\bf Theorem C.} {\it $MCG(\Sigma)$ equivariantly quasi-isometrically
  embeds in a finite product of hyperbolic spaces.}
\medskip

\begin{proof}
The proof uses the remarkable Masur-Minsky formula \cite{mm2}, which
asserts that
$$|g|\simeq \sum_W\{\{d_W(\alpha,g(\alpha)\}\}_M$$
where $g\in MCG(\Sigma)$, $|g|$ is the word-norm of $g$ with respect
to any fixed finite generating set for $MCG(\Sigma)$, $\simeq$ is {\it coarse
equivalence}, i.e. each side is bounded by a linear function of the
other, $\{\{x\}\}_M=x$ if $x>M$ and otherwise it is 0, the sum
is taken over all subsurfaces of $\Sigma$, $\alpha$ is a fixed finite
binding set of curves in $\Sigma$, and $d_W(\alpha,g(\alpha))$ is the
distance in the curve complex of $W$ between the projections of a
curve in $\alpha$ and a curve in $g(\alpha)$ (we must choose a curve
that has a projection; choosing a different such curve changes the
distance by a bounded amount), and $M$ is a sufficiently large
constant. By enlarging $M$ or $K'$ from Theorem \ref{masur-minsky1} we
may assume that $M=K'$. The two estimates combine to give that
$|g| \leq Ad(\Psi(1),\Psi(g))+B$ for universal constants $A,B$. The
reverse bound follows from the fact that $\Psi$ is Lipschitz.
\end{proof}

\medskip
\noindent
{\bf Theorem D.} {\it Let $\Sigma$ be a compact orientable surface
  with (possibly empty) boundary. Then $\asdim(MCG(\Sigma))<\infty$.
}
\medskip

\begin{proof}
If $\chi(\Sigma)>0$ then $MCG(\Sigma)$ is finite and
$\asdim(MCG(\Sigma))=0$.  If the $\Sigma$ is a torus,
$MCG(\Sigma)$ is virtually free and hence
$\asdim(MCG(\Sigma))=1$. Assume $\chi(\Sigma)<0$.
By the Product Formula and Theorem \ref{SKY} it follows that
$\asdim(\Pi)<\infty$. 
Note that the Product formula applies to the $\ell_1$-product.
Proposition \ref{embedding} then implies that
$\asdim(MCG(\Sigma))<\infty$.
\end{proof}

Let $\Sigma$ be a possibly punctured closed surface and $\mathcal
T(\Sigma)$ its Teichm\"uller space equipped with the Teichm\"uller
metric. 

\begin{thm} \label{teichmuller}
$\asdim (MCG(\Sigma)) \le \asdim(\mathcal T(\Sigma))<\infty$.
\end{thm}

Since $MCG(\Sigma)$ acts on $\mathcal T(\Sigma)$
properly discontinuously, an orbit map $MCG(\Sigma) \to \mathcal T(\Sigma)$ is a coarse embedding. Thus we have 
$\asdim (MCG(\Sigma)) \le \asdim(\mathcal T(\Sigma))$.
The proof of the second inequality 
will use the following facts. When $\gamma$ is a curve in
$\Sigma$ and $\epsilon>0$ denote by $Thin_\epsilon(\Sigma,\gamma)$ the
subset of $\mathcal T(\Sigma)$ where $\gamma$ has hyperbolic length
$<\epsilon$. 

\begin{enumerate}[(A)]
\item {\bf Minsky's Product Theorem.} \cite{minsky-product}
If $\epsilon$ is small enough,
  the subspace $Thin_\epsilon(\Sigma,\gamma)$ is quasi-isometric to the product 
  $\mathcal T(\Sigma/\gamma) \times Z$ where $Z$ is a horoball in
  hyperbolic plane and $\Sigma/\gamma$ denotes the surface obtained
  from $\Sigma$ by cutting open along $\gamma$ and crushing the
  boundary components to punctures (if $\gamma$ is separating this
  Teichm\"uller space is the product of Teichm\"uller spaces of the
  components).
\item For every $R>0$ there is $\epsilon_0>0$ such that whenever
  $\gamma$ and $\gamma'$ intersect then
  $Thin_{\epsilon_0}(\Sigma,\gamma)$ and
  $Thin_{\epsilon_0}(\Sigma,\gamma')$ are $R$-separated.
\end{enumerate}

Statement (B) follows easily from Kerckhoff's Theorem \cite{kerckhoff}, or
indeed from (A).

\begin{proof}[Proof of Theorem \ref{teichmuller}]
The proof is by induction on the complexity of the surface, which is
the dimension of $\mathcal T(\Sigma)$. Induction starts with the case
of 2-dimensional Teichm\"uller space (hyperbolic plane) when
asymptotic dimension is 2.

For the inductive step, note that (A) and the Product Formula for
asymptotic dimension
immediately imply that thin parts have finite asymptotic
dimension. Write the collection of all curves on $\Sigma$ as a finite
disjoint union $C_1\sqcup C_2\sqcup \cdots \sqcup C_k$ so that curves
in the same collection intersect. It was shown that this is possible
for closed $\Sigma$ in Lemma \ref{colors}, but the punctured case
follows quickly from the closed case (e.g. blow up the punctures to
boundary components and double).

Consider the subsets $$Thick=X_0\subset X_1\subset
X_2\subset\cdots\subset X_k=\mathcal T(\Sigma)$$
where $X_i$ is the subset of $\mathcal T(\Sigma)$ consisting of
hyperbolic surfaces with the property that if $\gamma$ is a curve with
length $<\epsilon$ then $\gamma\in C_1\cup\cdots\cup C_i$, 
and Thick consists of hyperbolic surfaces with no essential 
curves of length $< \epsilon$.
Let $N$ be chosen so that $\asdim(MCG(\Sigma))\leq N$ and so that 
$\asdim(Thin_\epsilon(\Sigma,\gamma))\leq N$ for every curve
$\gamma$. We will argue by induction on $i$ that $\asdim(X_i)\leq N$.

When $i=0$ this follows from the fact that $X_0$ (the thick part) is
quasi-isometric to $MCG(\Sigma)$. Suppose $\asdim(X_{i})\leq N$.

Now write $$X_{i+1}=X_i\cup\bigcup_{\gamma\in C_{i+1}}Y^i_\gamma$$
where $Y^i_\gamma$ is the set of hyperbolic structures in
$Thin_\epsilon(\Sigma,\gamma)$ where every
curve shorter than $\epsilon$ is either equal to $\gamma$ or belongs
to $C_1\cup\cdots\cup C_i$. We will check the conditions of the Union
Theorem. 

Let $R>0$ be given, let $\epsilon_0$ be as in (B) (we may assume 
that $\epsilon_0 < \epsilon$).
Define 
$$Y_R=X_i\cup \bigcup_{\gamma\in C_{i+1}}Z^i_\gamma$$
where $Z^i_\gamma$ is the set of hyperbolic structures where $\gamma$
has length in the interval $[\epsilon_0,\epsilon)$ and any curve of
  length $<\epsilon$ is either $\gamma$ or belongs to
  $C_1\cup\cdots\cup C_i$. By (B) the sets $Y^i_\gamma\setminus Y_R$
  are $R$-separated and 
each set is  contained in $Thin_\epsilon(\Sigma,\gamma)$ and
the latter sets have $\asdim\leq N$ uniformly, since there are only
finitely many isometry types of such sets.
Therefore we only need to argue that $\asdim(Y_R)\leq
  N$. But $Y_R$ is contained in a Hausdorff neighborhood of
  $X_i$, as follows easily from Minsky's Product Theorem.
That $\asdim(X_i)\leq N$ is
the inductive hypothesis.
\end{proof}

A variation of the argument also shows that Teichm\"uller space
equipped with Weil-Petersson metric has finite asymptotic
dimension. Denote this space by $\mathcal T_{WP}(\Sigma)$. Let
$\mathcal P(\Sigma)$ be the {\it pants complex} for $\Sigma$, where a
vertex is represented by a pants decomposition of $\Sigma$ and an edge
corresponds to a pair of pants decompositions that differ in only one
curve in each, and the two curves intersect minimally (one or two
points, depending on whether their removal produces a complementary
component which is a punctured torus or a 4-punctured sphere). There
is a natural coarse map $\Upsilon:\mathcal P(\Sigma)\to \mathcal
T_{WP}(\Sigma)$ that sends a pants decomposition to the (bounded) set
consisting of hyperbolic metrics where the curves in the decomposition
have length bounded by a Bers constant. Brock \cite{brock1,brock2}
proved that $\Upsilon$ is an equivariant quasi-isometry.

\begin{thm}\label{weil}
$\asdim(\mathcal T_{WP}(\Sigma))=\asdim(\mathcal P(\Sigma))<\infty$.
\end{thm}

\begin{proof}
Consider an orbit map $MCG(\Sigma)\to \mathcal P(\Sigma)$ and
define a (pseudo) metric on $MCG(\Sigma)$ by restricting the one from $\mathcal P(\Sigma)$ (some pairs of points may have distance 0). Since the
action of $MCG(\Sigma)$ on the pants complex has finitely many orbits
of simplices, $MCG(\Sigma)$ with this metric, $d$, is quasi-isometric to the
pants complex.
There is a
Masur-Minsky estimate for the distance between 1 and $g\in
MCG(\Sigma)$ (see the discussion in
\cite[Section 8]{mm2}):
$$d(1,g)\simeq \sum_W\{\{d_W(\alpha,g(\alpha)\}\}_M$$
where $W$ runs over subsurfaces which are {\it not} annuli. We have an
action of $MCG(\Sigma)$ on $$\Pi =
\cC(\bY^1)\times\cC(\bY^2)\times\cdots\times\cC(\bY^k)$$ as before,
where we delete all annuli from the $\bY^i$'s. The orbit map is a
quasi-isometric embedding (with respect to the new metric on
$MCG(\Sigma)$) by exactly the same argument as before. The theorem
follows.
\end{proof}

\bibliography{./ref2}

\end{document}